\documentclass[singlespace,notugly]{uofcthesis}
\pdfoutput=1

\usepackage{amsmath,amssymb,amsfonts,amsthm}
\usepackage[all,arc]{xy}
\usepackage{enumerate}
\usepackage{mathrsfs}
\usepackage{aliascnt}
\usepackage[unicode=true, pdfusetitle=true, draft=false, bookmarks=true,bookmarksnumbered=false,bookmarksopen=false, breaklinks=true,pdfborder={0 0 0},backref=false,colorlinks=false, plainpages=false,pdfpagelabels,linktocpage=false]{hyperref}
\usepackage{wasysym}
\usepackage{calligra}
\DeclareMathAlphabet{\lscr}{T1}{calligra}{m}{n}
\DeclareFontShape{T1}{calligra}{m}{n}{<->s*[1.1]callig15}{}

\DeclareFontFamily{OMS}{rsfs}{\skewchar\font'60}
\DeclareFontShape{OMS}{rsfs}{m}{n}{<-5>rsfs5 <5-7>rsfs7 <7->rsfs10 }{}
\DeclareSymbolFont{rsfs}{OMS}{rsfs}{m}{n}
\DeclareSymbolFontAlphabet{\scr}{rsfs}

\newcommand{\sA}{\scr{A}}
\newcommand{\sB}{\scr{B}}
\newcommand{\sC}{\scr{C}}

\newcommand{\sG}{\scr{G}}
\newcommand{\sH}{\scr{H}}
\newcommand{\sI}{\scr{I}}

\newcommand{\sL}{\scr{L}}
\newcommand{\sM}{\scr{M}}
\newcommand{\sN}{\scr{N}}
\newcommand{\sO}{\scr{O}}
\newcommand{\sP}{\scr{P}}

\newcommand{\sR}{\scr{R}}

\newcommand{\sU}{\scr{U}}

\newcommand{\bC}{\mathbb{C}}

\newcommand{\bF}{\mathbb{F}}

\newcommand{\bQ}{\mathbb{Q}}
\newcommand{\bR}{\mathbb{R}}

\newcommand{\bT}{\mathbb{T}}

\newcommand{\bZ}{\mathbb{Z}}

\def\quickop#1{\expandafter\newcommand\csname #1\endcsname{\operatorname{#1}}}
\quickop{Hom} \quickop{End} \quickop{Aut} \quickop{Tel} \quickop{Mic} 
\quickop{Ext} \quickop{Tor} \quickop{Id} \quickop{Coker} \quickop{Ker}
\quickop{Lim} \quickop{Colim} \quickop{Holim} \quickop{Hocolim}
\quickop{id} \quickop{tel} \quickop{mic} \quickop{coker} 
\quickop{colim} \quickop{holim} \quickop{hocolim} \quickop{im}

\newcommand{\lofspace}{\hspace{0ex}}

\DeclareMathAlphabet{\mathpzc}{OT1}{pzc}{m}{it}
\DeclareMathAlphabet{\mathfrc}{T1}{frc}{m}{n}
\DeclareMathAlphabet{\mathitbf}{OML}{cmm}{b}{it}
\usepackage{bbm} 

\quickop{trace}

\newcommand{\boxp}{\mathbin{\square}}
\newcommand{\oboxp}{\mathbin{\overline{\square}}}

\newcommand{\defeq}{\mathrel{\mathop:}=}

\newcommand{\ground}{\mathbbm{k}}
\newcommand{\rarticle}{a } 
\newcommand{\rmodule}{a $\ground$-module}
\newcommand{\rmod}{\ground\textrm{-mod}}

\newcommand{\spaces}{\textrm{Top}}
\newcommand{\gtop}{\G\textrm{-Top}}

\newcommand{\hogtop}{\textrm{ho-}\G\textrm{-Top}_*}
\newcommand{\mackey}{\lscr{Mf}\phantom{x}}
\newcommand{\htmf}{\underline{\pi}^\G}
\newcommand{\mapmf}[2]{\underline{[{#1},{#2}]}_G}
\newcommand{\op}{^{\text{op}}}
\newcommand{\fund}[1]{\Pi_\G{#1}}
\newcommand{\fundx}[1]{\Pi}
\newcommand{\fundhom}[2]{\fundx{}\!\left( {#1},{#2} \right)}
\newcommand{\chr}{\text{Ch}_\ground}
\newcommand{\fmap}[1]{(\alpha_{#1},[\gamma_{#1}])}
\newcommand{\cycp}{{C_p}}

\newcommand{\fp}{\bF_p}
\newcommand{\fq}{\bF_q}
\newcommand{\ftwo}{\bF_2}
\newcommand{\cyctwo}{\bZ/2}
\newcommand{\cpv}[1]{\bC P({#1})}
\newcommand{\ab}{\textbf{Ab}}
\newcommand{\bZm}{\bZ_{-}}

\newcommand{\repnoneq}{{H\ground}}
\newcommand{\repeq}{{H\! A}}
\newcommand{\paper}{thesis}

\newcommand{\myspan}[3]{\xymatrix@R=0.5pc@C=1pc{ & {\gset{#2}} \ar[dr] \ar[dl] & \\ {\gset{#1}} & & {\gset{#3}}\\}}
\newcommand{\inlinespan}[3]{{\gset{#1}}\longleftarrow{\gset{#2}}\longrightarrow{\gset{#3}}}
\newcommand{\contraspan}[3]{\xymatrix@R=0.5pc@C=1pc{ & {#1} \ar@{=}[dr] \ar[dl]_{#3} & \\ {#2} & & {#1}\\}}
\newcommand{\cospan}[3]{\xymatrix@R=0.5pc@C=1pc{ & {#1} \ar@{=}[dl] \ar[dr]^{#3} & \\ {#1} & & {#2}\\}}
\newcommand{\idspan}[1]{\xymatrix@R=0.5pc@C=1pc{ & {#1} \ar@{=}[dl] \ar@{=}[dr] & \\ {#1} & & {#1}\\}}
\newcommand{\minicontraspan}[3]{\xymatrix@R=0.2pc@C=.4pc{ & {\scriptstyle #1} \ar@{=}[dr] \ar[dl]_{\scriptstyle #3} & \\ {\scriptstyle #2} & & {\scriptstyle #1}\\}}
\newcommand{\minicospan}[3]{\xymatrix@R=0.2pc@C=.4pc{ & {\scriptstyle #1} \ar@{=}[dl] \ar[dr]^{\scriptstyle #3} & \\ {\scriptstyle #1} & & {\scriptstyle #2}\\}}
\newcommand{\miniidspan}[1]{\xymatrix@R=0.2pc@C=.4pc{ & {\scriptstyle #1} \ar@{=}[dl] \ar@{=}[dr] & \\ {\scriptstyle #1} & & {\scriptstyle #1}\\}}

\newcommand{\mf}[5]{\xymatrix{ {#1} \ar@/_2ex/[d]_{#4} \\ {#2} \ar@(d,dr)[]_{#5} \ar@/_2ex/[u]_{#3} \\ }}
\newcommand{\mymatrix}[1]{\left( \begin{smallmatrix}#1 \end{smallmatrix} \right)}
\newcommand{\tw}[2]{\makebox[0pt][r]{\phantom{x}}_{#2}{#1}}
\newcommand{\jg}[1]{\langle {#1} \rangle}
\newcommand{\mft}{{t_{\!\rho\,}}}
\newcommand{\mfr}{{r_{\!\rho\,}}}
\renewcommand{\r}{R} 
\renewcommand{\l}{L}
\newcommand{\rminus}{R_{-}}
\newcommand{\lminus}{L_{-}}

\newcommand{\trivo}{\bullet}
\newcommand{\freeo}{{\raisebox{-0.25ex}{\sun}}} 

\renewcommand{\G}{G}
\newcommand{\eqchains}[1]{C_*^\G({#1})}
\newcommand{\sh}{h_\G} 
\newcommand{\ulh}{\underline{H}_\G} 
\newcommand{\h}{{H}_\G} 
\newcommand{\hl}{{H}^\G} 
\newcommand{\ulhl}{\underline{H}^\G} 
\newcommand{\neqh}{H} 

\newcommand{\hr}{\makebox[0in][l]{$\widetilde{{H}}$}\phantom{H}_{\G}} 
\newcommand{\hlr}{\makebox[0in][l]{$\widetilde{{H}}$}\phantom{H}^{\G}} 
\newcommand{\ulhr}{\makebox[0in][l]{$\widetilde{\underline{H}}$}\phantom{H}_{\G}} 
\newcommand{\ulhrnoph}{\widetilde{\underline{H}}_{\G}} 
\newcommand{\ulhrcapt}{\widetilde{H}_{\G}} 
\newcommand{\ulhlr}{\makebox[0in][l]{$\widetilde{\underline{H}}$}\phantom{H}^{\G}} 
\newcommand{\coh}{\mathcal{H}}
\newcommand{\neqhr}{\widetilde{H}} 

\newcommand{\cohomplot}[9]{%
  \xymatrix@R=0.5pc@C=#9{
  &  &  &  &  &  & \makebox[0pt][c]{${#2}$} &  &  &  &  &  \\
{\phantom{\makebox[0pt][c]{${#8}$}}} &   & \makebox[0pt][c]{${#7}$} &   & \makebox[0pt][c]{${#7}$} &   & \makebox[0pt][c]{${#6}$} &   & . &   & . &   \\
   & . &   & . &   & . &   & . &   & . &   & . \\
{\phantom{\makebox[0pt][c]{${#8}$}}} &   & \makebox[0pt][c]{${#7}$} &   & \makebox[0pt][c]{${#7}$} &   & \makebox[0pt][c]{${#6}$} &   & . &   & . &   \\
   & . &   & . &   & . &   & . &   & . &   & . \\
    \ar@{<.>}[rrrrrrrrrrrr] &    & \makebox[0pt][c]{$#4$} &   & \makebox[0pt][c]{$#4$} &   & \makebox[0pt][c]{${#3}$} &   & \makebox[0pt][c]{$#5$} &   & \makebox[0pt][c]{$#5$} &   &  & \makebox[0pt][c]{${#1}$} \\
{\phantom{\makebox[0pt][c]{${#8}$}}} & . &   & . &   & . &   & . &   & \makebox[0pt][c]{${#7}$} &   & \makebox[0pt][c]{${#7}$} \\
{\phantom{\makebox[0pt][c]{${#8}$}}} &   & . &   & . &   & \makebox[0pt][c]{${#6}$} &   & . &   & . &   \\
{\phantom{\makebox[0pt][c]{${#8}$}}} & . &   & . &   & . &   & . &   & \makebox[0pt][c]{${#7}$} &   & \makebox[0pt][c]{${#7}$} \\
{\phantom{\makebox[0pt][c]{${#8}$}}} &   & . &   & . &   & \makebox[0pt][c]{${#6}$} &   & . &   & . &   \\
    &   &   &   &   &   & \ar@{<.>}[uuuuuuuuuu]  &   &   &   &   & \\
  } %
}

\newcommand{\shiftedpic}[8]{%
  \xymatrix@R=0.5pc@C=0.7pc{
  &  &  &  &  &  & \makebox[0pt][c]{${#2}$} &  &  &  &  &  \\
{\phantom{\makebox[0pt][c]{${#8}$}}} &   & \makebox[0pt][c]{${#7}$} &   & \makebox[0pt][c]{${#7}$} &   & \makebox[0pt][c]{${#6}$} &   & . &   & . &   \\
   & . &   & . &   & . &   & . &   & . &   & . \\
 &    & \makebox[0pt][c]{$#4$} &   & \makebox[0pt][c]{$#4$} &   & \makebox[0pt][c]{${#3}$} &   & \makebox[0pt][c]{$#5$} &   & \makebox[0pt][c]{$#5$} &   &  & \\
{\phantom{\makebox[0pt][c]{${#8}$}}} & . &   & . &   & . &   & . &   & \makebox[0pt][c]{${#7}$} &   & \makebox[0pt][c]{${#7}$} \\
{\phantom{\makebox[0pt][c]{${#8}$}}} &   & . &   & . &   & \makebox[0pt][c]{${#6}$} &   & . &   & . &   \\
{\phantom{\makebox[0pt][c]{${#8}$}}} & . &   & . &   & . &   & . &   & \makebox[0pt][c]{${#7}$} &   & \makebox[0pt][c]{${#7}$} \\
   \ar@{<.>}[rrrrrrrrrrrr] {\phantom{\makebox[0pt][c]{${#8}$}}} &   & . &   & . &   & \makebox[0pt][c]{${#6}$} &   & . &   & . &  & & \makebox[0pt][c]{${#1}$} \\
{\phantom{\makebox[0pt][c]{${#8}$}}} & . &   & . &   & . &   & . &   & \makebox[0pt][c]{${#7}$} &   & \makebox[0pt][c]{${#7}$} \\
{\phantom{\makebox[0pt][c]{${#8}$}}} &   & . &   & . &   & \makebox[0pt][c]{${#6}$} &   & . &   & . &   \\
    &   &   &   &   &   & \ar@{<.>}[uuuuuuuuuu]  &   &   &   &   & \\
  } %
}

\newcommand{\kernelpic}[8]{%
  \xymatrix@R=0.5pc@C=0.7pc{
  &  &  &  &  &  & \makebox[0pt][c]{${#2}$} &  &  &  &  &  \\
{\phantom{\makebox[0pt][c]{${#8}$}}} &   & . &   & . &   & . &   & . &   & . &   \\
   & . &   & . &   & . &   & . &   & . &   & . \\
 &    & \makebox[0pt][c]{$#4$} &   & \makebox[0pt][c]{$#4$} &   & \makebox[0pt][c]{${#3}$} &   & \makebox[0pt][c]{$#5$} &   & \makebox[0pt][c]{$#5$} &   &  & \\
{\phantom{\makebox[0pt][c]{${#8}$}}} & . &   & . &   & . &   & . &   & \makebox[0pt][c]{${#7}$} &   & \makebox[0pt][c]{${#7}$} \\
{\phantom{\makebox[0pt][c]{${#8}$}}} &   & . &   & . &   & . &   & . &   & . &   \\
{\phantom{\makebox[0pt][c]{${#8}$}}} & . &   & . &   & . &   & . &   & \makebox[0pt][c]{${#7}$} &   & \makebox[0pt][c]{${#7}$} \\
   \ar@{<.>}[rrrrrrrrrrrr] {\phantom{\makebox[0pt][c]{${#8}$}}} &   & . &   & . &   & . &   & . &   & . &  & & \makebox[0pt][c]{${#1}$} \\
{\phantom{\makebox[0pt][c]{${#8}$}}} & . &   & . &   & . &   & . &   & . &   & . \\
{\phantom{\makebox[0pt][c]{${#8}$}}} &   & . &   & . &   & . &   & . &   & . &   \\
    &   &   &   &   &   & \ar@{<.>}[uuuuuuuuuu]  &   &   &   &   & \\
  } %
}

\newcommand{\mydim}[1]{(|{#1}^\G|,|{#1}|)}
\newcommand{\signrep}{\zeta}
\newcommand{\rotrep}{\lambda}
\newcommand{\cxrep}{\phi}

\newcommand{\covproj}[1]{P_{\!#1}} 
\newcommand{\contraproj}[1]{P^{\!#1}} 
\newcommand{\covinj}[1]{I_{\!#1}} 
\newcommand{\contrainj}[1]{I^{\!#1}} 

\newcommand{\gset}[1]{\mathbf{#1}}
\newcommand{\sm}{\wedge}
\newcommand{\defword}[1]{\textbf{#1}}
\newcommand{\cxuniverse}{{\sU_{\bC}}}
\newcommand{\runiverse}{\sU_{\bR}}
\newcommand{\triv}{\{e\}}
\newcommand{\Ccell}{C^{\text{cell}}}
\newcommand{\Ccelleq}{\underline{C}^{\text{cell}}}
\newcommand{\floor}[1]{\lfloor {#1}\rfloor}

\newcommand{\pt}{pt}

\DeclareMathOperator{\proj}{proj}

\renewcommand{\to}{\longrightarrow}

\newcommand{\eqnref}[1]{(\ref{#1})}

\newtheorem{thm}{Theorem}[chapter]

\newaliascnt{cor}{thm} 
\newtheorem{cor}[cor]{Corollary}
\aliascntresetthe{cor}

\newaliascnt{prop}{thm}
\newtheorem{prop}[prop]{Proposition}
\aliascntresetthe{prop}

\newaliascnt{lem}{thm}
\newtheorem{lem}[lem]{Lemma}
\aliascntresetthe{lem}

\newaliascnt{conj}{thm}
\newtheorem{conj}[conj]{Conjecture}
\aliascntresetthe{conj}

\newaliascnt{fact}{thm}
\newtheorem{fact}[fact]{Fact}
\aliascntresetthe{fact}

\newaliascnt{axiom}{thm}
\newtheorem{axiom}[axiom]{Axiom}
\aliascntresetthe{axiom}

\newaliascnt{axioms}{thm}
\newtheorem{axioms}[axioms]{Axioms}
\aliascntresetthe{axioms}

\theoremstyle{definition}

\newaliascnt{defn}{thm}
\newtheorem{defn}[defn]{Definition}
\aliascntresetthe{defn}

\newaliascnt{exmp}{thm}
\newtheorem{exmp}[exmp]{Example} 
\aliascntresetthe{exmp}

\theoremstyle{remark}

\newaliascnt{rem}{thm}
\newtheorem{rem}[rem]{Remark}
\aliascntresetthe{rem}

\makeatletter
\let\c@equation\c@thm
\makeatother
\numberwithin{equation}{chapter}

\ifdim\overfullrule>0pt 

    \usepackage[firstpage]{draftwatermark}
    \SetWatermarkScale{4.0}
    \SetWatermarkLightness{0.8}
    \SetWatermarkAngle{50}
    \SetWatermarkText{DRAFT-\linebreak\today}
    
    \usepackage[notref,notcite]{showkeys}
\fi

\usepackage{ucs} 
\usepackage[utf8x]{inputenc}  
\PrerenderUnicode{∞}
\PrerenderUnicode{≠}

\makeatletter \def\HyPsd@CatcodeWarning#1{}\makeatother

\begin{document}

\title{Equivariant Local Coefficients and the $RO(G)$-graded cohomology of
classifying spaces}
\author{Megan Elizabeth Shulman}
\date{June 2010}
\department{Mathematics}
\division{Physical Sciences}
\degree{Doctor of Philosophy}
\maketitle

\dedication
\begin{center}
        To Mom and Dad
\end{center}



\tableofcontents




\listoffigures





\topmatter{Abstract}
This \paper{} consists of two main parts.  In the second part,
we recall how a description of local coefficients that Eilenberg introduced
in the 1940s leads to spectral sequences for the computation of homology
and cohomology with local coefficients.  We then show how to construct new
equivariant analogues of these spectral sequences for $RO(G)$-graded Bredon
homology and cohomology.  Finally, we use these spectral sequences to
complete a sample calculation, in which we use the equivariant Serre
spectral sequence and the equivariant cohomology of complex projective
spaces to compute the cohomology of the equivariant classifying space
$B_{C_p} O(2)$.

However, to complete this sample computation, we need to know the
cohomology of the complex projective space $\cpv{\cxuniverse} = B_{C_p}
SO(2)$.  This calculation was done in \cite{lgl}, but relies on a theorem
whose proof as given was incorrect.  We spend the first part of this \paper{}
providing a correct proof and summarizing the results of \cite{lgl}.

\topmatter{Acknowledgements} 
I would like to thank my advisor, Peter May, for many years' worth of 
discussions about all things equivariant; John Greenlees, for several very
helpful discussions about the multiplicative structure on cohomology; and
Vigleik Angeltveit and Anna Marie Bohmann for proofreading drafts of this
work.  I would also like to thank Michael Shulman for many helpful
conversations about category theory and Day tensor products, and Gaunce
Lewis for writing so much about Mackey functors.

\mainmatter



\chapter{Introduction}

\section{Overview}

From a theoretical point of view, there are many reasons to think that
$RO(G)$-graded Bredon cohomology, and in particular Mackey functor valued
$RO(G)$-graded Bredon cohomology, is the correct equivariant analogue of
ordinary cohomology.  Recent results, such as \cite{kervaire} and
\cite{AGH}, provide additional reasons to want to be able to compute using
this theory.  Unfortunately, computations in $RO(G)$-graded Bredon
cohomology are extremely difficult, and even the few nontrivial
calculations in the literature are not well known.

There are two immediate problems when approaching computations in
$RO(G)$-graded Mackey functor valued Bredon cohomology.  First, the reservoir
of known low-level computations on which to build further results is
extremely small; as an example, \cite{kervaire} had to compute the
cohomology of a point for $G=C_{2^n}$, the cyclic group of order $2^n$,
because this was not previously known.  Second, even if there were such a
reservoir of computations, many familiar nonequivariant tools do not yet
have tractable equivariant analogues.  This \paper{} represents a first
step towards remedying both hurdles. 

We begin in \autoref{ch::background} by recalling some of the basic
definitions, which are unfamiliar to many algebraic topologists, and
discussing the computation by Stong and Lewis of the cohomology of a point
when $G=C_p$, a cyclic group of prime order.  We then summarize the results
of L.\ Gaunce Lewis in \cite{lgl} concerning the equivariant cohomology of
complex projective spaces.  The central result of \cite{lgl} rests on a
lemma whose proof as given is incorrect; a correction is included in
\autoref{ch::freeness}.

After establishing these basic calculations, we turn in \autoref{ch::ss} to
more sophisticated tools for computing $RO(G)$-graded cohomology.  Moerdijk
and Svensson in \cite{MS} developed a Serre spectral sequence in
integer-graded Bredon cohomology, and Kronholm extended this to a spectral
sequence for $RO(G)$-graded Bredon cohomology in \cite{bill}.  However,
this Serre spectral sequence has not yet been used for any computations, in
large part because it is impossible to ignore local coefficients when
working equivariantly with Bredon (co)homology:  in contrast to the
nonequivariant situation, almost no interesting examples reduce to trivial
local coefficients.  It is thus necessary to develop some tools for working
with homology and cohomology with local coefficients.

To this end, we recall a simple universal coefficient spectral sequence
which appears in \cite[p.\ 355]{CE}, but which, to the best of our
knowledge, has not previously been applied in conjunction with the Serre
spectral sequence.  Part of the reason is that the definition of local
coefficients that appears in the construction of the Serre spectral
sequence is not tautologically the same as the definition that gives the
cited universal coefficient spectral sequence. 

The connection comes from an old result of Eilenberg \cite{Eil},
popularized by Whitehead \cite[VI.3.4]{GW} and, more recently, by Hatcher
\cite[App3.H]{Hat}.  It identifies the local coefficients that appear in
the context of fibrations with a more elementary definition in terms of the
chains of the universal cover of the base space.  The identification makes
working with local coefficients much more feasible.  We shall first say
in \autoref{sec::noneq} how this goes nonequivariantly, and illustrate with
an example.  We then explain the equivariant generalization in
\autoref{sec::eq}.  The later parts of this \paper{} will review some necessary
background on equivariant classifying spaces and finally go through a
sample calculation in \autoref{ch::computations}.  This calculation
computes characteristic classes in $RO(G)$-graded Mackey functor valued
cohomology of equivariant real 2-plane bundles when $G$ is cyclic of prime
order.

Finally, we conclude by discussing the multiplicative structure on this
spectral sequence in \autoref{ch::mult}, and go through some nonequivariant
sample calculations in \autoref{sec::ex}.


\section{Other aspects of equivariant cohomology}

We have mentioned the lack of known computations in $RO(G)$-graded Bredon
cohomology, but perhaps a few words about the reasons are in order.

To begin, although it is clear that $G$-spectra and thus equivariant
cohomology theories should be graded on a richer structure than the
integers, it is possible that the real $G$-representations are not quite
the ``true'' indexing set.  As we discuss later in this \paper{}, the
Mackey functor valued Bredon cohomology $\ulh^*(X)$ of a $G$-space $X$
should be considered free if it is a direct sum of suspended copies
$\Sigma^\alpha \ulh^*(\pt)$ of the cohomology of a point.  However, as
remarked in \cite{FL}, there are nontrivially distinct sets $\{\alpha_i\}$
and $\{\beta_i\}$ of virtual representations such that $\bigoplus_i
\Sigma^{\alpha_i} \ulh^*(\pt)$ and $\bigoplus_i \Sigma^{\beta_i}
\ulh^*(\pt)$ are isomorphic free modules.  This makes it harder to
determine the ``correct'' $RO(G)$ dimensions of the generators of a free
$\ulh^*(\pt)$-module.

Additionally, there are several possible definitions of a $G$-CW-complex.
If we allow the spheres and disks of these complexes to have nontrivial
actions of $G$---by allowing them to be the unit spheres and disks of
representations of $G$, for example---then there is no equivariant
Whitehead theorem: a $G$-map between $G$-CW-complexes need not be
$G$-homotopic to a cellular map.  Thus, if we define our cohomology
theories using an appropriate equivariant analogue of cellular cohomology,
a non-cellular $G$-map need not induce a map on cellular cohomology.  In
particular, given two different $G$-CW-structures on a $G$-space $X$, the
identity map of $X$ may not induce a self-map of the corresponding cellular
cohomologies.

Some of the numerous other subtleties along these lines will become evident
during the course of this \paper{}.

\vspace{2ex}

Throughout this \paper{}, $\ground$ will refer to a fixed ground ring.
$G$ will be a finite group throughout, although a few results may apply
more generally to discrete groups or compact Lie groups.  Reduced
cohomology theories are denoted with $\widetilde{H}$, unreduced theories
with $H$.


\chapter[Background]{Background on Mackey functors and Bredon cohomology}\label{ch::background}
\section{Bredon cohomology and extension to $RO(G)$-grading}\label{sec::bredon}

A nonequivariant (reduced) cohomology theory $\neqhr^*(-)$ is defined to be a
sequence of contravariant functors from the homotopy category of based
spaces  to $\rmod$, the category of $\ground$-modules, together with suspension
natural isomorphisms ${\Sigma: \neqhr^n\to \neqhr^{n+1}}$,
satisfying the Eilenberg-Steenrod axioms.  If it satisfies the dimension
axiom as well, we say that the theory is ordinary.  If $X$ is an unbased
space, then its unreduced cohomology is given by $\neqh^*(X) =
\neqhr^*(X_+)$, the reduced homology of $X$ with a disjoint
basepoint adjoined.

Similarly, for a fixed group of equivariance $G$, an \defword{equivariant
(reduced) cohomology theory} indexed on the integers is a sequence of
contravariant functors from the $G$-homotopy category of based $G$-spaces
to $\rmod$, satisfying the obvious generalizations of the
Eilenberg-Steenrod axioms: weak equivalence, exactness, additivity, and
suspension.  Note that a basepoint of a $G$-space $X$ is required to be in
the $G$-fixed points $X^G$.  Also, a $G$-homotopy of $G$-maps $X\to Y$
consists of a $G$-map $X\times I\to Y$, where the interval $I$ is given the
trivial $G$ action and the product has the diagonal $G$ action.  
An equivariant cohomology theory is
\defword{ordinary} if it satisfies the dimension axiom: $\hr^n({G/K}_+) = 0$
for all nonzero integers $n\in\bZ$, $n\ne 0$, and all subgroups $K < G$.


There are some complications, of course.  The above axioms assume that
the collection of functors making up a cohomology theory, as in the
nonequivariant situation, is indexed by $\bZ$.  However, equivariant
theories are more naturally given by a collection of functors indexed on
the free abelian group generated by the irreducible representations of $G$.
For historical reasons, such a theory is called \defword{$RO(G)$-graded},
and we will continue this convention.  However, keep in mind that, with the
usual definition, the underlying abelian group of $RO(G)$ consists of
\emph{equivalence classes} of formal sums of representations.  We cannot
pay attention only to equivalence classes: the only way to make signs work
out is to remember the isomorphisms between isomorphic representations.
So, for the purposes of this \paper{}, we will define $RO(G)$ to be the free
abelian group on irreducible representations of $G$.  \label{rem::grading} 
Concretely, this means that an $RO(G)$-graded theory consists of a functor
$\hr^{\sum_i a_i \rotrep_i}\colon\hogtop\to\rmod$ for each formal sum
$\sum_i a_i\rotrep_i$ of $G$-representations, where $\{\rotrep_i\}_i$ runs
over the irreducible representations of $G$ and each $a_i$ is a possibly
negative integer.  There are some important bookkeeping details which
ensure that this is well-defined.  We will largely ignore those issues in
this \paper{}; see \cite{AK} or \cite{parametrized} for details.

At any rate, after replacing $n$ by $\alpha\in RO(G)$, the functors of an
equivariant cohomology theory satisfy the weak equivalence, exactness, and
additivity axioms, together with the following modified suspension axiom.

\begin{axiom}\label{ax::susp}
\defword{Equivariant suspension}: For each $\alpha \in RO(G)$ and actual
representation $V$, there is a natural isomorphism
\[ \Sigma^V \colon  \hr^\alpha(X) \to \hr^{\alpha+V}(\Sigma^V X) =
\hr^{\alpha+V}(S^V \sm X),\] 
where $S^V$ is the one-point compactification of the representation $V$.
The suspension isomorphisms for different representations are compatible;
for representations $V$ and $W$, $\Sigma^{V+W} \cong \Sigma^V \Sigma^W$.
\end{axiom}

Somewhat surprisingly, an $RO(G)$-graded theory is ordinary if it satisfies
the dimension axiom for the integer-graded part of the theory: that is,
$\hr^n({G/K}_+) = 0$ for nonzero \emph{integers}, but the dimension axiom
says nothing about $\hr^\alpha({G/K}_+)$ for $\alpha\notin\bZ$.  These
cohomology groups in ``off-integer'' dimensions are determined by the other
axioms, and are generally nonzero.  As a result, the $RO(G)$-graded
cohomology of a point is quite complicated.  This provides the first of
many hurdles to computing with equivariant cohomology.

The ordinary equivariant cohomology theories are usually called
\defword{Bredon cohomology theories}; ordinary theories graded on the
integers were introduced by Bredon in \cite{bredon}.  There are
corresponding notions of ordinary homology theories, satisfying the
expected modifications of the above axioms.

The integer-graded Bredon cohomology theories have a very concrete
description.  The Bredon cohomology $\hr^*(X)$ can be thought of as compiling
data about the nonequivariant cohomology rings $\neqhr^*(X^K)$ of the $K$-fixed
points, for each subgroup $K$ of $G$.  They do this using the language of
coefficient systems.

\begin{defn}
Fix a (finite) group of equivariance $G$.  The \defword{orbit category}
$\sO_G$ has objects the orbits $G/K$, for $K$ a subgroup of $G$; morphisms
are $G$-maps of $G$-sets.
\end{defn}

There is a map $G/K \to G/J$ if and only if $K$ is subconjugate to $J$ in
$G$.  If this is the case, then a $G$-map $\alpha\colon G/K\to G/J$ is
determined by what it does to the identity coset $eK$; $\alpha(eK) =
g_\alpha J$, for some $g_\alpha$ such that $g_\alpha^{-1} K g_\alpha
\subset J$.  It follows that the automorphism group of $G/K$ is the
\defword{Weyl group} ${W_G K = ({N_G K})/K}$, a quotient of the
normalizer of $K$.

\begin{defn}
A \defword{coefficient system} is a contravariant functor
$\sO_G\op\to\rmod$.
\end{defn}

As the name suggests, the coefficients in an integer-graded ordinary
equivariant cohomology theory are coefficient systems.  One way to define
Bredon cohomology on $G$-CW-complexes is as follows.

\begin{defn}\label{def::trivialCW}
A \defword{$G$-CW-complex with trivial cells} $X$ is a colimit of
$n$-skeleta $X_n$, formed as follows.  $X_0$ is a discrete $G$-set;
$X_{n+1}$ is formed from $X_n$ by attaching cells of the form $G/K \times
D^{n+1}$ along boundary $G$-maps $G/K \times S^n \to X_n$.  Here $D^{n+1}$
and $S^n$ have the trivial $G$ action, and $K$ denotes a subgroup of $G$;
different cells may have different $K$.
\end{defn}

A $G$-CW-complex with trivial cells $X$ may be thought of as an ordinary
CW-complex equipped with a cellular $G$ action $G\times X\to X$ such that,
for each $n$-cell $D^n$ and $g\in G$, the action map $\varphi_g\colon
x\mapsto gx$ either fixes $D^n$ pointwise or gives a homeomorphism from
$D^n$ to a distinct second $n$-cell.

We will later generalize this notion to $G$-CW-complexes where the cells
are allowed to have nontrivial $G$ actions.

\begin{exmp}
Let $G = {C_2}$, the cyclic group with two elements, acting on $X = S^2$ by
rotation by $180^\circ$.  We can give this a $G$-CW structure as follows.
\begin{itemize}
\item $X_0 = {C_2}/{C_2} \times S^0$, i.e.\ $S^0$ with the trivial $G$ action.
\item $X_1$ is given by attaching the single cell ${C_2}/\triv \times D^1$ along
the map \[{C_2}/\triv \times S^0 \to S^0 \colon (g,x)\mapsto x.\]  $X_1$ is thus
$G$-homeomorphic to the unit circle in the complex plane $\bC$, with
${C_2}$ acting by complex conjugation.  
\item Finally, $X = X_2$ is given by attaching
the single two-cell ${C_2}/\triv \times D^2$ along the obvious map
${C_2}/\triv \times S^1 \to X_1$.
\end{itemize}
\end{exmp}

Now recall that nonequivariant cellular cohomology of a CW-complex $X$
(which agrees with singular cohomology) is defined as the cohomology of the
\defword{cellular chain complex}.  This has
\[ \Ccell_n(X) = \neqh_n(X_n,X_{n-1};\ground).\]
We can use this to define a cellular chains coefficient system in the
equivariant situation.

\begin{defn}\label{def::cellchains}
The $n^\text{th}$ \defword{cellular chains coefficient system} for a
$G$-space $X$ is the functor $\Ccelleq_n\colon \sO_G\op\to\rmod$
given on objects by
\[ G/K \mapsto \Ccell_n(X^K).\]
A map $\alpha\colon G/{K_1} \to G/{K_2}$ given by $\alpha\colon  eK_1
\mapsto g_\alpha K_2$ defines a map
\[ X^{K_2}\to X^{K_1}\colon  x\mapsto g_\alpha x.\]
We use this to define the required map $\Ccell_n(X^{K_2}) \to
\Ccell_n(X^{K_1})$ on morphisms.
\end{defn}

Given a coefficient system $M$, we can then look at the set of natural
transformations $\Hom_{\sO_G}(\Ccelleq_n(X),M)$ for each $n$.  This has the
structure of \rmodule{}.  Further, $\Ccell_*(X^K)$ is a chain
complex.  This gives $\Ccelleq_*(X)$ the structure of a chain complex in
the category of coefficient systems, and makes $\Hom(\Ccelleq_*(X),M)$ a
(co)chain complex in the category of $\ground$-modules.

\begin{defn}
Given a group $G$, a $G$-CW-complex $X$, and a coefficient system $M\colon
\sO_G\op\to\rmod$, the unreduced \defword{Bredon cohomology} $\h^*(X;M)$ of
$X$ with coefficients in $M$ is the cohomology of the chain complex of
$\ground$-modules
\[ \Hom_{\sO_G}(\Ccelleq_*(X),M).\]
As usual, the reduced Bredon cohomology of a based space $X$ with basepoint
$x_0$ is the cohomology of the pair $\h^*(X,x_0;M)$.
\end{defn}

\noindent
It can be checked that each theory $\hr^*(-;M)$ is ordinary: for each $K <
G$,
\[ \hr^n({G/K}_+;M) = 
\begin{cases}
0 & n\in\bZ,\ n\ne 0 \\
M(G/K) & n = 0.
\end{cases}\]

\begin{rem}
The notation $\Hom_{\sO_G}$ is used because natural transformations are the
``hom of functors.''  There is also a notion of a tensor product of
functors $\otimes_{\sO_G}$, defined using a coend; the Bredon
\emph{homology} with coefficients in some covariant coefficient system
$N\colon \sO_G\to\rmod$ is the homology of the chain complex with
$n^\text{th}$ term
\[ \Ccelleq_n(X)\otimes_{\sO_G} N = \int^{G/K\in \sO_G}
\Ccelleq_n(X)(G/K)\otimes N(G/K);\]
see \cite{AK}.
\end{rem}

However, as mentioned above, equivariant cohomology is more naturally
graded on $RO(G)$\footnote{Recall our nonstandard definition of $RO(G)$ as
the free abelian group generated by the irreducible representations of
$G$.}.  This is necessary in order to have Poincar\'{e} duality, for
example.  We will also see later in this \paper{} that the cohomology of
complex projective spaces and Grassmannians is free (in the appropriate
sense) as a module over the cohomology of a point for $RO(G)$-graded
cohomology, but not for integer-graded cohomology.

The following theorem describes when the integer-graded theory can be
extended to an $RO(G)$-graded theory.  Mackey functors will be defined in
\autoref{sec::mackey}.

\begin{thm}[Lewis, May, McClure]\label{thm::rog}
For a group $G$ and $M\colon \sO_G\op\to\rmod$,  the ordinary
integer-graded cohomology theory $\hr^*(-;M)$ extends to an $RO(G)$-grading
if and only if $M$ is the underlying contravariant coefficient system of a
Mackey functor.  Likewise, if $N\colon \sO_G\to\rmod$ is a covariant
coefficient system, Bredon homology $\hl_*(-;N)$ with coefficients in $N$
extends to an $RO(G)$ grading if and only if $N$ is the underlying
coefficient system of a Mackey functor.
\end{thm}

From now on, we will turn our attention from coefficient systems to Mackey
functors.

%

\section{A review of Mackey functors}\label{sec::mackey}

Recall that $G$ always denotes a finite group.


\begin{defn}\label{def::burnsidestable}
The \defword{Burnside category} $\sB_G$ is the full subcategory of the
equivariant stable category on objects $\Sigma^\infty (\gset{b}_+)$, where
$\gset{b}$ is a finite $G$-set.  Explicitly, the objects of $\sB_G$ are
finite $G$-sets and the morphisms are the stable $G$-maps.
\end{defn}

\begin{defn}\label{def::mf}
A \defword{Mackey functor} is a contravariant additive functor
from the Burnside category $\sB_G$ to the category $\rmod$.
\end{defn}

For obvious reasons, $\sB_G$ is sometimes called the ``stable orbit
category'' and Mackey functors ``stable coefficient systems.''  The
expansion from orbits to finite $G$-sets, which are disjoint unions of
orbits, is largely for later convenience; since Mackey functors are
additive functors, they are determined by their values on the orbits.

Although \autoref{def::burnsidestable} makes the connection between $\sO_G$
and $\sB_G$ clear, it is generally easier to work with the following
combinatorial definition.  It is shown in \cite{AK} that the two
definitions are equivalent.

\begin{defn}\label{def::preburnside}
The category $\sB_G^+$ is the category having:
\begin{itemize}
\item objects: the finite $G$-sets
\item morphisms: equivalence classes of spans \[ \myspan{b}uc \] 
with composition given by pullbacks.
\end{itemize}
\end{defn}

Two spans $\inlinespan{b}uc$ and $\inlinespan{b}vc$ are equivalent if there
is a commutative diagram as follows:
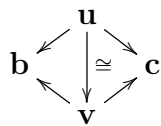
\[ \xymatrix@R=0.5pc@C=1pc{
  & \gset{u} \ar[dr] \ar[dl] \ar[dd]^{\cong} & \\
  \gset{b} & & \gset{c} \\
  & \gset{v} \ar[ur] \ar[ul] & \\
}\]
Each hom set $\sB_G^+(\gset{b},\gset{c})$ has the structure of an abelian
monoid.  We may take the sum of $\inlinespan{b}uc$ and $\inlinespan{b}vc$
to be the span
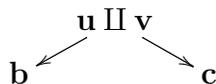
\[ \myspan{b}{u\amalg v}c \]
We may thus apply the Grothendieck construction to the hom sets of
$\sB_G^+$ to get an abelian group.

For convenience in the remainder of this \paper{}, we will additionally enrich
$\sB_G$ over $\rmod$, rather than abelian groups.  Doing so does not change the
Mackey functors ${\sB_G\op\to\rmod}$, but it will make the notation for
represented functors less cumbersome.

\begin{defn}\label{def::burnside}
The \defword{Burnside category} $\sB_G$ is the category enriched over
$\rmod$ having:
\begin{itemize}
\item objects: the objects of $\sB_G^+$
\item morphisms: $\sB_G(\gset{b},\gset{c})$ is the tensor product (over
$\bZ$) of the ground ring $\ground$ with the Grothendieck group of
$\sB_G^+(\gset{b},\gset{c})$.
\end{itemize}
\end{defn}

Mackey functors form a category $\mackey$, with morphisms given by natural
transformations of functors $\sB_G\op\to\rmod$.  Explicitly, this means
that a map of Mackey functors $M\to N$ consists of compatible maps
$M(\gset{b})\to N(\gset{b})$ for every finite $G$-set $\gset{b}$.
Since Mackey functors are additive functors, a Mackey functor $M$ over a
commutative ring $\ground$ is determined by its values on the orbits $G/K$, and
likewise a map $M\to N$ of Mackey functors is determined by the maps
$M(G/K)\to N(G/K)$.

The Burnside category $\sB_\G$ has several nice properties which will be
useful when simplifying the box product formula in \autoref{subs::boxp}.
Both are self-evident as long as \autoref{def::burnside} is used.

\begin{lem}\label{lem::selfdual}
$\sB_\G$ is self-dual, via the functor $D$ which is the identity on objects
and takes a span to its mirror image.  \hfill$\Box$
\end{lem}

Additionally, the monoidal product $\times$ makes $\sB_\G$ a closed symmetric
monoidal category---and the internal hom is \emph{also} $\times$.  Note,
however, that although $\times$ denotes the Cartesian product in $\sB_G$, it is
not the categorical product.

\begin{lem}\label{lem::closed}
Let $D$ be the self-duality functor of \autoref{lem::selfdual}.  Then, for
any $\gset b\in\sB_\G$, the functor $-\times\gset b$ is left adjoint to
$\gset b\times -$:
\[
\sB_\G(\gset a\times \gset b,\gset c)
\cong
\sB_\G(\gset a,\gset b\times \gset c)
,\]
naturally in $\gset a$ and $\gset c$.  It is also natural in $\gset b$
if we correct variance issues by replacing one $\gset b$ with $D\gset
b$.\hfill$\Box$
\end{lem}

\subsection{Connection to coefficient systems}
The (unstable) orbit category embeds contravariantly and covariantly in
$\sB_G$.  Both embeddings are the identity on objects; a map of $G$-sets
$G/L\xrightarrow{\alpha} G/K$ is sent to either 
\[ \contraspan{G/L}{G/K}{\alpha} \hspace{10pt}\text{ or }\hspace{10pt} \cospan{G/L}{G/K}{\alpha} \]
as appropriate.  Furthermore, every morphism in $\sB_G$ can be written as
a composite of such morphisms:
\[ \xymatrix@R=0.5pc@C=1pc{
  & \gset{u} \ar[dl] \ar[dr] & \\
  \gset{b} & & \gset{c}\\}
 \hspace{10pt}=\hspace{10pt}
 \xymatrix@R=0.5pc@C=1pc{
  & \gset{u} \ar[dl] \ar@{=}[dr] & & \gset{u} \ar@{=}[dl] \ar[dr] & \\
  \gset{b} & & \gset{u} & & \gset{c}\\
}\]
Furthermore, a general span breaks up as the disjoint union of spans
involving orbits; disjoint union is both the categorical product and the
categorical coproduct on $\sB_G$.  Hence a Mackey functor $M$ determines
and is determined by a pair of contravariant and covariant coefficient
systems which agree on objects and which satisfy certain compatibility
diagrams encoding the composition in $\sB_G$.  

These embeddings $\sO_G\to\sB_G$ come up frequently, and so we fix some
notation.

\begin{defn}\label{def::restransfer}
Let $f:\gset b\to \gset c$ be a map of finite $G$-sets.  Then the two spans
\[ \cospan{\gset b}{\gset c}{f} \hspace{20pt}\text{and}\hspace{20pt}
\contraspan{\gset b}{\gset c}{f} \]
are called the \defword{restriction} and \defword{transfer}, respectively,
of $f$.  If a Mackey functor $M$ is implicit, denote its value on the
restriction by $r_f\colon M(\gset c)\to M(\gset b)$ and its value on the
transfer by $t_f\colon M(\gset b)\to M(\gset c)$.
\end{defn}

The notation comes from the language of representation theory, one of the
first places where Mackey functors were studied.  Alternatively, if we view
$\sB_\G$ as a subcategory of the stable category using
\autoref{def::burnsidestable}, restrictions and transfers correspond
exactly to the maps of those names on representation spheres.

We are now in a position to understand the statement of \autoref{thm::rog},
which says that equivariant homology and cohomology theories can be
extended to $RO(G)$-gradings if and only if they have coefficient systems
which extend to Mackey functors.

The subsequent subsections will address the issue of which Mackey functor
ought to be taken as the ``universal'' coefficient Mackey functor for
cohomology.


%

\subsection{Box products and maps out of box products}\label{subs::boxp}
The category $\mackey$ of Mackey functors is symmetric monoidal, with
tensor product $\boxp$ given by the Day tensor product\footnote{In the same
way, if graded $\ground$-modules are thought of as functors from the discrete
category $\bZ$ to $\rmod$, the usual notion of a tensor product of graded
$\ground$-modules is exactly the Day product on the functor category
$[\bZ,\rmod]$.}.  Explicitly, $\boxp$ is a left Kan extension.  Given
Mackey functors $M,N\colon \sB_G\op\to \rmod$, we can form the external
product 
\[ 
\begin{array}{l}
M\oboxp N \colon \sB_G\op\times\sB_G\op\to \rmod\\[3pt]
\phantom{M\oboxp N \colon} 
\makebox[49pt][c]{$(\gset{b},\gset{c})$}\longmapsto M(\gset{b})\otimes N(\gset{c}),
\end{array}
\]
and $M\boxp N$ is the left Kan extension of $M\oboxp  N$ along the
Cartesian product functor
$\times\colon \sB_G\times\sB_G \to \sB_G$:
\[ \xymatrix@C=3pc{
  \sB_G\times\sB_G \ar^{M\oboxp N}[r] \ar_{\times}[d] & \rmod \\
  \sB_G \ar@{.>}_{M\boxp N}[ur] \\
}\]
In other words, natural transformations from $M\boxp N$ to another Mackey
functor $P$ are the same as natural transformations from $M\oboxp N$ to
$P\circ\times$:
\[ [\sB_G\op,\rmod](M\boxp N,P) \cong
[\sB_G\op\times\sB_G\op,\rmod](M\oboxp N,P\circ \times).\]
It follows that a map $M\boxp N\to P$ is specified by giving maps
\[M(\gset b)\otimes N(\gset c)\to P(\gset b\times\gset c),\]
 for all $\gset b,\gset c\in\sB_\G$, natural with respect to $\gset b$ and
$\gset c$.  In fact, we can get away with less.\footnote{The fact that
Dress pairings are the same as maps out of box products is stated in
\cite{purple} and \cite{lgl}, but without any indication of a proof.  The
proof is formal, of course, but not entirely obvious.}

\begin{lem}\label{lem::dresspairing}
A map $\theta\colon M\boxp N\to P$ determines and is determined by a
collection of maps
\[\theta_{\gset b}\colon M(\gset b)\otimes N(\gset b) \to P(\gset b)\]
for $\gset b\in\sB_G$, such that the following diagrams commute for each map
of $G$-sets $f\colon\gset b\to \gset c$.\footnote{Recall the notation for
restrictions and transfers from \autoref{def::restransfer}.}
\[
\xymatrix{
  M(\gset c) \otimes N(\gset c) \ar[r]^-{\theta_{\gset c}} \ar[d]_{r_f\otimes r_f} & P(\gset c) \ar[d]^{r_f} \\
  M(\gset b) \otimes N(\gset b) \ar[r]^-{\theta_{\gset b}} & P(\gset b) \\
}
\hspace{40pt}
\xymatrix@R=1pc@C=1pc{
  & M(\gset b) \otimes N(\gset b) \ar[rr]^-{\theta_{\gset b}} & & P(\gset b) \ar[dd]^{t_f} \\
  \makebox[20pt][c]{$M(\gset b) \otimes N(\gset c)$} \ar[dr]_(0.35){t_f\otimes\id} \ar[ur]^(0.35){\id\otimes r_f} & & & \\
  & M(\gset c) \otimes N(\gset c) \ar[rr]^-{\theta_{\gset c}} & & P(\gset c) \\
}
\]
\[
\xymatrix@R=1pc@C=1pc{
  & M(\gset b) \otimes N(\gset b) \ar[rr]^-{\theta_{\gset b}} & & P(\gset b) \ar[dd]^{t_f} \\
  \makebox[20pt][c]{$M(\gset c) \otimes N(\gset b)$} \ar[dr]_(0.35){\id\otimes t_f} \ar[ur]^(0.35){r_f\otimes\id} & & & \\
  & M(\gset c) \otimes N(\gset c) \ar[rr]^-{\theta_{\gset c}} & & P(\gset c) \\
} \]
\end{lem}

\begin{defn}
The data in \autoref{lem::dresspairing} is called a \defword{Dress pairing},
after \cite{dress}.  When the commuting diagrams above are expressed in
terms of equations and elements, they are called the \defword{Frobenius
relations}.
\end{defn}

In particular, notice that the maps $\theta_{\gset b}$ are natural with
respect to restrictions, but not generally with respect to transfers!
However, when $f$ is an isomorphism, it is straightforward to check that
that the transfer associated to $f$ is equivalent, as a span in $\sB_G$, to
the restriction associated to $f^{-1}$.  It then follows that the
appropriate map $\theta_{\gset b}$ is natural with respect to $t_f$, and
the commuting pentagons become redundant.

\begin{proof}
To go from a map $\theta\colon M\boxp N\to P$ to a Dress pairing, consider
the maps
\[ \theta_{\gset b}\colon M(\gset b)\otimes N(\gset b)\to P(\gset b\times
\gset b) \xrightarrow{r_\Delta} P(\gset b),\]
where the first map comes from the Kan adjunction and $r_\Delta$ is the
restriction associated to the diagonal map of $G$-sets $\gset b\to \gset
b\times \gset b$.  We then get the commuting diagrams above by applying $P$
to commuting diagrams in $\sB_G$.

For the other direction, given the maps $\theta_{\gset b}$ subject to the
three commuting diagrams of the lemma, we obtain the maps making up
$\theta: M\boxp N\to P$ via the composites
\[ M(\gset b)\otimes N(\gset c) \xrightarrow{r_{\rho_1} \otimes r_{\rho_2}}
M(\gset b\times \gset c)\otimes N(\gset b\times \gset c)
\xrightarrow{\theta_{\gset b\times \gset c}} P(\gset b\times \gset c),\]
where $r_{\rho_1}$ and $r_{\rho_2}$ are the restrictions associated with
the projections of $\gset b\times \gset c$ onto its first and second factors.

We need to show that these maps are natural with respect to maps in
$\sB_G\times\sB_G$.  Since every span can be written as the composite of
a restriction and a transfer, it suffices to check naturality with respect
to restrictions and transfers.  Naturality with respect to restrictions
follows immediately from the commuting square in the lemma.  For
transfers, given maps $f\colon \gset b\to \gset b'$ and $g\colon \gset c\to
\gset c'$, the map ${t_f\otimes t_g\colon M(\gset b)\otimes N(\gset c) \to
M(\gset b')\otimes N(\gset c')}$ factors as
\[ M(\gset b)\otimes N(\gset c) \xrightarrow{t_f\otimes\id} M(\gset
b')\otimes N(\gset c) \xrightarrow{\id\otimes t_g} M(\gset b')\otimes
N(\gset c'),\]
so it suffices to consider naturality in the case $g=\id\colon \gset c\to
\gset c$.  The desired result follows from considering the diagram
\[ \xymatrix@R=2pc@C=1pc{
  M(\gset b)\otimes N(\gset c) \ar[rr]^-{r_{\rho_1}\otimes r_{\rho_2}}
  \ar[dr]^(0.6){r_{\rho_1}\otimes r_{\rho_2}} \ar@/_1pc/[dd]_{t_f\otimes\id} 
  && M(\gset b\times \gset c)\otimes N(\gset b\times \gset c) \ar[rr]^-{\theta_{\gset b\times \gset c}} &&
  P(\gset b\times \gset c) \ar[dd]^-{t_{f\times\id}}\\
  & \makebox[20pt][c]{$M(\gset b\times \gset c)\otimes N(\gset b'\times \gset c)$}
  \ar[ur]_(0.6){\ \id\otimes r_{f\times\id}} \ar[dr]^(0.6){t_{f\times\id}\otimes\id}
  &&& \\
  M(\gset b')\otimes N(\gset c) \ar[rr]^-{r_{\rho_1}\otimes r_{\rho_2}} && M(\gset b'\times
  \gset c)\otimes N(\gset b'\times \gset c) \ar[rr]^-{\theta_{\gset b'\times \gset c}} && P(\gset b'\times \gset c) \\
} \]
The pentagon commutes by definition of a Dress pairing; the upper triangle
comes from applying $N$ to a commuting triangle in $\sB_G$, thus commutes;
and the lower left quadrilateral comes from applying $M$ and $N$ to
commuting diagrams in $\sB_G$, hence commutes as well.  It follows that the
outer rectangle commutes, giving the desired naturality.

It is straightforward to check that, if the maps $\theta_{\gset b}\colon
M(\gset b)\otimes N(\gset b)\to P(\gset b)$ came from a natural
transformation $M\oboxp N\to P\circ\times$, then the construction described
above returns the original maps $M(\gset b)\otimes N(\gset c)\to P(\gset
b\times \gset c)$, and vice versa.
\end{proof}


\subsection{Explicit formula for the box product}\label{subs::formula}

There is also a general formula for left Kan extensions in terms of a coend
\cite{ML}, a particular type of colimit.  In our situation, it simplifies to
give an explicit formula for $\boxp$ in terms of coends; coends of the form
appearing below can be thought of as ``tensor products of functors.''

\begin{defn}\label{def::coend}
Let $\sC$ be a small category, and let $F\colon\sC\op\to\rmod$ be a contravariant
and $G\colon\sC\to\rmod$ a covariant functor.  The \defword{coend}
${\int^{\gset x\in\sC} F(\gset x)\otimes G(\gset x)}$, sometimes written as the tensor
product of functors ${F\otimes_\sC G}$, is the coequalizer of the diagram
of $\ground$-modules
\[ \xymatrix@C=4pc{
  \displaystyle\bigoplus_{f\colon \gset x\rightarrow \gset y} {F(\gset y)\otimes G(\gset x)}
  \ar@<2ex>[r]^{F(f)\otimes\id} \ar[r]_{\id\otimes G(f)} &
  \displaystyle\bigoplus_{\gset z\in \sC} {F(\gset z)\otimes G(\gset z)}.
}\]
That is, for each pair of objects $\gset x,\gset y$ of $\sC$, and for each
morphism $f\colon \gset x\rightarrow \gset y$ in $\sC$, we have a component
$F(\gset y)\otimes G(\gset x)$ in the left-hand direct sum.  The upper
arrow on this component $F(\gset y)\otimes G(\gset x)$ is the composite
\[ F(\gset y)\otimes G(\gset x)\xrightarrow{F(f)\otimes\id} F(\gset
x)\otimes G(\gset x) \xrightarrow{\text{inclusion}} \bigoplus_{\gset z\in
\sC} {F(\gset z)\otimes G(\gset z)},\]
while the lower arrow is
\[ F(\gset y)\otimes G(\gset x)\xrightarrow{\id\otimes G(f)} F(\gset
y)\otimes G(\gset y) \xrightarrow{\text{inclusion}} \bigoplus_{\gset z\in
\sC} {F(\gset z)\otimes G(\gset z)}.\]
\end{defn}

The coend ${\int^{\gset x\in\sC} F(\gset x)\otimes G(\gset x)}$ may thus be
thought of as a quotient of $\bigoplus_{\gset z\in \sC} {F(\gset z)\otimes
G(\gset z)}$ by appropriate equivalence relations induced by the morphisms
of $\sC$.

\begin{lem}\label{lem::boxformula}
Let $M,N$ be Mackey functors, and let $D$ be the self-duality functor of
\autoref{lem::selfdual}.  Then the value of $M\boxp N$ at $\gset c$ is
given by a coend
\[ (M\boxp N)(\gset c) \cong \int^{\gset a\in\sB_G} M(\gset a)\otimes
N(\gset c\times D\gset a).\]
For a morphism $\gset b\to \gset c$, $(M\boxp N)(\gset c)\to (M\boxp
N)(\gset b)$ is induced by $N(\gset c\times D\gset a)\to N(\gset b\times
D\gset a)$.
\end{lem}

\begin{proof}
There is a general formula for left Kan extensions in terms of a coend
\cite{ML};  in our situation, it gives
\[ (M\boxp N)(\gset{c}) = \int^{(\gset{a},\gset{b})\in\sB_\G\times\sB_\G}
M(\gset{a})\otimes N(\gset{b}) \otimes
\sB_\G(\gset{c},\gset{a}\times\gset{b}).\]
That is, the value of $M\boxp N$ at $\gset c$ is given by the tensor
product of functors of $M\oboxp N$ with $\sB_\G(\gset c, -\times -)$.

Since tensor products commute with colimits, we may rewrite the left Kan
extension formula as
\begin{align*}
  (M\boxp N)(\gset c)  & = \int^{(\gset a,\gset b)} M(\gset a)\otimes N(\gset
  b) \otimes \sB_\G(\gset c,\gset a\times\gset b) \\
  & \cong \int^{\gset a} \int^{\gset b} M(\gset a) \otimes N(\gset b) \otimes
  \sB_\G(\gset c,\gset a\times\gset b) \\
  & \cong \int^{\gset a} M(\gset a) \otimes \int^{\gset b} N(\gset b) \otimes
  \sB_\G(\gset c,\gset a\times\gset b) \\
  & \cong \int^{\gset a} M(\gset a) \otimes \int^{\gset b} N(\gset b) \otimes
  \sB_\G(\gset c\times D\gset a,\gset b) \\
  & \cong \int^{\gset a} M(\gset a) \otimes N(\gset c\times D\gset a).\\
\end{align*}
The last isomorphism comes from the Yoneda lemma, which tells us that
tensoring $N$ with a representable functor $\sB_\G(\gset c\times D\gset
a,-)$ is the same as evaluating ${N(\gset c\times D\gset a)}$.

The naturality of coends gives the final statement.
\end{proof}

\begin{rem}
It is immediate that the coend in the general formula for left Kan
extensions satisfies the desired universal property.  Note, however, that
\autoref{lem::dresspairing}, which introduces Dress pairings, does not
immediately follow from the coend in \autoref{lem::boxformula}.
\end{rem}

Using the fact that tensoring with a represented functor is the same as
evaluating at the representing object, \autoref{lem::boxformula} shows that
the represented Mackey functor $\sB_\G(-,G/G)$ is the unit for the $\boxp$
monoidal structure.  More generally,
\[ \left(\sB_\G(-,\gset b) \boxp M\right)(\gset c) = M(\gset b\times\gset
c);\]
this example will be discussed further in \autoref{subs::newmf}.
The functor represented by $G/G$ is usually called the \defword{Burnside
ring Mackey functor}.

\begin{defn}
The \defword{Burnside ring Mackey functor} $\makebox[0pt]{}_G A$, abbreviated $A$ when the
group $G$ is implicit, is the represented functor
$\sB_G(-,G/G)$\footnote{Although we have enriched $\sB_\G$ over $\rmod$,
it is more standard to enrich it over abelian groups.  In that case, the
Burnside ring Mackey functor is given by $\ground\otimes \sB_G(-,G/G)$.}.
\end{defn}

\begin{lem}
$(\sB_G,A,\boxp)$ is a symmetric monoidal category.\hfill $\Box$
\end{lem}

\begin{rem}\label{lem::unitiso}
For any Mackey functor $M$, the canonical isomorphism $A\boxp M\to M$
is adjoint to the map $A\oboxp M\to M\circ\times$ which takes ${A(\gset
b)\otimes M(\gset c)\to M(\gset b\times \gset c)}$ via ${f\otimes x\mapsto
(f\times\id)^*x}$, for a span $f$ in $\sB_G(\gset b,G/G)$.
\end{rem}


\subsection{Graded Mackey functors}

In what follows, it will often be convenient to consider the category
$\mackey^*$ of \defword{$RO(G)$-graded Mackey functors}.  As expected, a
graded Mackey functor, written $M^*$ or $M_*$ depending on context,
consists of a Mackey functor $M^\alpha$ (or $M_\alpha$) for each $\alpha
\in RO(G)$.  The Day tensor product then allows us to define a graded box
product $\boxp_*$ from the ungraded box product $\boxp$.  It can be
explicitly described by
\[ (M^* \boxp_* N^*)^\alpha \cong \bigoplus_{\beta_1+\beta_2=\alpha}
(M^{\beta_1} \boxp N^{\beta_2}).\]
$\boxp_*$ makes $\mackey^*$ into a symmetric monoidal category.  The unit
is the graded Mackey functor $A^*$ which has $A^0 = A$, the Burnside ring
Mackey functor, and $A^\alpha = 0$ for all $\alpha \ne 0$.


\subsection{Green functors and modules over Green functors}\label{subs::green}
The monoidal structure gives rise to the notion of monoids in
$\mackey$.  The monoids are Mackey functors $T$ together with a ``product
map'' ${\mu\colon T\boxp T\to T}$ and a unit map ${\eta\colon A\to T}$
satisfying the usual unit and associativity commuting diagrams\footnote{If
$f\colon M_1\to M_2$ and $g\colon N_1\to N_2$, then $f\oboxp g\colon
M_1\oboxp N_1\to M_2\oboxp N_2$ is the natural transformation with
components $f_{\gset{b}}\otimes g_{\gset{c}}\colon M_1(\gset{b})\otimes
N_1(\gset{c}) \to M_2(\gset{b}) \otimes N_2(\gset{c})$.  We may then use
the naturality in all variables of the adjunction
\[ [\sB_G\op,\rmod](M\boxp N,P) \cong
[\sB_G\op\times\sB_G\op,\rmod](M\oboxp N,P\circ\times) \]
to obtain a map $f\boxp g\colon M_1\boxp N_1 \to M_2\boxp N_2$.}:
\[
\xymatrix{
  A \boxp T \ar[r]^{\eta\boxp\id} \ar[dr]_{\cong} & T \boxp T \ar[d]^{\mu} & T \boxp A \ar[l]_{\id\boxp\eta}
  \ar[dl]^{\cong} \\
  & T & 
}
\hspace{5pt}\text{ and }\hspace{5pt}
\xymatrix{
  T\boxp T\boxp T \ar[r]^{\ \mu\boxp\id} \ar[d]_{\id\boxp\mu} & T\boxp T
  \ar[d]^{\mu} \\
  T\boxp T \ar[r]_{\mu} & T
}
\]
These monoids are known as \defword{Mackey functor
rings} or \defword{Green functors}.  The same definition, with $T$ replaced
by a graded Mackey functor $T^*$ and $\boxp$ replaced by $\boxp_*$, gives
the notion of a \defword{graded Green functor}.

\begin{rem}
\autoref{lem::dresspairing} tells us that specifying a product map $T\boxp
T\to T$ is the same as giving maps ${T(\gset b)\otimes T(\gset b)\to
T(\gset b)}$ for each $G$-set $\gset b$, satisfying the commuting square
and two commuting pentagons describing naturality with respect to maps in
$\sB_G$.  Further, by the Yoneda lemma, the unit map $\eta\colon A\to T$ is
determined by the image of the identity map $e \defeq \eta_\trivo(\id) \in
T(\trivo)$, where $\trivo$ is the terminal $G$-set, $\trivo = G/G$.  For
each $G$-set $\gset b$, we can consider the image $e_{\gset b} \defeq
r_{\gset b}(e)$ of $e$ under the restriction $r_{\gset b}$ associated to
the projection $\gset b\to \trivo$.  Then the Dress maps ${T(\gset
b)\otimes T(\gset b)\to T(\gset b)}$ make $T(\gset b)$ into a ring with
unit $e_{\gset b}$.  This follows from two facts.  First, the adjoint to
the unit diagram for the Green functor $T$ is the commuting triangle of
functors $\sB_G\op\times\sB_G\op\to\rmod$
\[ \xymatrix{
 A \oboxp T \ar[r]^{\eta\oboxp\id} \ar[dr] & T \oboxp T
 \ar[d]^{\overline{\mu}} \\
 & T\circ\times
} \]
Using \autoref{lem::unitiso}, we see that
$A\oboxp T\to T\circ\times$ takes $\id\oboxp x\mapsto x$ for any $x\in
T(\gset b)$.  It follows that $\eta_\trivo(\id)\oboxp x = e\oboxp x\mapsto x$ as
well.  We may then consider the diagram below.
\[ \xymatrix{
  T(\trivo)\otimes T(\gset b) \ar[r]^{\overline{\mu}} \ar[d]_{r_\gset
  b\otimes\id} & T(\gset b) \ar[d]_{r_{\text{proj}}} \ar@{=}[dr] \\
  T(\gset b)\otimes T(\gset b) \ar[r]^{\overline\mu} & T(\gset b\times
  \gset b) \ar[r]^{r_\Delta} & T(\gset b)
}\]
The square commutes because $T\oboxp T\to T\circ\times$ is a natural
transformation; the triangle comes from applying $T$ to a commuting diagram
in $\sB_G$.  Since the bottom composite is the Dress pairing, it follows
that $e_{\gset b} x = x$ under this pairing.
\end{rem}

\begin{rem}
The unit isomorphism $A\boxp A\to A$ makes the Burnside ring Mackey
functor $A$ into a Green functor.  It gets its name from the fact that,
when $\ground=\bZ$, and taking into consideration the levelwise ring structure
coming from the Dress pairing, its value $A(G/K)$ at the orbit $G/K$ is the
classical Burnside ring $A(K)$.
Likewise, $A^*$ is a graded Green functor.
\end{rem}

Once we have monoids in $\mackey$ and $\mackey^*$, it is a short step to
modules over the monoids.  If $T$ is a Green functor, then a $T$-module $M$
is a Mackey functor together with an action map $\phi\colon T\boxp M\to M$
satisfying the expected diagrams.
\[
\xymatrix{
  A \boxp M \ar[r]^{\eta\boxp\id} \ar[dr]_{\cong} & T \boxp M \ar[d]^{\phi} \\
  & M \\ 
}
\hspace{20pt}\text{ and }\hspace{20pt}
\xymatrix{
  T\boxp T\boxp M \ar[r]^{\ \mu\boxp\id} \ar[d]_{\id\boxp\phi} & T\boxp M
  \ar[d]^{\phi} \\
  T\boxp M \ar[r]_{\phi} & M
}
\]

By definition, every Mackey functor is a module over $A$, just as every
abelian group is a $\bZ$-module, and every graded Mackey functor is a
module over $A^*$.

Finally, we say that a Green functor $T$ is \defword{commutative} if the
multiplication ${\mu\colon T\boxp T\to T}$ satisfies the expected diagram
\[ \xymatrix{
T\boxp T \ar[rr]^{\text{switch}} \ar[dr]_\mu & & T\boxp T \ar[dl]^\mu \\
& T & \\
} \]
When we talk about ``commutative'' \emph{graded} Green functors, we will mean
``graded commutative,'' because cohomology theories give graded-commutative
graded Green functors.  Recall that a nonequivariant graded ring is graded
commutative if, for any two elements $x$ and $y$, ${yx =
(-1)^{\dim(x)\dim(y)} xy}$.  The relevance of $(-1)^{\dim(x)\dim(y)}$ here
is that $\pm 1$ are exactly the units of the ring $\bZ$, the universal
coefficients for ordinary nonequivariant cohomology.  Following the
analogous argument for equivariant cohomology, we see that changing the
order of multiplication in a graded Green functor $\ulhr^*(X;M)$ must
induce a ``sign,'' coming from the degree of a certain homotopy equivalence
between representation spheres.  This degree is necessarily a unit in the
Burnside Green functor $A^*$.  We will defer an explicit description of
signs until later, and make the following definition.  A graded Green
functor $T^*$ is \defword{commutative} if the diagram below commutes for
every $\alpha,\beta\in RO(G)$.  Here $u(\alpha,\beta)$ is a unit in the
Burnside Green functor $A^*$ which depends only on $\alpha$ and $\beta$.
\[ \xymatrix@C=3pc{
  T^\alpha\boxp T^\beta \ar[r]^{\text{switch}} \ar[d]_\mu & T^\beta\boxp T^\alpha \ar[d]_\mu \\
  T^{\alpha+\beta} \ar[r]^{u(\alpha,\beta)} & T^{\alpha+\beta} \\
} \]
We will later discuss the signs $u(\alpha,\beta)$ in more detail for $G =
\cycp$.

Since $A$ is the unit for $\boxp$, it is the analogue of $\bZ$ in abelian
groups and the natural choice of coefficient Green functor for cohomology.
In what follows, when coefficients are not specified, Burnside ring
coefficients are intended.

\subsection{Building new Mackey functors}\label{subs::newmf}

Various ways of producing new Mackey functors from old, or from objects in
other categories, will recur throughout the course of this \paper{}, and so
deserve some special mention now.

First, $\mackey$ is tensored over $\rmod$; given a Mackey functor $M$ and
\rmodule{} $X$, the Mackey functor $M\otimes X$ has $(M\otimes
X)(\gset{b}) = M(\gset{b})\otimes X$ with the appropriate morphisms.

\begin{defn}\label{def::unnamed}
Given a Mackey functor $M$ and $\gset{b}\in \sB_G$, the \defword{shifted}
Mackey functor $M_\gset{b}$ is the Mackey functor given on objects by
\[ M_\gset{b}(\gset{c}) = M(\gset{b}\times\gset{c}),\]
with the evident value on morphisms.  Alternatively, keeping in mind the
comment after \autoref{lem::boxformula}, we could define the shift of $M$
by $\gset b$ to be the Mackey functor $M\boxp A_{\gset b}$.
\end{defn}

Now consider a Mackey functor $M$.  For each subgroup $K<G$, the
automorphism group of the $G$-set $G/K$ is the Weyl group $W_G K = {N_G
K}/K$.  This acts on the left on $G/K$, viewed as an object of the Burnside
category via the covariant embedding of $\sO_G$ into $\sB_G$.  It follows
that the $\ground$-module $M(G/K)$ inherits the structure of \rarticle
$\ground[W_G K]$-module.  We therefore have a forgetful functor $\mackey\to
\ground[W_G K]\text{-mod}$  for each subgroup $K$.  When $K=\triv$, the
trivial subgroup, it turns out that this forgetful functor has both left
and right adjoints, which we will not describe explicitly here; explicit
descriptions for the simplified case $G=\cycp$ will be given later.

\begin{defn}\label{def::adjoints}
We write $\sL$ for the left adjoint to the forgetful functor $M\mapsto
M(G/\triv)$, and $\sR$ for the right adjoint.
\end{defn}

\begin{rem}\label{rem::j}
For $K\ne\triv$, we can generalize the construction of the right adjoint
$\sR$.  The generalized version will only become an adjoint
after restricting to a certain subcategory of $\mackey$, but the
construction is frequently useful.  This is discussed in detail in
unpublished works of Gaunce Lewis, such as \cite{purple}, where the right
adjoint is called $J_{G/K}$.
\end{rem}

\subsection{Free Mackey functors}\label{subs::free}

The question of which Mackey functors modules over a given Mackey functor
ring deserve to be called ``free'' will be very important later, when we
start considering whether various cohomology groups are free.  Certainly
any object called free should be projective, but not every projective is
free.

In a functor category, it is usual to define the free objects to be
coproducts of the representable functors.  The Yoneda lemma confirms that
these are projective.

\begin{prop}\label{prop::yoneda}
Let $\sC$ be a category enriched over $\rmod$.  Then for any object
$\gset{b}\in\sC$, the represented functor $A_\gset{b} = \sC(-,\gset{b})$ is
a projective functor $\sC\op\to\rmod$.
\end{prop}


\begin{proof}
By the Yoneda lemma, for any functor $F\colon \sC\op\to\rmod$, a map
$A_\gset{b}\to F$ is the same data as an element of $F(\gset{b})$.  The
correspondence is given by looking at the image of the identity map in
$\sC(\gset{b},\gset{b})$.

To show that $A_\gset{b}$ is projective, we must show that any levelwise
surjection ${\phi\colon F\to A_\gset{b}}$ admits a section $A_\gset{b} \to F$.
But, since $F(\gset{b})\to A_\gset{b}(\gset{b})$ is surjective, there is
some $y\in F(\gset{b})$ with $\phi_\gset{b}(y) = \id$.  The map
$A_\gset{b}\to F$ defined by taking $\id\mapsto y$ is then the required
section.
\end{proof}

\begin{exmp}
The category of $\ground[G]$-modules for a ring $\ground$ and group $G$ can be viewed
as the category of functors from $G$, viewed as a one-object category, to
$\rmod$.  Since $G$ has only one object $\trivo$, there is only one type of
free object, namely the free $\ground$-module on $G(-,\trivo)$, i.e.\ $\ground[G]$ itself.
So free objects are direct sums $\bigoplus \ground[G]$; we know from basic
algebra that every projective $\ground[G]$-module is a direct summand of such a
free object.
\end{exmp}

It follows that the Burnside ring Mackey functor $A = \sB_G(-,G/G)$ is
free, as it should be; but now there are other types of free objects as
well, namely $\sB_G(-,G/K)$ for each subgroup $K<G$.  As already mentioned,
the category $\sB_G$ is closed monoidal under the Cartesian product
$\times$, with internal hom also given by $\times$.  As a result, we can
give the name $A_{G/K}$ to $\sB_G(-,G/K) = \sB_G(-,G/K\times G/G)$ with no
conflict of notation with \autoref{subs::newmf}.  

For us, since the objects of $\sB_G$ are finite $G$-sets,
a coproduct of representable Mackey functors is again representable.  We
may thus make the following simplified definition.

\begin{defn}\label{def::free}
A Mackey functor $M$ is \defword{free} if it is of the form
\[ M = A_\gset{b} = \sB_G(-,\gset{b})\]
for some $\gset{b}\in\sB_G$.
\end{defn}

The free objects in $\mackey$ also tell us what the free objects in
$\mackey^*$ should be.  Writing $\Sigma^\alpha M^*$ to mean the Mackey
functor with $(\Sigma^\alpha M^*)^\beta \defeq M^{\beta-\alpha}$, the free
objects in $\mackey^*$ are coproducts of graded Mackey functors of the form
$\Sigma^\alpha A_{\gset b}^*$, for $\gset b\in \sB_G$ and $\alpha \in RO(G)$.

Finally, we will consider the question of which Mackey functor modules over
a Green functor $T$ are free.  The forgetful functor
$T\text{-mod}\to\mackey$ has a left adjoint $M\mapsto T\boxp M$.  We thus
make the following definition.

\begin{defn}\label{def::modfree}
Let $T$ be a Green functor, and $T\text{-mod}$ the category of Mackey
functor modules over $T$.  A \defword{free} $T$-module is one of the form
$T\boxp A_\gset{b} \cong T_\gset{b}$, for some $\gset{b}\in\sB_G$.
\end{defn}

Replacing $T$ by $T^*$ gives the analogous notion of a free $T^*$-module.

\section{Mackey functor valued cohomology theories}\label{sec::eqcoh}

At this point, we have informally described an $RO(G)$-graded theory
$\hr^*(-;M)$ for each Mackey functor $M$.  However, there is still some
structure arising from the $G$ action of which we are not yet taking full
advantage.  Specifically, for each subgroup $K < G$, we could consider the
theory {\renewcommand{\G}{K}$\hr^*(-;M|_K)$}.  Here $M|_K\colon
\sB_K\op\to\rmod$ is the Mackey functor with values on orbits given by
$M|_K(K/L) \defeq M(G \times_K K/L) \cong M(G/L)$.  For a fixed $\alpha$, these
theories fit together to form a Mackey functor.  This gives rise to a
Mackey functor valued cohomology theory.  We can view this as either a
collection of functors $\ulhr^\alpha(-;M)\colon \hogtop\to \mackey$, where
as before $\mackey$ is the category of Mackey functors, or as a single
functor \[\ulhr^*(-;M)\colon \hogtop\to \mackey^*\] to the category of
graded Mackey functors.  In fact, there is a cup product on $\ulhr^*$,
meaning that we can view it as a functor to the category of graded Green
functors, i.e.\ monoids in $\mackey^*$.


The extension from the $\rmod$ valued theory $\hr^*$ to the Mackey functor
valued theory $\ulhr^*$ works as follows.  On orbits, 
\[\ulhr^*(X;M)(G/K) \defeq 
\hr^*(G_+\sm_K X;M) \cong \hr^*({G/K}_+ \sm X;M) \cong
{\renewcommand{\G}{K}\hr^*(X;M|_K)}.\]
On morphisms of type
\[\cospan{\gset{b}}{\gset{c}}{}\]
the required map $\ulhr^*(X)(\gset{c})\to\ulhr^*(X)(\gset{b})$ is induced on
cohomology by the map of spaces $\gset{b}_+\sm X \to \gset{c}_+\sm X$.
On morphisms of type 
\[\contraspan{\gset{b}}{\gset{c}}{}\]
the map $\ulhr^*(X)(\gset{b})\to\ulhr^*(X)(\gset{c})$ is induced by an
appropriate transfer map $\Sigma^V\gset{c_+} \to \Sigma^V\gset{b_+}$ and
the suspension isomorphism on cohomology.  Hence $\ulhr^*$ ties together
information about the equivariant theories {\renewcommand{\G}{K}$\hr^*$}
for all subgroups $K<G$.  Also, it makes the dimension axiom look a little
cleaner.

\begin{axiom}
\defword{Mackey functor valued dimension}:
\[ \ulhr^\alpha(S^0;M) = \begin{cases}
M & \alpha = 0\\
0 & \alpha \in \bZ, \alpha\ne 0.
\end{cases}\]
For $\alpha\notin \bZ$, $\ulhr^\alpha(S^0;M)$ is not specified by the
dimension axiom but is uniquely determined by the cohomology axioms.
\end{axiom}

Note that we need only state the axiom for $S^0 = {G/G}_+$: by definition,
the Mackey functor $\ulhr^*({G/K}_+;M)$ is the shifted Mackey functor
${\ulhr^*(S^0;M)\boxp A_{G/K}}$.

\subsection{The equivariant stable category}
We could alternatively approach Mackey functor valued cohomology from a
stable point of view.  It is known that the homotopy category of
genuine\footnote{i.e.\ indexed on $RO(G)$} $G$-spectra is enriched over the
category of Mackey functors.  Specifically, given genuine $G$-spectra $D$
and $E$, we can define a Mackey functor $\mapmf{D}{E}$ on objects by
\[ \mapmf{D}{E} (\gset{b}) \defeq [D \sm \gset{b}_+ , E]_G = [D \sm
\Sigma^\infty \gset{b}_+, E]_G.\]
When $D = \Sigma^\alpha S$, a suspension of the sphere spectrum, this construction
gives the \defword{homotopy Mackey functor}
\[ \htmf_\alpha(E) \defeq \mapmf{\Sigma^\alpha S}{E}.\]
The following theorem, which appears in an appendix to \cite{LM}, shows
that the homotopy Mackey functor interacts nicely with the smash product on
the equivariant stable category.  Lewis and Mandell say the theorem was
folklore long before their paper.

\begin{thm}
The functor $\htmf_*$ is a lax symmetric monoidal functor from the
genuine equivariant stable category (with monoidal structure given by the
smash product) to the category of $RO(G)$-graded Mackey functors.  
\end{thm}

That is, for $\alpha,\beta\in RO(G)$ and for spectra $D$ and $E$, there is
an associative, symmetric, and unital natural transformation
\[ \htmf_\alpha(D) \boxp \htmf_\beta(E) \to \htmf_{\alpha+\beta}(D\sm E).\]
These fit together to form a natural transformation
\[ \htmf_*(D) \boxp_* \htmf_*(E) \to \htmf_*(D\sm E)\]
with the same properties.

For each Mackey functor $M$, there is an \defword{Eilenberg-MacLane
$G$-spectrum} $HM$ which has the property that, as Mackey functors,
\[ \htmf_n(HM) \cong 
\begin{cases}
M  & n = 0\\
0  & n\in\bZ,\ n \ne 0.
\end{cases}\]
(See \cite{LMMc} or \cite{AK} for a proof.)  $HM$ represents
$RO(G)$-graded, Mackey functor valued Bredon cohomology, in the sense that 
\[ \ulhr^{\alpha}(X;M) \cong \mapmf{\Sigma^\infty X}{\Sigma^\alpha HM}\]
for a based space $X$.

Corresponding to these cohomology theories are homology theories, given by 
\[ \ulhl_\alpha(X;M) \cong \mapmf{\Sigma^\alpha S}{X_+\sm HM}.\]

\subsection{Summary of important facts about homology and cohomology}

Most of the arguments in this  \paper{} will rely on the following small
collection of facts about $RO(G)$-graded, Mackey functor valued cohomology.
They are collected here for easy reference later.

\begin{axioms}\label{ax::finalcoh}
A reduced ordinary equivariant cohomology theory $\ulhr^*$ indexed on
$RO(G)$ consists of a collection of functors $\hr^\alpha\colon \hogtop\to\mackey$
satisfying the following axioms.
\begin{enumerate}
\item \label{ax::we} \defword{Weak equivalence}: Weak $G$-equivalences induce
isomorphisms on $\hr^*$.

\item \label{ax::exact} \defword{Exactness}: if $i\colon  A\to X$ is a
$G$-cofibration with cofiber $X/A$, then the sequence \[ \ulhr^\alpha(X/A)
\to \ulhr^\alpha(X) \to \ulhr^\alpha(A) \] is exact in the abelian category
of Mackey functors for each $\alpha\in RO(G)$.

\item \label{ax::add} \defword{Additivity}: if $X = \bigvee_i X_i$ as based
spaces, then the inclusions $X_i\to X$ induce isomorphisms $\ulhr^*(X)\to
\prod_i \ulhr^*(X_i)$.

\item \label{ax::susp} \defword{Suspension}: for each $\alpha\in RO(G)$ and
representation $V$, there is a natural isomorphism
\[ \Sigma^V \colon  \ulhr^\alpha(X) \to \ulhr^{\alpha+V}(\Sigma^V X) =
\ulhr^{\alpha+V}(S^V \sm X).\]

\item\label{ax::dim}
\defword{Dimension}: $\ulhr^n(S^0) = 0$ for all integers $n\in\bZ$, $n\ne 0$.

\item\label{ax::bookkeeping}
\defword{Bookkeeping}: $\Sigma^0$ is the identity natural transformation,
and the suspension isomorphism ${\Sigma^V\colon \ulhr^\alpha(X)\to
\ulhr^{\alpha+V}(\Sigma^V X)}$ is covariantly natural in $V$ and
contravariantly natural in $X$.  Let $V$, $W$, and $W'$ be $G$
representations, and $f:S^W\to S^{W'}$ the stable homotopy class of the map
associated to a $G$-linear isometric isomorphism $W\to W'$.  Then
$\Sigma^V\Sigma^W \cong \Sigma^{V\oplus W}$ and the following diagrams commute:
\[ \xymatrix{
\ulhr^\alpha(X) \ar[r]^-{\Sigma^W} \ar[d]_{\Sigma^{W'}} & \ulhr^{\alpha+W}(\Sigma^W X) \ar[d]^-{\ulhrnoph^{\id+f}(\id)} \\
\ulhr^{\alpha+W'}(\Sigma^{W'} X) \ar[r]^{(\Sigma^f\id)} & \ulhr^{V+W'}(\Sigma^W X) \\
}
\hspace{30pt}
\xymatrix@C=.2pc{
\ulhr^\alpha(X) \ar[dr]_(0.4){\Sigma^{V+W}} \ar[rr]^{\Sigma^V} & & \ulhr^{\alpha+V}(\Sigma^V X) \ar[dl]^(0.37){\Sigma^W} \\
& \makebox[20pt]{$\ulhr^{\alpha+V+W}(\Sigma^{V+W} X)$} & \\
} \]
\end{enumerate}
\end{axioms}

Similar axioms hold for reduced homology $\ulhlr_*(-;M)$.

Note that, since $\ulhr^*(S^0;M) \ne M$ as a graded Green functor, saying
that something is a module over the cohomology of a point is very different
from saying it is a module over the coefficient Mackey functor $M$.

\begin{fact}
If $M$ is a Green functor (Mackey functor ring), then unreduced cohomology
$\ulh^*(-;M)$ takes values in graded commutative Green functors.  In
particular, for every space $X$ and $\alpha,\beta\in RO(G)$, we have maps
\[\ulhr^\alpha(X_+;M)\boxp \ulhr^\beta(X_+;M) \to \ulhr^{\alpha+\beta}(X_+;M).\]
\end{fact}

\begin{fact}
If $M$ is a Green functor, each $\ulhr^*(X_+;M)$ is an algebra over the
cohomology $\ulhr^*(S^0;M)$ of a point, in the sense that there is a map of
Green functors
\[\ulhr^*(S^0;M)\boxp \ulhr^*(X_+;M)\to \ulhr^*(X_+;M)\]
satisfying the expected diagrams.
\end{fact}

\begin{fact}
The homology theory $\ulhlr_*(-;A)$ has a Hurewicz map, i.e.\ a natural
transformation $\htmf_*(-) \to \ulhlr_*(-;A)$.
\end{fact}

For the following, see \cite{LMS} or \cite{wirth}.

\begin{fact}
For any group $G$, the orbits ${G/K}_+$ satisfy equivariant
Spanier-Whitehead duality: there is an isomorphism in reduced (co)homology
\[ \ulhr^\alpha({G/K}_+;A) \cong \ulhlr_{-\alpha}({G/K}_+;A).\]
\end{fact}

%

\section{Mackey functors for $G=\cycp$}\label{sec::modp}

\renewcommand{\G}{\cycp}

When $p$ is a prime, the cyclic group $\cycp$ has only two subgroups:
$\triv$ and $\cycp$ itself.  As a result, it is not difficult to give
explicit descriptions of concrete Mackey functors when the group $G$ of
equivariance is $\cycp$.  \cite{lgl}, and this \paper{}, are concerned with 
calculations in cohomology $\ulhr^*(-;A)$ for $G=\cycp$, and so we will
introduce here some explicit notation which will be useful later on.

\subsection{Diagrams for Mackey functors}

Recall from \autoref{def::mf} that a Mackey functor $M$, as an additive functor
$\sB_G\op\to\rmod$, is determined by its values $M(G/K)$ on orbits.  In
particular, when $G = \G$, Mackey functors are determined by their values
at $\G/\G$ and $\G/\triv$.  For convenience, we will write $\G/\G = \trivo$
(meant to suggest a single fixed point) and $\G/\triv = \freeo$ (meant to
suggest all of $\G$).

Rephrasing the definition, Mackey functors are additive $\sB_\G$-shaped
diagrams in $\rmod$.  The adjective ``additive'' means that we can restrict our
attention to the full subcategory of $\sB_\G$ on objects $\trivo$ and
$\freeo$.  Considering the maps, i.e.\ spans, between these objects, we see
that $\sB_\G$ can be thought of as the category enriched over $\rmod$
generated by the following diagram, modulo one relation.
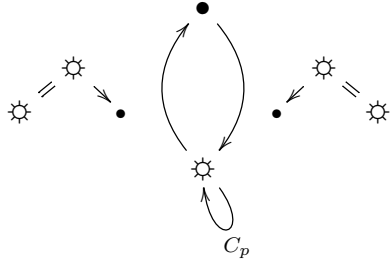
\[ \xymatrix@R=4pc@C=4pc{
\trivo \ar@/^3ex/[d]^(0.2){\hspace{1ex}\minicontraspan{\freeo}{\trivo}{}} \\
\freeo \ar@(dr,d)[]^{\cycp} \ar@/^3ex/[u]^(0.8){\minicospan{\freeo}{\trivo}{}\hspace{9ex}} \\
} \hspace{8ex}\]
Explicitly, the automorphism group of $\trivo$ is a free $\ground$-module on two
generators, the identity span and the composite of the downward arrow
followed by the upward one.  $\sB_\G(\trivo,\freeo)$ and
$\sB_\G(\freeo,\trivo)$ are each free on one generator.  The automorphism
$\ground$-module of $\freeo$ is the free $\ground$-module on generators 
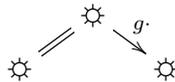
\[ \cospan{\freeo}{\freeo}{g\cdot}\]
for each $g\in\G$, and so has an action of $\G$.  The one relation comes
from the fact, straightforward to check explicitly, that the composite of
the upward followed by the downward arrow is equal to the sum of the $p$
generators of $\sB_\G(\freeo,\freeo)$.

The consequence of all of this is that we have an explicit description of
Mackey functors for $G=\G$.  A Mackey functor $M\colon \sB_\G\op\to\rmod$
consists of a diagram of the form
\[\mf{M(\trivo)}{M(\freeo)}{\mft}{\mfr}{\cycp}\]
where $\mft$ and $\mfr$ are the restriction and transfer associated to the
projection of $G$-sets $\rho\colon \freeo\to\trivo$; $M(\freeo)$ is
\rarticle $\ground[\G]$-module; $M(\trivo)$ is \rmodule{} which can be
viewed as \rarticle $\ground[\G]$-module with the trivial action; and all arrows
are maps of $\ground[\G]$-modules.  The one relation discussed above shows
that $\mfr\mft x = \sum_{g\in\G} gx = \trace(x)$ for each $x\in M(\freeo)$.
Conversely, any diagram satisfying these conditions ($M(\trivo)$ is
\rarticle $\ground[\G]$-module with the trivial action, $M(\freeo)$ is
\rarticle $\ground[\G]$-module, $\mfr\mft = \trace$) defines a Mackey
functor.

\subsection{Some important Mackey functors}\label{subs::favoritemf}

Since Mackey functors for $G=\G$ are so simple, we can give explicit
descriptions of all of the Mackey functors which will appear in our later
calculations.

We have already seen the Burnside ring Mackey functor $A =
\sB_\G(-,\trivo)$.  $\sB_\G(\trivo,\trivo)$ is generated by the two spans
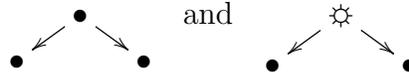
\[ 
\xymatrix@R=0.5pc@C=1pc{ & \trivo \ar[dl] \ar[dr] & \\ \trivo & & \trivo\\}
\hspace{5pt}\text{ and }\hspace{5pt}
\xymatrix@R=0.5pc@C=1pc{ & \freeo \ar[dl] \ar[dr] & \\ \trivo & & \trivo\\}
\]
while $\sB_\G(\freeo,\trivo)$ is generated by the single span
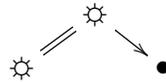
\[ \cospan{\freeo}{\trivo}{}\]
It follows that $A(\trivo) \cong \ground\mu \oplus \ground\tau$, a free
module on generators $\mu$ corresponding to the identity span and $\tau$
corresponding to the non-identity span above.  $A(\freeo) \cong
\ground\iota$, a free module on the single generator $\iota$, with the
trivial $\G$ action.  The arrows in $A$ are given by precomposition with
appropriate spans; it follows that $\mft(\iota) = \tau \in A(\trivo)$,
$\mfr(\mu) = \iota$, and $\mfr(\tau) = \mfr(\mft(\iota)) = p\iota$.  Thus
the Mackey functor $A$ has diagram
\[\mf{\ground\oplus \ground}{\ground}{\mymatrix{0 & 1}}{\mymatrix{1\\p}}{\text{triv}}\]

We can modify this construction slightly to produce what Lewis calls
\defword{twisted Burnside} Mackey functors.
For any integer $d$, define $\tw{A}d$ to be the
Mackey functor given by\label{ex::twisted}
\[ \mf{\ground\oplus \ground}{\ground}{\mymatrix{0&1}}{\mymatrix{d\\p}}{\text{triv}}\]

\begin{lem}\label{lem::twequiv}
There is an isomorphism $\tw{A}{d_1} \cong \tw{A}{d_2}$ if and only if
there is a unit $u\in \ground$ and an element $x\in \ground$ such that $d_1 = u d_2 +
px$.
\end{lem}

\begin{proof}
Suppose that $d_1 = u d_2 + px$ for some unit $u$ and element $x$.  Then
the following is a map of Mackey functors.
\[ \xymatrix@R=3pc@C=7pc{
  {\ground\oplus \ground} \ar@/_2ex/[d]_{\mymatrix{d_1\\p}} \ar[r]^{\mymatrix{u&x\\0&1}} 
  & {\ground\oplus \ground} \ar@/_2ex/[d]_{\mymatrix{d_2\\p}} \\
  {\ground} \ar@(d,dr)[]_{\text{triv}} \ar@/_2ex/[u]_{\mymatrix{0&1}} \ar@{=}[r]
  & {\ground} \ar@(d,dr)[]_{\text{triv}} \ar@/_2ex/[u]_{\mymatrix{0&1}} \\ 
} \]
It is an isomorphism because the matrix $X$ at the top of the diagram has
determinant $u$.

For the other direction, observe that, in any isomorphism $\tw{A}{d_1}\cong
\tw{A}{d_2}$, the bottom horizontal map must be either the identity or
multiplication by a unit of $\ground$; we may assume without loss of generality
that it is the identity.  Since $X$ must then satisfy $\mymatrix{0&1}X =
\mymatrix{0&1}$, the bottom row of $X$ is determined.  Since the
determinant of $X$ must be invertible in $\ground$, the upper left entry must
be a unit $u$.  It follows that, in any isomorphism $\tw{A}{d_1}\cong
\tw{A}{d_2}$, we must have $d_2 = u d_1 + px$ for a unit $u$ and $x\in \ground$.
\end{proof}

Next, for any $\ground$-module $C$, we have the Mackey functor $\langle
C\rangle$ given by
\[ \mf{C}{0}{}{}{} \]
where all maps are the zero map.  This is the special case of the functor
$J_{\trivo}$ mentioned in \autoref{rem::j}.

Although the twisted Burnside Mackey functors are generally distinct
for different $d$, modulo the isomorphisms of \autoref{lem::twequiv},
adding a copy of $\jg{\ground}$ to two twisted Burnside Mackey functors
almost always produces isomorphic Mackey functors; we will make this
precise in \autoref{lem::twiso}.  This may be slightly
surprising at first, but it is also necessary in order for $RO(G)$-graded
cohomology to be well defined.

\begin{lem}\label{lem::twiso}
For any integers $d_1,d_2$ prime to $p$, there is an isomorphism
\[\tw{A}{d_1}\oplus \jg{\ground} \cong \tw{A}{d_2}\oplus \jg{\ground}.\]
\end{lem}

\begin{proof}
Since the $d_i$ are prime to $p$, we can find integers $a_i,b_i$ such that
$a_i d_i + b_i p = 1$.  Then it can be checked that the following map of
Mackey functors is an isomorphism.
\[ \xymatrix@R=3pc@C=12pc{
  {\ground\oplus \ground\oplus \ground} \ar@/_2ex/[d]_{\mymatrix{d_1\\p\\0}}
  \ar[r]^{\mymatrix{d_1 a_2 & d_1 b_2 & (b_1+b_2-b_1 b_2 p)\\0&1&0\\p&-d_2
  & -a_1 d_2}} & {\ground\oplus \ground\oplus \ground} \ar@/_2ex/[d]_{\mymatrix{d_2\\p\\0}} \\
  {\ground} \ar@(d,dr)[]_{\text{triv}} \ar@/_2ex/[u]_{\mymatrix{0&1&0}} \ar@{=}[r] &
  {\ground} \ar@(d,dr)[]_{\text{triv}} \ar@/_2ex/[u]_{\mymatrix{0&1&0}} \\ 
} \]
The matrix $X$ at the top of the diagram is constructed as follows: first,
$X$ should satisfy $X \mymatrix{d_2\\p\\0} = \mymatrix{d_1\\p\\0}$ and 
$\mymatrix{0&1&0}X = \mymatrix{0&1&0}$, which determines the second row and
the first two columns.  The entries in the final column are then chosen to
make the determinant equal to 1.
\end{proof}

\begin{rem}
\autoref{lem::twiso} holds more generally if we allow $d_1,d_2$ to be
elements of $\ground$ such that there exist $a_i,b_i\in \ground$ satisfying ${a_i d_i +
b_i p = 1}$.
\end{rem}

We are also in a position to describe the left and right adjoints $\sL$ and
$\sR$ to the forgetful functor $M\mapsto M(\freeo)$.  If $B$ is \rarticle
$\ground[\G]$-module, then $\sL(B)$ and $\sR(B)$ are given by
\[
\xymatrix@R=0.7pc{ \\ \sL(B) =  } 
\mf{B/\G}{B}{\proj}{\trace}{}
\hspace{30pt}
\xymatrix@R=0.7pc{ \\ \sR(B) =  } 
\mf{B^\G}{B}{\trace}{\text{inclusion}}{}
\]
The trace map takes $[x] \mapsto \sum_g {gx}$ in $\sL(B)$ and
$x\mapsto \sum_g {gx}$ in $\sR(B)$.  The cases $B = \ground$ with either the
trivial action or, when $p = 2$, the sign action, will come up frequently.
For $\ground$ with the trivial action, we have
\[
\xymatrix@R=0.7pc{ \\ \sL(\ground) =  } 
\mf{\ground}{\ground}{\id}{\cdot p}{\text{triv}}
\hspace{30pt}
\xymatrix@R=0.7pc{ \\ \sR(\ground) =  } 
\mf{\ground}{\ground}{\cdot p}{\id}{\text{triv}}
\]
If $p = 2$ and $\ground$ has no 2-torsion, let $\ground_{-}$ denote
$\ground$ with the sign action.  Then we have
\[
\xymatrix@R=0.7pc{ \\ \sL(\ground_{-}) =  } 
\mf{\ground/2}{\ground}{\proj}{0}{-1}
\hspace{30pt}
\xymatrix@R=0.7pc{ \\ \sR(\ground_{-}) =  } 
\mf{0}{\ground}{0}{0}{-1}
\]
Here and in the rest of this \paper{}, if $C$ is \rmodule{}, we write $C/p$
for the $\ground$-module $C/{pC} = C\otimes_\bZ \bZ/p$.  Since the
instances of $\sL$ and $\sR$ above come up so frequently, we will give them
their own names:
\[ \r \defeq \sR(\ground) \hspace{4ex}
   \rminus \defeq \sR(\ground_{-}) \hspace{4ex}
   \l \defeq \sL(\ground) \hspace{4ex}
   \lminus \defeq \sL(\ground_{-}) \hspace{4ex}
\]

Now that we have assigned names to the Mackey functors above, we state the
following immediate corollaries of \autoref{lem::twequiv}.  This explains why
we are usually only interested in $\tw{A}{d}$ when $d$ is not a multiple of
$p$.

\begin{cor}
If $d = px$ for some $x\in \ground$, then $\tw{A}{d}$ splits up as the direct sum
$\jg{\ground}\oplus \l$.
\end{cor}

\begin{cor}
If $p$ is invertible in $\ground$, then for any $d$, we have
\[\tw{A}{d} \cong A \cong \jg{\ground}\oplus \l.\]
\end{cor}

Finally, for any Mackey functor $M$, we defined the shifted Mackey functor
$M_\freeo$ in \autoref{def::unnamed}: $M_\freeo(\trivo) = M(\freeo)$ and
$M_\freeo(\freeo) = M(\freeo\times\freeo) \cong \oplus_{g\in\G} M(\freeo)$.
The choice of representatives for the orbits in $\freeo\times\freeo$
affects the choice of basis in the identification $M_\freeo(\freeo)\cong
\oplus_{g\in\G} M(\freeo)$; we will make the choice which yields the
diagram below.\label{def::pshifted}
\[
\mf{M(\freeo)}{M(\freeo)^{\oplus p}}{\text{fold}}{\text{diag}}{\text{perm}}
\]
In particular, the represented Mackey functor $A_\freeo$ has the following
diagram.
\[ \mf{\ground}{\ground^{\oplus p}}{\text{fold}}{\text{diag}}{\text{perm}} \]
We know from \autoref{subs::free} that a map of Mackey functors $A\to M$ is
determined by an element of $M(\trivo)$, and that a map $A_\freeo \to M$
is determined by an element of $M(\freeo)$.

\subsection{Box Products}\label{subs::modpbox}

We have already mentioned the box product $\boxp$, which gives the category
$\mackey$ of Mackey functors a monoidal structure.  There are two ways of
looking at this, in terms of the defining universal property and in terms
of an explicit formula; both will be useful at various times.

We know from \autoref{lem::boxformula} that, for Mackey functors $M$ and
$N$,
\[ (M\boxp N)(\gset c) \cong \int^{\gset a} M(\gset a)\otimes N_{\gset
c}(D\gset a) = M\otimes_{\sB_\G} (N_{\gset c}\circ D).\]

Since $\sB_\G$ is generated by the objects $\trivo$ and $\freeo$, we can
simplify the above formula even further.\footnote{Note that one of the
relations is accidentally omitted in \cite{lgl}.}

\begin{lem}\label{lem::pboxformula}
Given Mackey functors $M$ and $N$, their box product $M\boxp N$ is given by
the diagram 
\[ \mf
    {\left[ M(\trivo)\otimes N(\trivo) \oplus M(\freeo)\otimes
    N(\freeo)\right]/\sim}
    {M(\freeo)\otimes N(\freeo)}
    {\mft}{\mfr}{\G}\]
Using $\mfr$ and $\mft$ to also mean the restriction and transfer maps in
$M$ and $N$, the equivalence relation is generated by the relations
\[ {\setlength{\arraycolsep}{1pt}
\begin{array}{rll}
  x\otimes \mft y &\sim \mfr x\otimes y & \mathrm{for}\ x\in M(\trivo),\ y\in N(\freeo) \\
  \mft w \otimes z &\sim w\otimes \mfr z & \mathrm{for}\ w\in M(\freeo),\ z\in N(\trivo) \\
  g w \otimes y &\sim w\otimes g^{-1} y\hspace{2ex} & \mathrm{for}\ w\in M(\freeo),\ y\in N(\freeo),\ g\in\G\\
\end{array}} \]
The action of $\G$ on $(M\boxp N)(\freeo)$ is the diagonal action on
${M(\freeo)\otimes N(\freeo)}$, i.e. ${w\otimes y\longmapsto gw\otimes gy}$.
The transfer $\mft$ of the box product is induced by the inclusion
\[{M(\freeo)\otimes N(\freeo)} \to M(\trivo)\otimes N(\trivo) \oplus
M(\freeo)\otimes N(\freeo).\]
The restriction $\mfr$ of the box product is induced by $\mfr\otimes\mfr$
on $M(\trivo)\otimes N(\trivo)$ and by the trace on ${M(\freeo)\otimes
N(\freeo)}$, ${w\otimes y\longmapsto \sum_{g\in\G} gw\otimes gy}$.
\end{lem}

\begin{proof}
As remarked above, $(M\boxp N)(\gset c) \cong M\otimes_{\sB_\G} (N_{\gset
c}\circ D)$, which is given by the coequalizer of \autoref{def::coend}.
For this coequalizer, we may restrict our attention to the objects
$\trivo,\freeo$ and the generating spans
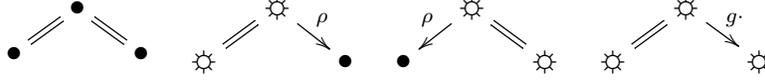
\[ \idspan{\trivo}
\hspace{10pt}\cospan{\freeo}{\trivo}{\rho}
\hspace{10pt}\contraspan{\freeo}{\trivo}{\rho}
\hspace{10pt}\cospan{\freeo}{\freeo}{g\cdot} \]
of $\sB_\G$.  It is then an instructive exercise to arrive at the explicit
formulas above.
\end{proof}


Similarly, \autoref{lem::dresspairing} gives us an explicit description of maps
out of $M\boxp N$.

\begin{prop}\label{lem::pboxmaps}
A map of Mackey functors ${f\colon M\boxp N\to P}$ is
determined by two maps
\[ {\setlength{\arraycolsep}{1pt}
\begin{array}{ccccc}
M(\trivo)& \otimes & N(\trivo) & \xrightarrow{f_\trivo} & P(\trivo) \\
M(\freeo)& \otimes & N(\freeo) & \xrightarrow{f_\freeo} & P(\freeo) \\
\end{array} }\]
such that the following four formulas, called the Frobenius relations, hold
for all $x\in M(\trivo)$, $x'\in M(\freeo)$, $y\in N(\trivo)$, and $y'\in
N(\freeo)$.\footnote{Note that \cite{lgl} forgets to mention the map $g$
when giving the compatibility relations.} 
\begin{align*}
f_\freeo (\mfr x \otimes \mfr y) & = \mfr f_\trivo(x\otimes y) \\
f_\trivo (\mft x' \otimes y) & = \mft f_\freeo(x'\otimes \mfr y) \\
f_\trivo (x\otimes \mft y') & = \mft f_\freeo(\mfr x \otimes y') \\
f_\freeo (g x' \otimes g y') & = g f_\freeo(x'\otimes y') \\
\end{align*}
\end{prop}

\begin{proof}
By \autoref{lem::dresspairing}, a map ${f\colon M\boxp N\to P}$ is given by two
maps $f_\trivo$ and $f_\freeo$, as above, such that the square and two
pentagons in the definition of a Dress pairing commute for the projection
map of $\G$-sets ${\rho:\freeo\to\trivo}$ and the multiplication map
${g:\freeo\to\freeo}$, where $g$ is a generator of $\G$.  Since the map $g$
is an isomorphism, $t_g = (r_g)^{-1}$ as in the remark after
\autoref{lem::dresspairing}, and so the diagrams of \autoref{lem::dresspairing}
reduce to the four formulas shown.
\end{proof}

\begin{exmp}\label{ex::l}
For any Mackey functor $M$ and $\ground[\G]$-module $B$, we have
\[ \sL(B) \boxp M \cong \sL(B\otimes M(\freeo)).\]
This comes from the fact that $\boxp$ actually
makes $\mackey$ into a closed monoidal category.  That is, there is an
internal hom $\langle -,-\rangle$ such that
\[ \mackey(M\boxp N,P) \cong \mackey(M,\langle N,P\rangle),\]
naturally in all variables.  The claim then follows from the following
facts: $\boxp$ is a colimit; colimits commute with left adjoints; and
$\langle N,P\rangle (\freeo) = \rmod(N(\freeo),P(\freeo))$.
\end{exmp}

\begin{rem}
In fact, for any finite group $G$, if we let $\freeo = G/\triv$, we have 
\[ (M\boxp N)(\freeo) \cong M(\freeo)\otimes N(\freeo) \]
and
\[ \langle M,N\rangle (\freeo) \cong \rmod(M(\freeo),N(\freeo));\]
this much is not special to $\G$.
\end{rem}

\begin{rem}
$\langle -,-\rangle$ is a right Kan extension, given by an end dual to the
coend describing $\boxp$.  The analogue to \autoref{lem::boxformula} shows
that $\langle M,N\rangle (\gset c) \cong \Hom_{\sB_G} (M,N_{\gset c})$,
i.e.\ natural transformations from $M$ to $N_{\gset c}$.
\end{rem}

\begin{exmp}\label{ex::jg}
It is immediate from \autoref{lem::pboxformula} that for any Mackey functor
$M$ and $\ground$-module $C$,
\[ \jg{C} \boxp M \cong \jg{C\otimes (M(\trivo))/{(\text{image }\mft)}}.\]
\end{exmp}

\begin{exmp}
We see by explicit computation that, for integers $c$ and $d$ relatively
prime to $p$,
\[ \tw{A}{d}\boxp \tw{A}{c} \cong \tw{A}{cd}.\]
To show this using \autoref{lem::pboxformula}, we need to identify
\[{[\left(\tw{A}{d}(\trivo)\otimes \tw{A}{c}(\trivo)\right) \oplus
\left(\tw{A}{d}(\freeo)\otimes \tw{A}{c}(\freeo)\right)]/\!\sim}\]
with ${\tw{A}{cd}(\trivo) \cong \ground\oplus \ground}$.  One approach is to choose the
standard generators $\mu_d,\tau_d \in \tw{A}{d}(\trivo)$ and $\iota_d\in
\tw{A}{d}(\freeo)$ for $\tw{A}{d}$, i.e.\ the ones which give the displayed
diagram on page \pageref{ex::twisted}, and similarly for $\tw{A}{c}$ and
$\tw{A}{cd}$.  Then the map ${\tw{A}{cd}(\trivo)\to
(\tw{A}{d}\boxp\tw{A}{c})(\trivo)}$ sending $\mu_{cd}\mapsto \mu_d\otimes
\mu_c$ and $\tau_{cd}\mapsto \mft(\iota_d\otimes\iota_c)$ gives the desired
isomorphism.
\end{exmp}

\begin{exmp}\label{ex::altbasis}
Box products with the Mackey functors $\sR(B)$ are more complicated, since
right adjoints do not generally commute with colimits.  However, for $B =
\ground$, viewed as \rarticle $\ground[G]$-module with the trivial action,
we can calculate ${\r\boxp \tw{A}{d} \cong \r}$ and ${\r\boxp \r \cong \r}$
using \autoref{lem::pboxformula}.  To see the first isomorphism, choose
integers $a_d,b_d$ such that $a_d d+b_d p=1$, and define a new basis for
$\tw{A}{d}(\trivo)$ by $\sigma_d = a\mu_d + b\tau_d$, $\kappa_d = p\mu_d -
d\tau_d$.  Ambiguously using $1$ to denote the generator at each level of
$\r$, the isomorphism $\r\to \tw{A}{d}\boxp \r$ is given by sending
$1\mapsto \sigma_d\otimes 1$ at the $\trivo$ level.
\end{exmp}

\begin{exmp}
Similarly, \autoref{lem::pboxformula} tells us the following box products
with $\rminus$:
\[ {\setlength{\arraycolsep}{2pt}
\begin{array}{rclcl}
\rminus & \boxp & \r &\cong   &  \rminus \\
\rminus & \boxp & \rminus &\cong & \l \\
\end{array}} \]
In particular, even though $\rminus(\trivo) = 0$,
$(\rminus\boxp\rminus)(\trivo) \ne 0$!
\end{exmp}

\begin{exmp}
Using the description of $M_\freeo$, we have
\[ A_\freeo \cong \l_\freeo \cong \r_\freeo \cong {\lminus}_\freeo \cong
{\rminus}_\freeo.\]
\end{exmp}

Putting all of these examples together, we have the 
multiplication table in \autoref{fig::multtable} for the Mackey functors introduced in
\autoref{subs::favoritemf}. 
\begin{figure}
\begin{center}
{\renewcommand{\arraystretch}{1.2}
\begin{tabular}{c|cccccc}
 & $A_{\freeo}$ & $\tw{A}{d}$ & $\jg{X_1}$ & $\sL(B_1)$ & $\r$ & $\rminus$ \\
\hline
$A_{\freeo}$ & $A_{\freeo\times\freeo}$ & $A_\freeo$ & $\jg{0}$ & $\sL(B_1^{\oplus p})$ & $A_{\freeo}$ & $A_{\freeo}$ \\
$\tw{A}{c}$ & $A_\freeo$ & $\tw{A}{cd}$ & $\jg{X_1}$ & $\sL(B_1)$ & $\r$ & $\rminus$ \\
$\jg{X_2}$ & $\jg{0}$ & $\jg{X_2}$ & $\jg{X_1\otimes X_2}$ & $\jg{0}$ & $\jg{ {X_2}/p }$ & $\jg{0}$ \\
$\sL(B_2)$ & $\sL{(B_2^{\oplus p})}$ & $\sL{(B_2)}$ & $\jg{0}$ & $\sL{(B_1\otimes B_2)}$ & $\sL{(B_2)}$ & $\sL(B_1\otimes \ground_{-})$ \\
$\r$ & $A_{\freeo}$ & $\r$ & $\jg{ {X_1}/p }$ & $\sL(B_1)$ & $\r$ & $\rminus$ \\
$\rminus$ & $A_{\freeo}$ & $\rminus$ & $\jg{0}$ & $\sL(B_1\otimes \ground_{-})$ & $\rminus$ & $\l$ \\
\end{tabular}}
\end{center}
\caption[\lofspace{}Multiplication table]{Multiplication table for some relevant Mackey functors.  Here,
$c,d\in\bZ$ are relatively prime to $p$; $X_1$, $X_2$ are
$\ground$-modules; and $B_1$, $B_2$ are $\ground[\G]$-modules.  We continue
to write $X/p$ for ${X\otimes_\bZ \bZ/p}$.}
\label{fig::multtable}
\end{figure}

%

\section{Bredon Cohomology for $G = \G$}

Using the definition of Mackey functor valued $RO(\G)$-graded Bredon
cohomology $\ulhr^*(-;M)$ in terms of representing Eilenberg-MacLane
spectra in the case $G=\G$, we see that the $\alpha$ cohomology
$\ulhr^\alpha(X;M)$ and homology $\ulhlr_\alpha(X;M)$ of a $\G$-space $X$
is given by the diagrams
\[
\mf{\hr^\alpha(X;M)}{\neqhr^{|\alpha|}(X;M(\freeo))}{\text{transfer}^*}{\text{proj}^*}{}
\hspace{30pt}\text{and}\hspace{30pt}
\mf{\hlr_\alpha(X;M)}{\neqhr_{|\alpha|}(X;M(\freeo))}{\text{proj}_*}{\text{transfer}_*}{}
\]
The identification of $\ulhr^\alpha(X;M)(\freeo)$ with the nonequivariant
cohomology group shown uses the adjunction $[\G_+\sm X,HM]_\G \cong [X,HM]$
together with the fact that the underlying nonequivariant spectrum of $HM$
is a $H(M(\freeo))$.  The $\G$ action comes from the action on the
$\ground[\G]$-module $M(\freeo)$.

Many of the results in the remainder of this section will implicitly use
the structure of $RO(\G)$, the free abelian group with generators the
irreducible real representations of $\G$.  When $p=2$, $C_2$ has exactly
two irreducible representations: the one-dimensional trivial
representation, and the one-dimensional sign representation.

For $p>2$, there are no nontrivial one-dimensional representations of
$\G$.  Every nontrivial irreducible representation is the underlying
two-dimensional real representation of a one-dimensional complex
representation of $\G$.  There are $p-1$ nontrivial complex
representations, given by choosing a generator $g\in\G$ and specifying that
$g z = e^{2\pi i k/p} z$ for some integer $k$.  When we forget down to
the underlying real representation, we also forget about orientation, and
so the representations given by $gz = e^{2\pi i k/p} z$ and $g z = e^{2\pi i
(-k)/p} z$ are identified.  It follows that the irreducible real
representations of $\G$ are the one-dimensional trivial representation and
the $\frac{p-1}2$ two-dimensional ``rotation'' representations
described above.

\subsection{Cup product and graded commutativity}\label{subs::signs}

The cup product structure on cohomology is a map 
\[\ulhr^*(X;M)\boxp \ulhr^*(X;M)\xrightarrow{\ \smile\ } \ulhr^*(X;M)\]
making $\ulhr^*(X;M)$ into a graded
commutative Green functor.  As with nonequivariant cohomology, this cup
product arises from the diagonal $X\to X\sm X$ for a based $\G$-space $X$.
In \autoref{subs::green}, we deferred a discussion of signs until later,
and so we will discuss them now for $G=\G$.

From a represented point of view, the cup product pairing arises from
the composite
\begin{align*}
& \mapmf{X}{S^\alpha \sm HM}\boxp \mapmf{X}{S^\beta \sm HM} \xrightarrow{\ \sm\ } 
\mapmf{X\sm X}{S^\alpha \sm HM\sm S^\beta \sm HM} \to \\ 
& \to \mapmf{X\sm X}{S^\alpha \sm S^\beta \sm HM \sm HM} \to
\mapmf{X}{S^{\alpha+\beta} \sm HM}.
\end{align*}
The arrow marked $\sm$ comes from taking the smash product of two maps; the
middle arrow comes from switching the order of $HM$ and $S^\beta$; and the
final arrow is induced by the diagonal $X\to X\sm X$, the identification of
$S^\alpha\sm S^\beta$ with $S^{\alpha+\beta}$, and the map $HM\sm
HM\to HM$ arising from the Green functor structure of $M$.

We can then consider the commutativity of the cup product.  Using $\gamma$
to denote the switch map, this amounts to considering the commutativity of
\autoref{fig::cupproduct}, for given maps $f\colon X\to S^\alpha \sm HM$ and
$g\colon X\to S^\beta \sm HM$.
\begin{figure}
\[ \xymatrix@R=2pc@C=0pc{
  & X \ar[dr]^\Delta \ar[dl]_\Delta \ar@{}[d]|\circlearrowleft & \\
  X\sm X \ar[d]_{f\sm g} \ar[rr]^\gamma \ar@{}[drr]|\circlearrowleft & & X\sm X\ar[d]^{g\sm f} \\
  (S^\alpha\sm HM) \sm (S^\beta\sm HM) \ar[d]_{\id\sm\gamma\sm\id} \ar[rr]^\gamma \ar@{}[drr]|\circlearrowleft & &
  (S^\beta\sm HM) \sm (S^\alpha\sm HM) \ar[d]_{\id\sm\gamma\sm\id} \\
  S^\alpha \sm S^\beta \sm HM\sm HM \ar[rr]^{\gamma\sm\gamma} \ar[dr] \ar@{}[drr]|? & &
  S^\beta \sm S^\alpha \sm HM\sm HM \ar[dl] \\
  & S^{\alpha+\beta} \sm HM & \\
} \]
\caption[\lofspace{}$f\smile g$ and $g\smile f$]{$f\smile g$ (the left-hand composite) and $g\smile f$ for the representing maps of elements $f,g\in \ulhrnoph^*(X;M)$.}
\label{fig::cupproduct}
\end{figure}
The three faces marked with $\circlearrowleft$ commute by inspection.
Provided that $M$ is a commutative Green functor, 
\[ \xymatrix@R=1pc@C=.5pc{
  HM\sm HM \ar[rr]^\gamma \ar[dr] & & HM\sm HM \ar[dl] \\
  & HM &
} \]
commutes as well.  We are left with the diagram
\[ \xymatrix@R=1pc@C=.5pc{
  S^\alpha\sm S^\beta \ar[rr]^\gamma \ar[dr]_(0.4)\simeq & & S^\beta\sm S^\alpha \ar[dl]^(0.4)\simeq \\
  & S^{\alpha+\beta} \ar@{}[u]|? &
} \]
In order to identify the vertical equivalences, we need to deal with some
of the technicalities of the $RO(\G)$-grading discussed on page
\pageref{rem::grading}.  If $\{\rotrep_i\}$ are the irreducible
representations of $\G$, we can write
\begin{align*}
\alpha & = \sum_i a_i\rotrep_i, & S^\alpha & = \bigwedge_i S^{a_i\rotrep_i} \\
\beta & = \sum_i b_i\rotrep_i, & S^\beta & = \bigwedge_i S^{b_i\rotrep_i} \\
\alpha+\beta & = \sum_i (a_i+b_i)\rotrep_i, & S^{\alpha+\beta} & = \bigwedge_i S^{(a_i+b_i)\rotrep_i}
\end{align*}
Provided we are working in the stable category, these are well defined for
all integer coefficients $a_i$ and $b_i$.
The maps $S^\alpha\sm S^\beta \xrightarrow{\ \simeq\ } S^{\alpha+\beta}$
and $S^\beta\sm S^\alpha \xrightarrow{\ \simeq\ } S^{\alpha+\beta}$
come from interchanging the factors corresponding to different
representations $\rotrep_i$ and then comparing the maps
$S^{a_i\rotrep_i}\sm S^{b_i\rotrep_i}\to S^{(a_i+b_i)\rotrep_i}$ and
$S^{b_i\rotrep_i}\sm S^{a_i\rotrep_i}\to S^{(a_i+b_i)\rotrep_i}$ for each
$i$. 

We know that, in the stable category,
$\htmf_{(a_i+b_i)\rotrep_i}({S^{(a_i+b_i)\rotrep_i}}) \cong A$, the
Burnside ring Green functor.  If we canonically identify $S^{a_i
\rotrep_i}\sm S^{b_i \rotrep_i}$ with $S^{(a_i +b_i) \rotrep_i}$, the
switch map $S^{a_i \rotrep_i}\sm S^{b_i \rotrep_i}\to S^{b_i \rotrep_i}\sm
S^{a_i \rotrep_i}$ represents some element of $A(\trivo)$, and in
particular some unit of $A(\trivo)$, since it is a homotopy equivalence.

We are considering $G = \G$.  Let $1$ denote the unit of the ring
$A(\trivo) = A(\G)$.  If $p$ is odd, then the only units in
$A(\trivo)$ are $\pm 1$.  Since the underlying nonequivariant degree of
the switch map must be $(-1)^{a_i b_i |\rotrep_i|^2}$, the equivariant
degree is the same.  Explicitly, it is $-1$ when $a_i,b_i$ are both odd and
$\rotrep_i$ is the one-dimensional trivial representation, and $1$
otherwise.\footnote{All other irreducible representations have even
dimension.}  Since the nontrivial irreducible representations of $\G$ all
have trivial fixed-point sets, we may phrase this as follows.

\begin{prop}\label{prop::odddegree}
If $p$ is an odd prime, then the switch map
${S^\alpha\sm S^\beta \to S^\beta \sm S^\alpha}$
has degree $(-1)^{|\alpha^\G|\cdot |\beta^\G|}$.  \hfill$\Box$
\end{prop}

If $p=2$, the signs are less trivial, and we must do some work to identify
them.  In the following, we will ambiguously use $\pi^\G_*$ to denote both
the stable and unstable homotopy groups, relying on context to distinguish
them; note that the homotopy Mackey functor $\htmf_*$ is only defined in
the stable context.

The following theorem of tom Dieck \cite[page 20]{tDstable} will be helpful
in identifying the homotopy class of the switch map.

{ \renewcommand{\G}{G}

\begin{thm}[tom Dieck]\label{prop::ghtpic}
For any group $G$, the map $\pi^\G_0(\Sigma^\infty S^0) \to \prod_{(K)}\bZ$
sending $f\colon S^V\to S^V$ to the tuple $(\deg f^K)$ is injective.
Here $\deg f^K$ is the degree of the restriction $f^K\colon (S^V)^K\to
(S^V)^K$ to the $K$-fixed points, and the product runs over conjugacy
classes of subgroups of $\G$.
\end{thm}


}

{\renewcommand{\G}{{C_2}}

In particular, when $G = \G$, two maps $f,g\colon S^V\to S^V$ are stably
$\G$-homotopic if and only if they have the same nonequivariant degree and
their restrictions $f^\G,g^\G$ to the $\G$-fixed points have the same
degree.  This will be very useful for establishing the equivariant degree
of self-maps of spheres.

For the following proposition, it may be helpful to recall the definition
of the group structure on the equivariant homotopy groups $\pi^\G_*$ for
spaces.  Suppose $V$ is an honest real representation containing at least
two copies of the real one-dimensional trivial representation.  If $|V| =
n+2$, write the elements of $V$ as $(n+2)$-tuples $(x_1,x_2,y_1,\ldots,y_{n})$,
where $\G$ acts trivially on $x_1$ and $x_2$, and view $S^V$ as the
quotient of the unit cube $[-1,1]^{n+2}$ by its boundary in $V$.  Then, given a
$\G$-space $X$ and maps $f,g\colon S^V\to X$, the inverse $-f$ of $f$ is
the continuous map given by the explicit formula
\[ (-f)(x_1,x_2,y_1,\ldots,y_{n}) = f(-x_1,x_2,y_1,\ldots,y_{n}).\]
Similarly, $f+g$ is given by
\[ (f+g)(x_1,x_2,y_1,\ldots,y_{n}) = \begin{cases}
f(2x_1+1,x_2,y_1,\ldots,y_{n}) & x_1 \le 0 \\
g(2x_1-1,x_2,y_1,\ldots,y_{n}) & x_1 \ge 0
\end{cases}
\]
We are ultimately interested in the stable homotopy groups
$\pi^\G_\alpha(\Sigma^\infty X)$ for a virtual representation $\alpha$.  To
get a handle on these, we can choose an honest representation $W$ such that
$\alpha+W$ is an honest representation containing at least two copies of
the trivial representation, and then consider $\pi^\G_{\alpha+W}(\Sigma^W
X)$; the stable homotopy group $\pi^\G_{\alpha}(\Sigma^\infty X)$ is by
definition $\colim_W \pi^\G_{\alpha+W}(\Sigma^W X)$.

As a final bit of notation, again let $1$ be the multiplicative identity of
the ring $A(\trivo)$, and set $\tau = \mft\mfr(1)$.  Then the units of
$A(\trivo)$ are $\pm 1$ and $\pm(1-\tau)$; recall that the Frobenius
relations tell us that $\tau^2 = p\tau = 2\tau$.  Since the switch map is a
representative of a unit in the Burnside ring Green functor, we need only
identify which of these four units it represents.

\begin{prop}\label{prop::eqdegree}
Let $\signrep$ be the one-dimensional real sign representation of
$C_2$.
If $\alpha = a_0 + a_1\signrep$ and $\beta = b_0+b_1\signrep$, then the
switch map $S^\alpha\sm S^\beta\to S^\beta\sm S^\alpha$ represents
$(-1)^{a_0 b_0} (1-\tau)^{a_1 b_1}$ in
$\pi^G_{\alpha+\beta}(\Sigma^\infty S^{\alpha+\beta}) \cong A(\trivo)$.
\end{prop}

Note that we are implicitly using the fact that $\signrep$ and the trivial
representation are the only irreducible representations of $\G$.

\begin{proof}
It suffices to determine the equivariant degree of the switch maps on the
spaces $S^1\sm S^1$ and $S^\signrep\sm S^\signrep$;
\autoref{prop::eqdegree} will then follow.

First consider the switch map on $S^1\sm S^1 = S^2$.  Since the $\G$-action
is trivial, it is immediate that this map $S^2\to S^2$ agrees with $-1 \in
\pi^\G_{2+W}(S^{2+W})$ for every $W$, and hence in the stable group
$\pi^\G_2(\Sigma^\infty S^2)$.

For the switch map on $S^\signrep \sm S^\signrep \simeq S^{2\signrep}$, we
first need explicit descriptions of representatives for $1$ and $\tau$ in
the stable homotopy groups of spaces.  $1$ is represented by the identity
map.  $\tau$ is the stable class given by the restriction associated to
${\freeo\to\trivo}$ followed by the transfer associated to this map.  This
composite corresponds to the map on $S^{2\signrep+W}$ given by first
collapsing to a point everything outside a small neighborhood of two
antipodes not fixed by the action of $\G$, and then taking the equivariant
fold map.

At this point, we can consider the degrees of the units of $A(\trivo)$,
viewing them as self-maps of $S^{2\signrep+W}$ for $W$ containing at least
two copies of the trivial representation.  Since $\tau$ has underlying
nonequivariant degree $2$, $1-\tau$ has nonequivariant degree $-1$ and
$-(1-\tau)$ has nonequivariant degree 1.  Restricted to the fixed-point set
$S^2\subset S^{2+2\signrep}$, $1-\tau$ is the identity, which has degree
$1$.  On the other hand, using the definition of the group operation on
$\pi^\G_*$, we see that $-1$ has degree $-1$ on both the underlying
nonequivariant sphere $S^{2+|W|}$ and the fixed-point sphere $S^{|W^\G|}$.

Finally, consider the stable class of the switch map $S^{2\signrep}\to
S^{2\signrep}$.  This map has underlying nonequivariant degree $-1$, but,
restricted to the fixed-point space of $S^{2\signrep+W}$, is the identity.  
Thus, by \autoref{prop::ghtpic}, the switch map must represent $1-\tau$.
\end{proof}

}

At this point we may return to the diagram of \autoref{fig::cupproduct}.
We have shown that $f\smile g$ and $g\smile f$ differ, as maps $X\to
\Sigma^{\alpha+\beta} HM$, by composition with a switch map on
$S^{\alpha+\beta}$ of known equivariant degree.  This means that there are
known units $u(\alpha,\beta)\in A(\trivo)$ making the diagram below
commute.
\[ \xymatrix@R=1pc@C=0pc{
  \ulhr^\alpha(X;M)\boxp \ulhr^\beta(X;M) \ar[rr]^{u(\alpha,\beta)\gamma} \ar[dr] & & 
  \ulhr^\beta(X;M)\boxp \ulhr^\alpha(X;M) \ar[dl] \\
  & \ulhr^{\alpha+\beta}(X;M) &
} \]
We will thus define a \defword{graded commutative} Green functor to be a
graded Green functor $T^*$ whose multiplication makes all diagrams of the
above form commute, for the same units $u(\alpha,\beta)$ established in
\autoref{prop::odddegree} and \autoref{prop::eqdegree}.

\subsection{Reinterpreting the cup product}\label{subs::levelwisecup}

The cup product structure is given by a map out of a box product, and so we
know from \autoref{lem::pboxmaps} or \autoref{lem::dresspairing} that it
consists of ring maps
\[ {\setlength{\arraycolsep}{1pt}
\begin{array}{ccccc}
\hr^*(X;M)& \otimes & \hr^*(X;M) & \to & \hr^*(X;M) \\
\neqhr^{|*|}(X;M(\freeo))& \otimes & \neqhr^{|*|}(X;M(\freeo)) & \to & \neqhr^{|*|}(X;M(\freeo)) \\
\end{array} }\]
subject to some compatibility conditions.  The second map is exactly the
usual cup product on the nonequivariant cohomology of $X$, so finding the
cup product structure on $\ulhr^*(X;M)$ amounts to finding it on
$\hr^*(X;M)$.  

The units of the Burnside ring Green functor which determine
the graded commutativity of $\ulhr^*(X;M)$ necessarily determine the graded
commutativity of the rings $\hr^*(X;M)$ and
$\neqhr^{|*|}(X;M(\freeo))$, as follows.  If the switch map
$S^\alpha\sm S^\beta\to S^\beta \sm S^\alpha$ has degree
$u(\alpha,\beta)\in A(\trivo)$, then, using the structure of $\ulhr^*(X;M)$
as a module over $A$, we have formulas
\begin{align*}
y x & = u(\alpha,\beta) x y & & x\in \hr^\alpha(X;M),\ y\in \hr^\beta(X;M) \\
z w & = \mfr(u(\alpha,\beta)) w z & & w\in \hr^{|\alpha|}(X;M(\freeo)),\ z\in \hr^{|\beta|}(X;M(\freeo)) \\
\end{align*}
Since $\mfr(-1) = -1\in A(\freeo)$ and, when $p=2$, $\mfr(1-\tau) = -1$ as
well, the second formula is compatible with the usual graded commutativity
of the nonequivariant cup product.

The nonequivariant cup product also determines certain parts of the
multiplicative structure on $\hr^*(X;M) = \ulhr^*(X;M)(\trivo)$.
The Frobenius relations dictate products of the form $x \mft(w)$, for $x\in
\ulhr^*(X;M)(\trivo)$ and $w\in \ulhr^*(X;M)(\freeo)$; the fact that $\mfr$
is a ring homomorphism may provide some additional information as well.

\subsection{Burnside ring coefficients}

We now turn our attention to cohomology with coefficients in the Burnside
Green functor $A$.  Recall that our convention is to suppress $A$
coefficients in the notation.  Since $\ulhr^*(X)$ is a module over
$\ulhr^*(S^0)$ for every $\G$-space $X$, it is desirable to have an
explicit description of $\ulhr^*(S^0)$.  As previously remarked, the
complexities of the $RO(\G)$-graded dimension axiom imply that
$\ulhr^*(S^0)$ is \emph{not} just the coefficient Green functor $A$.

The calculation of $\ulhr^*(S^0)$ is due to Stong when $p=2$, and to
Stong and Lewis for odd primes.  This calculation makes extensive use of
the cofiber sequences
\[ \begin{array}{c}
\G_+ \to S(\rotrep)_+ \to \Sigma \G_+\\
S(\rotrep)_+ \to S^0 \to S^\rotrep\\
{C_2}_+ = S(\signrep)_+ \to S^0 \to S^\signrep
\end{array} \]
where $\rotrep$ is a nontrivial irreducible real representation of $\G$;
$\signrep$ is the nontrivial one-dimensional real representation of $C_2$;
$S(V)$ is the unit sphere in the representation $V$; and $S^V$ is the
one-point compactification of $V$.  This calculation appears in the
appendix to \cite{lgl}, and will not be reproduced here; we will simply
state the results without proof in \autoref{subs::point2} and
\autoref{subs::pointodd}.

\begin{rem} The cohomology of a point for $G=\G$ and any Mackey functor
coefficients is known; see Chapters 8 and 9 of \cite{FL}.  However, we will
only discuss the result for $A$ coefficients in this \paper{}.
\end{rem}

\subsection{The cohomology of a point for $p = 2$}\label{subs::point2}

{\renewcommand{\G}{ {C_2} }
We have already noted that the group $\G$ has exactly one nontrivial
irreducible real representation, namely the one-dimensional sign
representation $\signrep$.  Thus every virtual representation class has a
representative of the form $m+n\signrep$ for $m,n\in \bZ$, where we use $m$
to denote the real $m$-dimensional trivial representation.  If $\alpha =
m+n\signrep$, we can recover $m$ and $n$ from the dimensions $|\alpha| =
m+n$ and $|\alpha^{C_2}| = m$.  For compatibility with $p$ odd, we will
classify our virtual representations $\alpha$ by the ordered pair
$(|\alpha^G|,|\alpha|)$.

\begin{thm}[Stong]\label{thm::add2structure}
When $G=\G$, the additive structure of $\ulhr^*(S^0)$ is as follows.
\[ \ulhr^\alpha (S^0) \cong \begin{cases}
A & \mydim\alpha = (0,0) \\
\r & \mydim\alpha = (-2m,0) \text{ for $m\ge 1$} \\
\rminus & \mydim\alpha = (1-2m,0) \text{ for $m\ge 0$} \\
\l & \mydim\alpha = (2m,0) \text{ for $m\ge 1$} \\
\lminus & \mydim\alpha = (2m+1,0) \text{ for $m\ge 1$} \\
\jg{\ground} & \mydim\alpha = (0,n) \text{ for $n\in\bZ$} \\
\jg{\ground/2} & \mydim\alpha = (-2m,n) \text{ for $m\ge 1$, $n\ge 1$} \\
\jg{\ground/2} & \mydim\alpha = (2m+1,-n) \text{ for $m\ge 1$, $n\ge 1$} \\
0 & \text{otherwise}
\end{cases} \]
\end{thm}

\begin{proof}
See the appendix to \cite{lgl}.
\end{proof}

This is perhaps best visualized using \autoref{fig::mod2}.

\begin{figure}
\[ \hspace{-2ex}\xymatrix@R=0pc@C=0pc{
  &  &  &  &  &  &  & \makebox[0pt]{$|\alpha|$} &  &  &  &  &  &  \\
  &  & {\phantom{\jg{\ground/2}}} & {\phantom{\jg{\ground/2}}} & {\phantom{\jg{\ground/2}}} & {\phantom{\jg{\ground/2}}} & {\phantom{\jg{\ground/2}}} & {\phantom{\jg{\ground/2}}} & {\phantom{\jg{\ground/2}}} & {\phantom{\jg{\ground/2}}} & {\phantom{\jg{\ground/2}}} & {\phantom{\jg{\ground/2}}} & {\phantom{\jg{\ground/2}}} &  \\
  &  & . & \jg{\ground/2} & . & \jg{\ground/2} & . & \jg{\ground} & . & . & . & . & . &   \\
  &  & . & \jg{\ground/2} & . & \jg{\ground/2} & . & \jg{\ground} & . & . & . & . & . &   \\
  &  & . & \jg{\ground/2} & . & \jg{\ground/2} & . & \jg{\ground} & . & . & . & . & . &   \\
  &  & . & \jg{\ground/2} & . & \jg{\ground/2} & . & \jg{\ground} & . & . & . & . & . &   \\
  \ar@{<.>}[rrrrrrrrrrrrrr] & &  \rminus & \r_{\phantom{-}} & \rminus & \r_{\phantom{-}} & \rminus & A_{\phantom{-}} & \rminus & \l_{\phantom{-}} & \lminus & \l_{\phantom{-}} & \lminus &   &  \makebox[0pt][l]{$|\alpha^\G|$} \\
  &  & . & . & . & . & . & \jg{\ground} & . & . & \jg{\ground/2} & . & \jg{\ground/2} &   \\
  &  & . & . & . & . & . & \jg{\ground} & . & . & \jg{\ground/2} & . & \jg{\ground/2} &   \\
  &  & . & . & . & . & . & \jg{\ground} & . & . & \jg{\ground/2} & . & \jg{\ground/2} &   \\
  &  & . & . & . & . & . & \jg{\ground} & . & . & \jg{\ground/2} & . & \jg{\ground/2} &   \\
  &  &   &   &   &   &   &   &   &   &   &    &   & \\
  &  &   &   &   &   &   & \ar@{<.>}[uuuuuuuuuuuu]  &   &   &   &   &   &
}\]
\caption[\lofspace{}A plot of ${\ulhrcapt^*(S^0)}$ in the plane]{A plot of ${\ulhrnoph^*(S^0)}$ in the plane.  A dot represents the zero Mackey functor.}
\label{fig::mod2}
\end{figure}

As outlined in \autoref{subs::levelwisecup}, we will use the Dress pairing
to describe the maps ${\ulhr^\alpha(S^0)\boxp \ulhr^\beta(S^0) \to
\ulhr^{\alpha+\beta}}$ making up the cup product structure on
$\ulhr^*(S^0)$.
This provides a very concrete picture of the cup product structure,
although it has the disadvantage that it is difficult to tell from the
Dress pairing when a map is an isomorphism.\footnote{For most
$\alpha,\beta\in RO(\G)$, $\ulhr^{\alpha+\beta}(S^0)$ is either zero or
isomorphic to ${\ulhr^\alpha(S^0)\boxp \ulhr^\beta(S^0)}$.  When the latter
is true, the cup product map ${\ulhr^\alpha(S^0)\boxp \ulhr^\beta(S^0) \to
\ulhr^{\alpha+\beta}(S^0)}$ is always an isomorphism.}
We will begin by defining generators in the $\ground$-modules
$\ulhr^*(S^0)(\trivo)$ and $\ulhr^*(S^0)(\freeo)$, and then discuss their
products.

Consider first the Mackey functors which appear in $\ulhr^*(S^0)$: $A$,
$\l$, $\lminus$, $\r$, $\rminus$, and $\jg{C}$.  Once we
choose an element $\mu\in A(\trivo)$ which restricts to a generator of
$A(\freeo) \cong \ground$, $\{\mu,\mft\mfr\mu\}$ gives a basis for $A(\trivo)$,
and $\mfr \mu$ is a generator of $A(\freeo)$.  Likewise, if a Mackey
functor $M$ is equal to $\l$, $\lminus$, or $\rminus$, choosing a
generator $\iota$ of $M(\freeo)$ gives us the generator $\mft\iota$ of
$M(\trivo)$.  Finally, if $M = \r$ or $M = \jg{\ground}$ or $M = \jg{\ground/2}$,
then choosing a generator $\xi$ for $M(\trivo)$ gives the generator
$\mfr\xi$ of $M(\freeo)$.

Recall that we use $\signrep$ to denote the real one-dimensional sign
representation of $\G$.  For ease of reference with \autoref{fig::mod2},
note that the trivial representation has dimension $(1,1)$ and $\signrep$
has dimension $(0,1)$; so the virtual representation $1-\signrep$ has
dimension $(1,0)$.

\begin{defn}
Define generators $1$, $\iota$, $\iota^{-1}$, $\epsilon$, and $\xi$ for
certain of the Mackey functors $\ulhr^\alpha(S^0)$ as follows.
\begin{enumerate}
\item Let $1\in \ulhr^0(S^0)(\trivo)$ be the image of the identity in
$A(\trivo)$ under the unit map $A\to \ulhr^*(S^0)$ of the graded Green
functor $\ulhr^*(S^0)$.

\item The underlying nonequivariant space of $S^\signrep$ is $S^1$.  Recall
that $\ulhr^\alpha(X)(\freeo) \cong \neqhr^{|\alpha|}(X;\ground)$, the
nonequivariant cohomology.  As a result, fixing a nonequivariant
identification of $S^\signrep$ with $S^1$ induces isomorphisms
\[\ulhr^0(S^0)(\freeo) \cong \ulhr^1(S^1)(\freeo) \to
\ulhr^1(S^\signrep)(\freeo) \cong \ulhr^{1-\signrep}(S^0)(\freeo)\]
and
\[\ulhr^0(S^0)(\freeo) \cong \ulhr^\signrep(S^\signrep)(\freeo) \to
\ulhr^\signrep(S^1)(\freeo) \cong \ulhr^{\signrep-1}(S^0)(\freeo).\]
Suggestively, we give the names $\iota \in\ulhr^{1-\signrep}(S^0)(\freeo)$
and $\iota^{-1}\in \ulhr^{\signrep-1}(S^0)(\freeo)$ to the images of $\mfr
1$ under these composites.

\item Using the suspension isomorphism, the inclusion $S^0\hookrightarrow
S^\signrep$ induces a map
\[ {A\cong \ulhr^0(S^0) \cong \ulhr^\signrep(S^\signrep)
\to \ulhr^\signrep(S^0)\cong \jg{\ground}}.\]
Let $\epsilon$ be the image of $1\in \ulhr^0(S^0)(\trivo)$ under this map.

\item Since $\mfr$ in $\r$ is an isomorphism, there is a unique element
$\xi$ in $\ulhr^{2\signrep-2}(S^0)(\trivo)$ whose image under $\mfr$ is
$(\iota^{-1})^2$, i.e.\ the image of $\iota^{-1}\otimes\iota^{-1}$ under the
appropriate Dress pairing.
\end{enumerate}
\end{defn}

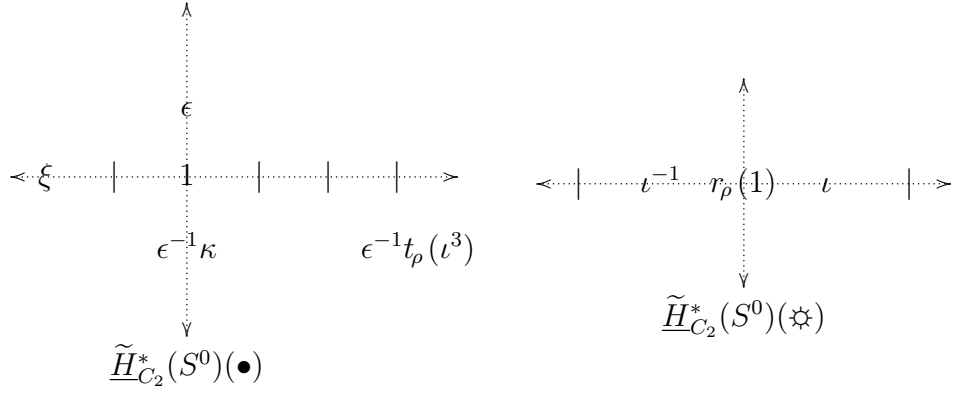
\begin{figure}
\[ 
\hspace{-20pt}
\xymatrix@R=.75pc@C=0pc{
  &   &   & \ar@{<.>}[dddddd] &   &   &   & \\
{\phantom{\epsilon^{-1}}}  &   &   &   &   &   &   & \\
{\phantom{\epsilon^{-1}}}  &   &   & \epsilon &   &   &   & \\
{\phantom{\epsilon^{-1}}}\ar@{<.>}[rrrrrrr]  & \xi & | & 1 & | & | & | & \\
{\phantom{\epsilon^{-1}}}  &   &   & \epsilon^{-1}\kappa &   &   &
{\phantom{\epsilon^{-1}}}  \makebox[0pt]{$\epsilon^{-1}\mft(\iota^3)$} & \\
  & \hspace{20pt} & \hspace{20pt} & \hspace{20pt} & \hspace{20pt} & \hspace{20pt} & \hspace{20pt} & \hspace{20pt} & \hspace{20pt} & \\
  &   &   & \makebox[0pt]{$\ulhr^*(S^0)(\trivo)$}  &   &   &   &   & \\
}
\hspace{-40pt}
\xymatrix@R=0.57pc@C=0pc{
{\phantom{\epsilon^{-1}}}  &   &   &   &   &   & \\
  & {\phantom{\mfr(1)}} & {\phantom{\mfr(1)}} & {\phantom{\mfr(1)}} \ar@{<.>}[dddd] & {\phantom{\mfr(1)}} & {\phantom{\mfr(1)}} & \\
{\phantom{\epsilon^{-1}}}  &   &   &   &   &   & \\
{\phantom{\epsilon^{-1}}}\ar@{<.>}[rrrrrr] & | & \iota^{-1} & \mfr(1)^{\makebox[0pt]{$\phantom{-1}$}} & \iota^{\makebox[0pt]{$\phantom{-1}$}} & | & \\
{\phantom{\epsilon^{-1}}}  &   &   &   &   &   & \\
  &   &   & \makebox[0pt]{$\ulhr^*(S^0)(\freeo)$}  &   &   & \\
} 
\] 
\caption[\lofspace{}The locations of the generators $\epsilon$, $\iota$, $\iota^{-1}$,
$\xi$, $\epsilon^{-1}\kappa$, and
$\epsilon^{-1}\mft(\iota^3)$]{The locations of the generators $\epsilon$, $\iota$, $\iota^{-1}$,
$\xi$, $\epsilon^{-1}\kappa$, and
$\epsilon^{-1}\mft(\iota^3)$.}
\label{fig::gens2}
\end{figure}

For the final classes of generators, we need some lemmas.  We have chosen the
basis $\{1,\mft\mfr 1\}$ of $\ulhr^0(S^0)(\trivo)$.  However, employing the
same trick as in the proof of \autoref{lem::twiso}, if we define $\kappa =
2-\mft\mfr 1$, then $\{1,\kappa\}$ is also a basis.  It has the advantage
that $\kappa$ generates the kernel of $\mfr$.

\begin{lem}[Lewis]
For each $m\ge 1$, there is a unique element $\epsilon^{-m}\kappa \in
\ulhr^{-m\signrep}(S^0)(\trivo)$ such that, under the Dress pairing maps,
$\epsilon^m (\epsilon^{-m}\kappa) = \kappa$.  Moreover,
$\epsilon^{-m}\kappa$ is a generator of $\ulhr^{-m\signrep}(S^0)(\trivo)$
for each $m$.
\end{lem}

Similarly, the elements $\mft (\iota^{2n+1})$ for $n\ge 1$ are divisible by
$\epsilon$.

\begin{lem}[Lewis]
For each $m\ge 1$ and $n\ge 1$, there is a unique element
$\epsilon^{-m}\mft(\iota^{2n+1})$ in dimension $(2n+1,-m)$ such that, under
the Dress pairing maps,
\[ \epsilon^m (\epsilon^{-m}\mft(\iota^{2n+1})) = \mft(\iota^{2n+1}).\]
Moreover,  each $\epsilon^{-m}\mft(\iota^{2n+1})$ is a generator of
$\ulhr^{(2n+1)(1-\signrep)-m\signrep}(S^0)(\trivo)$.
\end{lem}

When dealing with the multiplicative structure of $\ulhr^*(S^0)$, it is
often helpful to recall the Frobenius relations, i.e.\ the equations given in
\autoref{lem::pboxmaps} and coming from the commuting diagrams in
\autoref{lem::dresspairing}.  For example, looking at the levelwise ring
structure of $A = \ulhr^0(S^0)$, these tell us that
\[(\mft \mfr 1)^2 = \mft( (\mfr 1) (\mfr \mft \mfr 1) ) = 2(\mft\mfr 1),\]
using $G=\G$.  It follows that $\kappa^2 = 2\kappa$ as well.

\begin{thm}[Lewis]\label{thm::mult2structure}
As an algebra over $A(\freeo) = \ground$, $\ulhr^*(S^0)(\freeo) \cong
\ground[\iota,\iota^{-1}]$, the polynomial algebra on generators $\iota$ in
dimension $(1,0)$ and $\iota^{-1}$ in dimension $(-1,0)$, subject to the
relation $\iota \cdot \iota^{-1} = \mfr(1)$, the unit of $\ulhr^*(S^0)$.

As an algebra over $A(\trivo) = A(\G)$, the classical Burnside ring,
$\ulhr^*(S^0)(\trivo)$ is graded commutative, as discussed in
\autoref{prop::eqdegree} and \autoref{subs::levelwisecup}.  It is generated by the elements 
\begin{align*}
\epsilon &\in \ulhr^\signrep(S^0)(\trivo) \cong \jg{\ground}(\trivo) \\
\xi &\in \ulhr^{2\signrep-2}(S^0)(\trivo) \cong \r(\trivo) \\
\epsilon^{-m}\kappa & \in \ulhr^{-m\signrep}(S^0)(\trivo) \cong \jg{\ground}(\trivo) & & \text{ for each }m\ge 1 \\
\mft (\iota^m) &\in \ulhr^{m(1-\signrep)}(S^0)(\trivo) & & \text{ for each } m\ge 2 \\
\epsilon^{-m}\mft(\iota^{2n+1}) &\in \ulhr^{(2n+1)(1-\signrep)-m\signrep}(S^0)(\trivo) \cong \jg{\ground/2}(\trivo) & & \text{ for each }m,n\ge 1\\
\end{align*}
subject to the following relations.  First, the underlying graded
$A(\trivo)$-module of $\ulhr^*(S^0)(\trivo)$ is as described in
\autoref{thm::add2structure}; this forces the vanishing of some products.
Further, the Frobenius relation ${\mft(x)\cdot y = \mft(x\cdot \mfr(y))}$
dictates relations involving products with $\mft(\mfr (1))$ and the
generators $\mft(\iota^m)$.  Finally, we have the following relations:
\begin{align*}
\epsilon (\epsilon^{-1} \kappa) & = \kappa \\
\epsilon (\epsilon^{-1} \mft(\iota^{2n+1})) & = \mft(\iota^{2n+1}) \\
\end{align*}
The generators of $\ulhr^*(S^0)(\trivo)$ and $\ulhr^*(S^0)(\freeo)$ are
related by $\mfr \xi = \iota^{-2}$.
\end{thm}

\begin{proof}
See the appendix to \cite{lgl}.
\end{proof}

\begin{rem}
We know $\ulhr^*(S^0)$ is graded commutative.  Using the notation of
\autoref{prop::eqdegree}, $(1-\tau)$ acts trivially on each
$\ground$-module $\ulhr^{m\signrep}(S^0)(\trivo) \cong \jg{\ground}$, so
\[ \epsilon^{m_1+m_2} (\epsilon^{-m_1}\kappa)(\epsilon^{-m_2}\kappa) = \epsilon^{m_1}(\epsilon^{-m_1}\kappa) \epsilon^{m_2}(\epsilon^{-m_2}\kappa) = \kappa^2 = 2\kappa\]
and thus
\[ (\epsilon^{-m_1}\kappa)(\epsilon^{-m_2}\kappa) = 2\epsilon^{-(m_1+m_2)}\kappa.\]
Similar reasoning, combined with the Frobenius relations, shows that
\[ (\epsilon^{-m_1}\mft(\iota^{2n+1}))(\epsilon^{-m_2}\kappa) = 0 \]
and 
\[ \xi (\epsilon^{-m}\mft(\iota^{2n+1})) = \epsilon^{-m} \mft(\iota^{2n-1})
\text{ for } m\ge 0, n\ge 2.\]
\end{rem}

Instead of describing the underlying additive structure of $\ulhr^*(S^0)$
and using this to simplify the description of the multiplicative structure,
we could alternatively present $\ulhr^*(S^0)$ solely in terms of generators
and relations by explicitly listing all relations forced by the vanishing
of a target, e.g.\ $\mft(\iota) = 0$; $2\xi\epsilon = 0$; and, using the
Frobenius relations, 
\[ \mft(\iota^{m_1}) \mft(\iota^{m_2}) = \begin{cases}
0 & \text{ if $m_1$ or $m_2$ is odd;} \\
2\mft(\iota^{m_1+m_2}) & \text{ if both are even.}
\end{cases} \]
See \cite{lgl}, Theorem 4.3, for a complete presentation from this point of
view.

}

\subsection{The cohomology of a point for $p > 2$}\label{subs::pointodd}

As we did in \autoref{subs::point2}, we will give the additive and
multiplicative structure of $\ulhr^*(S^0)$ without proof.  Proofs for all
claims in this subsection can be found in \cite{lgl}.

Recall from the beginning of this section that $\G$ for $p>2$ has
$\frac{p+1}2$ irreducible representations: the one-dimensional trivial
representation, and $\frac{p-1}2$ two-dimensional representations
$\rotrep_1,\ldots,\rotrep_{(p-1)/2}$ on which $\G$ acts by rotation.  For
concreteness, we make the following definition.

\begin{defn}
Fix a generator $g$ of $\G$.  For each $1\le k \le \frac{p-1}2$, let
$\rotrep_k$ be the underlying real representation of the complex
one-dimensional representation given by $gz = e^{2\pi i k/p} z$.
\end{defn}

It turns out that the isomorphism type of the Mackey functor
$\ulhr^\alpha(S^0)$ depends only on the ordered pair of dimensions
$(|\alpha^\G|,|\alpha|)$ and a constant $d(\alpha)\in \fp^\times/\{\pm
1\}$, defined below.  In fact, if $|\alpha|$ and $|\alpha^\G|$ are not both
zero, the isomorphism type depends only on the ordered pair
$(|\alpha^\G|,|\alpha|)$.  As a result, we will continue to display
$\ulhr^*(S^0)$ pictorially on a two-dimensional grid with axes
corresponding to $|\alpha^\G|$ and $|\alpha|$.  Since the irreducible
representations have dimensions $(1,1)$ and $(0,2)$, the difference between
$|\alpha^G|$ and $|\alpha|$ is even for every virtual representation
$\alpha$ of $\G$; thus not all ordered pairs give the dimensions of a virtual
representation.

When $\alpha$ has dimension $(0,0)$, $\ulhr^*(S^0) \cong
\tw{A}{d(\alpha)}$ for the constant $d(\alpha)$ alluded to above.
Recall from \autoref{lem::twequiv} that, for integers $d$, the isomorphism
type of $\tw{A}{d}$ depends only on the image of $d$ in
$\fp^\times/\{\pm 1\}$\footnote{If $\ground$ has more units than just $\pm
1$, it may in fact depend only on a quotient of this.}; so it makes sense
to talk about $\tw{A}{d}$ when $d\in  \fp^\times/\{\pm 1\}$.  We define the
function $d:RO(G)\to \fp^\times/\{\pm 1\}$ as follows.

\begin{defn}\label{def::reald}
Suppose that $\alpha = a_0 + \sum_{j} a_j \rotrep_j$, for $a_0,a_j\in\bZ$.
Then we set
\[ d(\alpha) = 1^{a_1} 2^{a_2}\cdots \left( \frac{p-1}2 \right)^{a_{(p-1)/2}}.\]
\end{defn}

This is the function which appears in \autoref{thm::addpstructure}.

\begin{thm}[Lewis, Stong]\label{thm::addpstructure}
When $G=\G$, the additive structure of $\ulhr^*(S^0)$ is as follows.
\[ \ulhr^\alpha (S^0) \cong \begin{cases}
\tw{A}{d(\alpha)} & \mydim\alpha = (0,0) \\
\r & \mydim\alpha = (-2m,0) \text{ for $m\ge 1$} \\
\l & \mydim\alpha = (2m,0) \text{ for $m\ge 1$} \\
\jg{\ground} & \mydim\alpha = (0,2n) \text{ for $n\in\bZ$, $n\ne 0$} \\
\jg{\ground/p} & \mydim\alpha = (-2m,2n) \text{ for $m\ge 1$, $n\ge 1$} \\
\jg{\ground/p} & \mydim\alpha = (2m+1,-2n+1) \text{ for $m\ge 1$, $n\ge 1$} \\
0 & \text{otherwise}
\end{cases} \]
\end{thm}

\begin{proof}
See the appendix to \cite{lgl}.
\end{proof}

This is depicted in \autoref{fig::modp}.

\begin{figure}
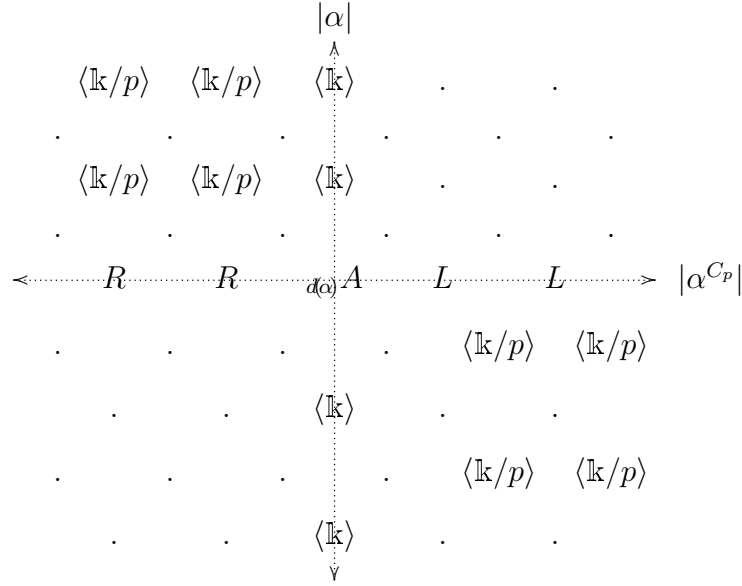

\[ \hspace{-2ex}
\cohomplot{|\alpha^\G|}{|\alpha|}{\tw{A}{d\!(\!\alpha\!)}}{\r}{\l}{\jg{\ground}}{\jg{\ground/p}}{}{1pc}
\]
\caption[\lofspace{}A plot of $\ulhrcapt^*(S^0)$ in the plane]{A plot of $\ulhrnoph^*(S^0)$ in the plane.  A
dot represents the zero Mackey functor.  A blank indicates that no virtual
representation has the given dimension.}
\label{fig::modp}
\end{figure}

We will finish this section by describing the cup product structure on
$\ulhr^*(S^0)$.   As with $p=2$, our approach will be to present
$\ulhr^*(S^0)(\trivo)$ and $\ulhr^*(S^0)(\freeo)$ in terms of generators
and relations.  Unfortunately, the presence of the twists $d(\alpha)$ for
$p>2$ significantly complicates this approach: choosing a generator
$\mu_\alpha$ of $\tw{A}{d(\alpha)}(\trivo)$ means choosing a $d(\alpha)\in
\ground$ rather than in $\fp^\times/\{\pm 1\}$, because $\mfr(\mu_\alpha)$
must be $d(\alpha)$ times the generator of $\tw{A}{d(\alpha)}(\freeo)$, and
in general there is no homomorphism $RO(\G)\to \ground$.  As a result, we
must choose a noncanonical set map $RO(\G)\to \ground$, and $\mu_\alpha
\mu_\beta$ will wind up being $\mu_{\alpha+\beta}$ plus an error term.
This complication will be compounded when we start thinking about algebras
over $\ulhr^*(S^0)$, e.g.\ $\ulhr^*(B_\G SO(2)) \cong
\ulhr^*(\cpv{\cxuniverse})$, the complex projective space on a complete
universe $\cxuniverse$\footnote{The $B_\G SO(2)$ indicates that this space
classifies real oriented two-plane bundles; see
\autoref{sec::eqclassifying}.}.

\begin{defn}\label{def::integerd}
Extend the $d$ of \autoref{def::reald} to a set map ${d\colon RO(\G)\to
\bZ}$ by reinterpreting the definition
\[ d(\alpha) = 1^{a_1} 2^{a_2}\cdots \left( \frac{p-1}2 \right)^{a_{(p-1)/2}}\]
as follows.  If $a_j \ge 0$, then
$j^{a_j}$ is as expected.  Otherwise, let $\text{inv}(j)$ be the smallest
positive integer such that $j\cdot \text{inv}(j) \equiv 1\pmod{p}$, and
interpret $j^{a_j}$ to mean $(\text{inv}(j))^{-a_j}\in\ground$.

We can similarly define a map of sets ${d\colon RO(\G)\to\ground}$ by using
$(j^{-1})^{-a_j}$ when $j$ is invertible in $\ground$, and
$(\text{inv}(j))^{-a_j}$ when it is not.
\end{defn}

We will usually rely on context to distinguish whether $d(\alpha)$ refers
to an element of $\fp$, an integer, or an element of $\ground$.

We are now in a position to define explicit generators of
$\ulhr^\alpha(S^0)$ for relevant $\alpha$.  The notation is intended to be
parallel to the notation given for $p=2$.  Note that the trivial
representation has dimension $(1,1)$, and each $\rotrep_j$ has dimension
$(0,2)$.  No virtual representation has dimension $(1,0)$, but
each $2-\rotrep_j$ has dimension $(2,0)$.

\begin{defn}\label{def::pgenerators}
For a generic $\alpha\in RO(\G)$, write $\alpha = a_0 +
\sum_{j=1}^{(p-1)/2} a_j\rotrep_j$.  Define generators $1$, $\mu_\alpha$,
$\iota_\alpha$, $\epsilon_\alpha$, and $\xi_\alpha$ in some of the Mackey
functors $\ulhr^\alpha(S^0)$ as follows.
\begin{enumerate}
\item Let $1\in \ulhr^0(S^0)(\trivo)$ be the image of $1\in A(\trivo)$
under the unit map $A\to \ulhr^*(S^0)$.  This is also the image of the
identity map $S^0\to S^0$ under the Hurewicz map.  For consistency with
\autoref{def::mualpha}, we write $\mu_0 = 1$.

\item \label{def::mualpha} If $\alpha\ne 0$ has dimension $(0,0)$, let
\[ W = \sum_{j \text{ s.t. } a_j<0} (-a_j)\rotrep_j.\]
Then $\alpha+W$ and $W$ are honest representations with $|\alpha+W| = |W| =
2m$ for some $m>0$, so each is the sum of $m$ nontrivial irreducible
representations.  Using our ordering of the $\rotrep_j$, we can then write
$\alpha = \sum_{j=1}^m \alpha_j$, where each $\alpha_j$ is the difference
of two distinct irreducible representations; that is, $\alpha_j =
\rotrep_{k_j} - \rotrep_{\ell_j}$ for distinct integers $1\le k_j,\ell_j
\le \frac{p-1}2$.  Then our explicit identification of the $\rotrep_j$
implies that the map $\bC^m\to\bC^m$ sending
\[ {(z_1,\ldots,z_m)\mapsto (z_1^{d(\alpha_1)},\ldots,z_m^{d(\alpha_m)})}\]
extends to one-point compactifications to give an equivariant map $S^W\to
S^{\alpha+W}$.  The image of this under the Hurewicz map is an element of
\[\ulhlr_{-\alpha}(S^0)(\trivo)\cong \ulhr^\alpha(S^0)(\trivo),\] which we
call $\mu_\alpha$. 

\item If $\alpha$ has dimension $(2n,0)$ for some integer $n$, then
$|\alpha| = 0$ and we can define $W$ as in \autoref{def::mualpha}.  The only
difference is that $\alpha$ may now have a nonzero number of trivial
representations, so $d(\alpha)$ no longer provides a way of constructing an
equivariant map $S^W\to S^{\alpha+W}$.  However, we can use our ordering of
the irreducible representations to produce a canonical nonequivariant
identification $S^W\to S^{\alpha+W}$.  This induces an isomorphism at the
$\freeo$ level
\[ \ulhr^0(S^0)(\freeo) \cong \ulhr^\alpha(S^\alpha)(\freeo)
\xrightarrow{\ \cong\ } \ulhr^\alpha(S^0).\]
Let $\iota_\alpha$ be the image of $\mfr(1)$ under this isomorphism.

\item If $\alpha$ has dimension $(0,2m)$ for some $m>0$, then we can write
$\alpha$ non-uniquely as $\alpha_0 + V$ for some honest representation $V$
of dimension $(0,2m)$ and $\alpha_0\in RO(\G)$ of dimension $(0,0)$.  The
inclusion $S^0\hookrightarrow S^V$ then induces a map on cohomology
${\ulhr^{\alpha_0}(S^0)\to \ulhr^\alpha(S^0)}$.  We define
$\epsilon_\alpha$ to be the image of $\mu_{\alpha_0}$ under this map.  It
is proved in \cite[Lemma A.11]{lgl} that $\epsilon_\alpha$ is independent
of the choice of $\alpha_0$ and $V$.

\item If $\alpha$ has dimension $(-2n,0)$ for $n>0$, then
$\ulhr^\alpha(S^0)\cong \r$, for which $\mfr$ is an isomorphism.  Let
$\xi_\alpha$ be $\mfr^{-1}(\iota_\alpha)$, so $\mfr(\xi_\alpha) =
\iota_\alpha$.  

\item Select any generator of $\ulhr^{3-2\rotrep_1}(S^0)(\trivo)$ and call
it $\nu_{3-2\rotrep_1}$.
\end{enumerate}
\end{defn}

For $p=2$, each copy of $\jg{\ground/2}$ in the
fourth quadrant had a unique generator at the $\trivo$ level.  Further, for
each such generator, multiplying by a power of $\epsilon$ gave a generator
of a copy of $\lminus(\trivo)$, which in turn was the image of a power of
$\iota$ under $\mft$.  Sadly, neither of these applies when $p>2$, since
there are no virtual representations of dimension $(2n+1,0)$.  It is
therefore necessary to make a noncanonical choice of generator in at least
one of the fourth-quadrant torsion groups; we have chosen
$\nu_{3-2\rotrep_1}$.  Fortunately, this one noncanonical choice is
sufficient.

\begin{prop}
Suppose $x\in \ulhr^{\alpha}(S^0)(\trivo)$ for some $\alpha$ in the fourth
quadrant, so $|\alpha| < 0$ and $|\alpha^G| > 0$.  Let $n,m\ge 0$.  For
every $\beta\in RO(\G)$ of dimension $(0,0)$, $\gamma$ of dimension
$(0,2m)$, and $\delta$ of dimension $(-2n,0)$, $x$ is uniquely divisible by
$\mu_\beta$, $\epsilon_\gamma$, and $\xi_\delta$.  That is, there are
elements $\mu_\beta^{-1}x$, $\epsilon_\gamma^{-1}x$, and $\xi_\delta^{-1}x$
with the expected properties.
\end{prop}

Also, $\mu_\alpha$ plays the desired role in $\ulhr^\alpha(S^0)$.

\begin{prop}
If $\alpha$ has dimension $(0,0)$, then $\mfr(\mu_\alpha) =
d(\alpha)\iota_\alpha$ in $\ulhr^\alpha(S^0)\cong \tw{A}{d(\alpha)}$.
\end{prop}

\begin{proof}
The map 
${(z_1,\ldots,z_m)\mapsto (z_1^{d(\alpha_1)},\ldots,z_m^{d(\alpha_m)})}$
of \autoref{def::mualpha} has nonequivariant degree
\[\prod_{j=1}^m d(\alpha_j) = d(\alpha),\]
by definition.  Since $\iota_\alpha$ come from the map
$(z_1,\ldots,z_m)\mapsto (z_1,\ldots,z_m)$, there is no sign introduced;
hence $\mfr(\mu_\alpha) = d(\alpha)\iota_\alpha$.
\end{proof}

If $\alpha$ has dimension $(0,0)$, \autoref{def::pgenerators} implicitly defines a basis
$\{\mu_\alpha,\mft\iota_\alpha\}$ for each $\ulhr^\alpha(S^0)(\trivo)$.
As we did for $p=2$, we will also define an alternative basis.  Write $\alpha = a_0 +
\sum_j a_j \rotrep_j$ as in \autoref{def::integerd}.  Then, by definition,
\[ d(-\alpha) d(\alpha) = \prod_j (j\cdot \text{inv}(j))^{a_j}.\]
By the definition of $\text{inv}(j)$, this product is congruent to 1 modulo
p, and thus there is an integer $b_\alpha$ such that $d(-\alpha) d(\alpha)
+ b_\alpha p = 1$.

\begin{defn}
Suppose $\alpha$ has dimension $(0,0)$, and set
$b_\alpha = \frac{1 - d(-\alpha) d(\alpha)}p$.
Then let $\{\kappa_\alpha,\sigma_\alpha\}$ be the basis of
$\tw{A}{d(\alpha)}$ given by
\begin{align*}
\kappa_\alpha & = p\mu_\alpha -d(\alpha) \mft(\iota_\alpha) \\
\sigma_\alpha & = d(-\alpha) \mu_\alpha + b_\alpha \mft(\iota_\alpha). 
\end{align*}
\end{defn}

As with $p=2$, the benefit is that $\mfr(\kappa_\alpha) = 0$ and
$\mfr(\sigma_\alpha) = \iota_\alpha$; with this basis
$\{\kappa_\alpha,\sigma_\alpha\}$, $\tw{A}{d(\alpha)}$ has diagram
\[ \mf{\ground\oplus\ground}{\ground}{\mymatrix{d(-\alpha) & p}}{\mymatrix{0\\1}}{\text{triv}} \]
We can also use $\kappa_\alpha$ to
define generators of the Mackey functors $\jg{\ground}$ in dimensions
$(0,-2m)$ for $m>0$.

\begin{prop}
If $\alpha$ has dimension $(0,0)$ and $\beta$ has dimension $(0,2m)$ for
$m>0$, then $\kappa_\alpha$ is divisible by $\epsilon_\beta$:  that is,
there is a unique element $\epsilon_\beta^{-1}\kappa_\alpha \in
\ulhr^{\alpha-\beta}(S^0)(\trivo)$ such that $\epsilon_\beta
(\epsilon_\beta^{-1}\kappa_\alpha) = \kappa_\alpha$.
\end{prop}

Putting all of these definitions and propositions together, we can get a
picture of the multiplicative structure of $\ulhr^*(S^0)$.  For the most
part, the structure is analogous to the simpler structure of
\autoref{thm::mult2structure}: there are now multiple generators
$\iota_\alpha$, but $\iota_\alpha \iota_\beta = \iota_{\alpha+\beta}$.
Similarly, $\epsilon_\alpha \epsilon_\beta =
\epsilon_{\alpha+\beta}$, $\mu_\alpha \epsilon_\beta =
\epsilon_{\alpha+\beta}$,  and $\xi_\alpha \xi_\beta =
\xi_{\alpha+\beta}$.  However, $\mu_\alpha \xi_\beta = d(\alpha)
\xi_{\alpha+\beta}$; and the noncanonical choices of the $\mu_\alpha$ for
$\alpha$ of dimension $(0,0)$ lead to some complications.  For example,
$\mfr(\mu_\alpha \mu_\beta) = d(\alpha) d(\beta) \iota_{\alpha+\beta}$, but
$\mfr(\mu_{\alpha+\beta}) = d({\alpha+\beta}) \iota_{\alpha+\beta}$, and in
general ${d({\alpha+\beta})\ne d(\alpha) d(\beta)}$. 

\begin{thm}[Lewis]\label{thm::multpstructure}
As an algebra over $A(\freeo) \cong \ground$,
\[ \ulhr^*(S^0)(\freeo) \cong
\ground[\iota_{\pm(\rotrep_2-\rotrep_1)},\ldots,\iota_{\pm(\rotrep_{(p-1)/2}-\rotrep_1)},\iota_{\pm(2-\rotrep_1)}]/\sim \]
where the relation $\sim$ is given by $\iota_{-\alpha}\iota_{\alpha} =
\mfr(1)$ for each $\iota_\alpha$ appearing in the list of generators.

As an algebra over $A(\trivo) = A(\G)$, the classical Burnside ring,
$\ulhr^*(S^0)(\trivo)$ is graded commutative, as discussed in
\autoref{prop::odddegree} and \autoref{subs::levelwisecup}.  It is generated
by the elements
\begin{align*}
\mu_\alpha &\in \ulhr^{\alpha}(S^0)(\trivo) \cong \tw{A}{d(\alpha)}(\trivo) & &
\text{for } \alpha = \pm(\rotrep_j-\rotrep_1),\ 2\le j\le \frac{p-1}2 \\
\epsilon_{\rotrep_1} &\in \ulhr^{\rotrep_1}(S^0)(\trivo)\cong \jg{\ground}(\trivo) \\
\xi_{\rotrep_1 - 2} &\in \ulhr^{\rotrep_1-2}(S^0)(\trivo)\cong \r(\trivo) \\
\epsilon_{\rotrep_1}^{-m}\kappa_0 &\in \ulhr^{-m\rotrep_1}(S^0)(\trivo)\cong \jg{\ground}(\trivo) & & \text{for each } m\ge 1 \\
\epsilon_{\rotrep_1}^{-m} \xi_{\rotrep_1-2}^{-n} \nu_{3-2\rotrep_1} &\in \ulhr^\alpha(S^0)(\trivo) \cong \jg{\ground/p}(\trivo) & &
\text{for each } m,n\ge 1 \text{ and }\\
& & & \alpha = (3-2\rotrep_1)-m\rotrep_1-n(\rotrep_1-2) \\
\end{align*}
subject to the following relations.  First, the underlying graded
$A(\trivo)$-module of $\ulhr^*(S^0)(\trivo)$ is as described in
\autoref{thm::addpstructure}; this forces the vanishing of some products.
Further, the Frobenius relation ${\mft(x)\cdot y = \mft(x\cdot \mfr(y))}$
dictates relations involving products with $\mft(\mfr (1))$ and the
generators $\mft(\iota_\alpha^m)$.  Finally, we have the following
relations; the restrictions on the dimensions of the representations
involved can be determined from context and so are left implicit.
\begin{align*}
\epsilon_\alpha \epsilon_\beta & =  \epsilon_{\alpha+\beta} \\
\mu_\alpha \epsilon_\beta & =  \epsilon_{\alpha+\beta} \\
\xi_\alpha \xi_\beta & =  \xi_{\alpha+\beta} \\
\mu_\alpha \xi_\beta & =  d(\alpha) \xi_{\alpha+\beta} \\
\epsilon_\alpha \xi_\beta & =  d(\delta-\beta) \epsilon_\gamma \xi_\delta & &
\text{whenever } \alpha+\beta =  \gamma+\delta \\
\epsilon_\beta^{-1}\kappa_\alpha & =  \epsilon_{\delta}^{-1}\kappa_\gamma & & 
\text{whenever } \alpha+\delta =  \beta+\gamma \\
\epsilon_\beta(\epsilon_\beta^{-1}\kappa_\alpha) & =  \kappa_{\alpha} \\
\mu_\alpha \mu_\beta & = \mu_{\alpha+\beta} + \left( \frac{d(\alpha) d(\beta) -
d({\alpha+\beta})}p \right) \mft(\iota_{\alpha+\beta}) \\
\end{align*}
The generators of $\ulhr^*(S^0)(\trivo)$ and $\ulhr^*(S^0)(\freeo)$ are
related by $\mfr \mu_\alpha = d(\alpha)\iota_\alpha$, for each $\alpha$ of
dimension $(0,0)$, and $\mfr \xi_\alpha = \iota_\alpha$.
\end{thm}

\begin{rem}
It follows from \autoref{thm::multpstructure} and the Frobenius relations
that $\kappa_\alpha \kappa_\beta = \kappa_{\alpha+\beta}$ and $\mu_\alpha
\kappa_\beta = \kappa_{\alpha+\beta}$.  Also, if $x \in
\ulhr^\beta(S^0)(\trivo)$ for some $\beta$ in the fourth quadrant, then
$\kappa_\alpha x = 0$ because $p x = 0$ and $\mft(\iota_\alpha) x =
\mft(\iota_\alpha \mfr(x)) = 0$.
\end{rem}

Note that, in choosing algebra generators at the $\trivo$ level, it is
sufficient to choose $\mu_\alpha$ for $\alpha$ of the form
$\rotrep_j-\rotrep_1$, since any other representation of dimension $(0,0)$
is a sum of these.  Then for any $(m,n)$, as long as we have a generator of
$\ulhr^\alpha(S^0)(\trivo)$ for some $\alpha$ of dimension $(m,n)$,
multiplication by the $\mu_\alpha$ will give generators of
$\ulhr^{\alpha'}(S^0)(\trivo)$ for every other $\alpha'$ of dimension
$(m,n)$.  This explains the somewhat arbitrary appearance of $\rotrep_1$ in
the set of generators in \autoref{thm::multpstructure}.

\chapter[Cells in even dimensions]{Cohomology of $\G$-CW-complexes with cells in even dimensions}\label{ch::freeness}
\section{A freeness theorem}\label{sec::freeness}

In this section, we will temporarily ignore the algebra structure of
$\ulhr^*(-)$, and view it as a module over $\ulhr^*(S^0)$.

When dealing with the nonequivariant singular cohomology of a CW-complex,
it is immediate from the definitions that a space with cells only in even
dimensions must have cohomology which is free as a module over the
cohomology of a point, because all boundary maps in the singular chain
complex are zero.  In this section, we will describe a similar result for
$\ulhr^*(-)$, and give some applications.

In \autoref{def::trivialCW}, we defined a $G$-CW-complex with trivial
cells.  If we are interested in putting CW structures on spaces found in
nature, this is a very restrictive definition, and so we make the following
definition.

\begin{defn}\label{def::CW}
A \defword{$G$-CW-complex} $X$ is the colimit of a sequence of subspaces
$X_n\subset X$ defined as follows.  $X_0$ is a finite $G$-set.  At each
stage, $X_{n+1}$ is formed from $X_n$ by attaching cells of the form
$G\times_K D(V)$ along the boundary $G\times_K S(V)$.  Here $K<G$ is a
subgroup, $V$ is a $K$-representation, and $D(V)$ and $S(V)$ are the unit
disk and unit sphere, respectively, of $V$.  We do not place restrictions
on the dimension of $V$.  Since the $n$ in $X_n$ has no geometric meaning,
we refer to $X_n$ as the $n^\text{th}$ \defword{filtration} of $X$.  A
$G$-CW-complex is \defword{finite} if it is built from finitely many cells.
\end{defn}

For example, a representation sphere $S^V$ has a $G$-CW structure with one
zero cell $G/G$, i.e.\ $G\times_G D(0)$, and one cell $G\times_G D(V)$
attached to the zero cell by collapsing the boundary $G\times_G S(V)$ to a
point.

This new definition includes the trivial $G$-CW-complexes of
\autoref{def::trivialCW}, since for a trivial representation $\bR^n$,
$G/K\times D^n = G\times_K D(\bR^n)$.

In general, it is not clear what the appropriate definition of an
``even-dimensional'' cell is in this context.  We can make the following
definition, however, which agrees with intuition for $G=\G$.

\begin{defn}
A cell $G\times_K DV$ is \defword{even-dimensional} if $|V^L|$ is even for
every subgroup $L$ of $K$.
\end{defn}

Return now to the group $G = \G$.  The only subgroups of $\G$ are
$\G$ itself and the trivial subgroup.  Thus the cells are either of the
form $\G\times D^n$ or $\G\times_\G D(V) = D(V)$ for a (possibly trivial)
$\G$-representation $V$.  Using the definition above, $\G\times D^n$ is
even-dimensional if $n$ is even; $D(V)$ is even-dimensional if $|V|$ and
$|V^\G|$ are both even.\footnote{Recall that, when $p>2$, $|V|$ and
$|V^\G|$ always have the same parity, so either both or neither are even.
They may have different parities for $p=2$, but we still use the same
definition of even-dimensional: both must be even.}

Consider for a moment a $\G$-CW-complex $X$ which is built by attaching one
cell ${\G\times_{K_n} D(V_n)}$ at each stage $n$.  We then have a cofiber
sequence
\[ X_{n-1} \to X_{n} \to \G_+\sm_{K_{n}} S^{V_{n}},\]
inducing a long exact sequence in cohomology,
\[ \xrightarrow{\ \partial\ } \ulhr^*(\G_+\sm_{K_{n}} S^{V_{n}}) \to \ulhr^*(X_{n}) \to \ulhr^*(X_{n-1}) \xrightarrow{\ \partial\ }
\ulhr^{*+1}(\G_+\sm_{K_{n}} S^{V_{n}}) \to\]
From the previous sections, we know that
\[ \ulhr^*(\G_+\sm_{\triv} S^n) \cong \ulhr^{*}(S^n)\boxp A_\freeo
\cong \ulhr^{*-n}(S^0)\boxp A_\freeo  \]
and
\[ \ulhr^*(S^V) \cong \ulhr^{*-V}(S^0),\]
both free modules over $\ulhr^*(S^0)$.  If we knew from an inductive
hypothesis that $\ulhr^*(X_{n-1})$ was free as a module over the graded
Green functor $\ulhr^*(S^0)$, and if we could additionally show
that the boundary map $\partial$ was zero\footnote{Something not necessarily
forced by dimensional considerations!}, then the long exact sequence would
break up into short exact sequences.  Using the freeness of
$\ulhr^*(X_{n-1})$, we could then conclude that
\[ \ulhr^*(X_{n}) \cong \ulhr^*(X_{n-1})\oplus \ulhr^*(\G_+\sm_{K_n} S^{V_n}),\]
making $\ulhr^*(X_{n})$ free as well.  The same reasoning would hold if we allowed
finitely many cells to be attached to $X$ at each stage, by the additivity
axiom for cohomology.  Thus, in such a situation, each $X_n$ would have
cohomology which was free over the cohomology of a point.  Further, since
each $\ulhr^*(X_n)\to \ulhr^*(X_{n-1})$ would then be a projection, the appropriate $\lim^1$ term
would vanish.  This would allow us to conclude that $\ulhr^*(X)$ was the
product of copies of $\ulhr^*(S^0)$ and $\ulhr^*(\G_+)$ in the appropriate
dimensions.  

In the examples we will discuss in this \paper{}, this product agrees with the
corresponding coproduct; so the above sketch outlines a method for showing
that the cohomology of such spaces $X$ is free with generators
corresponding to the cells attached.  
The theorem below, a slightly corrected version of Theorem $2.6$ in
\cite{lgl}, gives several examples of spaces $X$ for which
the sketch above can be fleshed out.\footnote{\cite[Theorem 2.6]{lgl} says
that $X_0$ should consist of a single orbit, but this is not a necessary
restriction.  \cite{lgl} also does not put restrictions on the dimensions of the
cells, which means that its product may not be the same as the desired
coproduct.} These examples are not exhaustive, but they are sufficient for
the applications we have in mind.

Recall the notation for shifted Mackey functors from \autoref{def::unnamed} on
page \pageref{def::unnamed}, interpreted for $G=\G$ on page
\pageref{def::pshifted}.

\begin{thm}[Lewis]\label{thm::freeness}
Suppose that $X$ is a $\G$-CW-complex meeting the following conditions.
\begin{itemize}
\item Each $X_n$ is a finite $\G$-CW-complex.

\item All cells of $X$ have even dimension.

\item Suppose $V$ and $W$ are $\G$-representations such that $DW$ is in the
$n^\text{th}$ filtration $X_n$ for some $n$ and $DV$ is one of the cells
attached to $X_n$ to create $X_{n+1}$.  If $|V|>|W|$, then
$|V^\G|\ge|W^\G|$.

\item For each positive integer $N$, only finitely many cells of the form
$D(V)$ have $|V|\le N$, only finitely many have $|V^\G|\le N$, and only
finitely many cells $\G\times D^n$ have $n\le N$.
\end{itemize}
Then $\ulhr^*(X_+)$ is free as a module over the graded Green functor
$\ulhr^*(S^0)$.  Further, $\ulhr^*(X_+)$ decomposes as a direct sum.  Its
summands consist of one copy of $\ulhr^*({X_0}_+)$, one copy of $\Sigma^V
\ulhr^*(S^0)$ for each cell of the form $D(V)$ with $|V|>0$, and one copy of
$\Sigma^{2n} \ulhr^*(\G_+) \cong \Sigma^{2n} \ulhr^*(S^0)\boxp A_\freeo$
for each cell of the form $\G\times D^{2n}$, $n>0$.
\end{thm}

\begin{proof}
As outlined above, it suffices to show two things: that the boundary maps
are all zero in the long exact sequences associated to the cofiber
sequences
\[ X_{n-1} \to X_n \to X_n/{X_{n-1}},\]
and that the product $\ulhr^*(X_+)$ agrees with the corresponding coproduct.
For ease of notation, suppose that each $X_n$ is formed by attaching a
single cell to $X_{n-1}$; the proof for the general situation is exactly the
same, but the notation is more cumbersome.

We will first address the question of the boundary maps $\partial$.
There are two cases to consider, corresponding to the two types of cell in
a $\G$-CW-complex.  First, suppose that $X_n$ is formed by attaching a cell
$\G\times D^{2m}$ to $X_{n-1}$, giving the long exact sequence 
\[ \cdots \xrightarrow{\ \partial\ } \ulhr^*(\G_+\sm S^{2m}) \to \ulhr^*({X_n}_+) \to \ulhr^*({X_{n-1}}_+)
\xrightarrow{\ \partial\ } \ulhr^{*+1}(\G_+\sm S^{2m}) \to \cdots\]
We know that $\ulhr^*(\G_+\sm S^{2m}) \cong \ulhr^{*-2m}(S^0)\boxp A_\freeo$,
so
\[ \ulhr^{\alpha+1}(\G_+\sm S^{2m}) \cong \begin{cases}
0 & |\alpha| \ne 2m-1 \\
A_\freeo & |\alpha| = 2m-1.
\end{cases} \]
By hypothesis, $\ulhr^*({X_{n-1}}_+)$ is a direct sum of free
$\ulhr^*(S^0)$-modules lying in even dimensions.  The forgetful functor
from $\ulhr^*(S^0)$-modules to graded Mackey functors is right adjoint to
the free functor $M\mapsto \ulhr^*(S^0)\boxp M$.  It follows that a map of
$\ulhr^*(S^0)$-modules out of $\ulhr^*({X_{n-1}}_+)$ is the same as a map of
Mackey functors out of an appropriate sum of shifted copies of $A$ and
$A_\freeo$.  However, these shifted copies of $A$ and $A_\freeo$ all lie in
even dimensions, and $\ulhr^{*+1}(\G_+\sm S^{2m})$ is zero in all even
dimensions.  The vanishing of the target thus forces $\partial = 0$.

Next, suppose that $X_n$ is formed by attaching a cell $D(V)$ to $X_{n-1}$,
so that we have the long exact sequence
\[ \cdots \xrightarrow{\ \partial\ } \ulhr^*(S^V) \to \ulhr^*({X_n}_+) \to \ulhr^*({X_{n-1}}_+)
\xrightarrow{\ \partial\ } \ulhr^{*+1}(S^V) \to \cdots\]
As before, $\partial$ is given by a map of Mackey functors out of an
appropriate sum of copies of $A$ and $A_\freeo$.  $\ulhr^{*+1}(S^V) \cong
\ulhr^{*+1-V}(S^0)$, where $V$ has even dimension, so we may use our
knowledge of $\ulhr^*(S^0)$ to examine the structure of the target.
Whether $p$ is odd or even, the only even degrees in which
$\ulhr^{*+1-V}(S^0)$ is nonzero correspond to the copies of
$\jg{\ground/p}$ in the fourth quadrant.  Suppose that a copy of $A$ in
dimension $W$ has a nonzero target in $\ulhr^{*+1-V}(S^0)$.  It follows
that ${(|W^\G|,|W|) = (|V^\G|-1+2k+1,|V|-1-2\ell)}$ for some positive
integers $k$ and $\ell$.  Hence $|V|>|W|$ but $|V^\G| < |W^\G|$, which
contradicts the assumption in the theorem about the dimensions of the cells
of $X$.  Thus $\partial$ is zero on the components of $\ulhr^*({X_{n-1}}_+)$
corresponding to cells $D(W)$.   Similarly, a map out of $A_\freeo$ in the
source is determined by selecting an element of $\jg{\ground/p}(\freeo)$ in
the corresponding dimension in the target.  Since $\jg{\ground/p}(\freeo) =
0$, $\partial$ must also be zero on these summands.

It follows that, for complexes $X$ meeting the hypotheses of the theorem,
the boundary maps $\partial$ are all zero.  Hence, by induction,
$\ulhr^*({X_n}_+)$ is free on generators in dimensions corresponding to the
cells of $X_n$ for each $n$.

Since $\ulhr^*({X_n}_+)$ is the direct sum of $\ulhr^*({X_{n-1}}_+)$ and a free
summand, and $\ulhr^*({X_n}_+)\to \ulhr^*({X_{n-1}}_+)$ is the projection, the
appropriate $\lim^1$ term vanishes and we see that
\[ \ulhr^*({X}_+) = \prod_V \ulhr^{*-V}(S^0) \times \prod_m \ulhr^{*-2m}(S^0)_\freeo\]
where the products run over the dimensions of the cells in $X$.
The final hypothesis of the theorem guarantees that only finitely many
cells of $X$ contribute to the cohomology in any given dimension, because
the cohomology of a point vanishes throughout the ``third quadrant,''
i.e.\ when $|\alpha^\G|<0$ and $|\alpha|<0$.  Hence the infinite product
above agrees with the corresponding coproduct, and $\ulhr^*({X}_+)$ is free on
generators in dimensions corresponding to the cells of $X$.
\end{proof}

\begin{rem}
An analogous but slightly stronger result holds for homology, by a similar
argument.  Since homology commutes with colimits, we may remove the final
hypothesis of the theorem.
\end{rem}

\begin{rem}
Ferland and Lewis have shown in \cite{FL} that, for homology $\ulhlr_*(-)$,
a similar result with a twist holds if we weaken the third hypothesis of
the theorem to say that there are only finitely many cells $D(V)$ attached
after $D(W)$ such that $|V|>|W|$ but $|V^\G|<|W^\G|$.  In this situation,
it is possible to have nonzero boundary maps $\partial$. However, Ferland
and Lewis show that the resulting homology $\ulhlr_*(X_+)$ is still free, but
on generators whose dimensions they are unable to determine!  In fact, this
brings up a related issue: even if $\ground = \bZ$, it is possible to have
two isomorphic free $\ulhlr_*(S^0)$ modules whose generators are in
different dimensions.  The situation becomes either worse or better,
depending on point of view, when $\ground = \bQ$; in that situation, all
the possible dimensions for the generators of $\ulhlr_*(X_+)$ yield
isomorphic free $\ulhlr_*(S^0)$ modules.
\end{rem}

The main application we have in mind for \autoref{thm::freeness} is $X =
\cpv{\cxuniverse}$, the space of complex lines in a complete complex
universe $\cxuniverse$.  By a complete universe, we mean a complex
$\G$-representation which contains countably infinitely many copies of each
finite-dimensional representation.  For concreteness, we will view
$\cxuniverse$ as follows.  Let $\cxrep$ be the complex one-dimensional
representation of $\G$ whose underlying real representation is the
$\rotrep_1$ of \autoref{subs::pointodd}, and write $\cxrep^n$ to mean
$\cxrep^{\otimes n} = \cxrep\otimes\cdots\otimes\cxrep$.  Then \[ \{\cxrep,
\cxrep^2, \ldots, \cxrep^{p-1}, \cxrep^{p} = 1\}\] gives a complete set of
the irreducible complex representations of $\G$.  We may then view
$\cxuniverse$ as the direct sum
\begin{align*}
\cxuniverse & = 1\oplus \cxrep\oplus \cxrep^2\oplus \cdots\oplus \cxrep^{p-1}\oplus
        \cxrep^p\oplus \cxrep^{p+1}\oplus \cdots \\
    & \cong 1\oplus \cxrep\oplus \cxrep^2\oplus \cdots\oplus \cxrep^{p-1}\oplus
        1\oplus \cxrep\oplus \cdots
\end{align*}
built up from the flag of subrepresentations ${\cxuniverse^0 \subset
\cxuniverse^1 \subset \cxuniverse^2 \subset \cdots}$, where we define
${\cxuniverse^n \defeq \bigoplus_{k=0}^n \cxrep^k}$.

The usual nonequivariant Schubert cell decomposition generalizes to give a
$\G$-CW structure on $\cpv{\cxuniverse}$.  Explicitly, suppose that we are
given a complex line in $\cxuniverse$.  We can use the decomposition
${\cxuniverse = 1\oplus \cxrep\oplus \cxrep^2 \oplus\cdots}$ to choose a
basis for $\cxuniverse$; in the coordinates for this basis, our line can be
viewed as the set of $\bC$-scalar multiples of a point with coordinates
${(x_0,x_1,\ldots,x_{N-1},1,0,0,\ldots)}$.  That is, we choose any point in
the line and then normalize so that the last nonzero coordinate is 1.
Then, for a fixed $N$, the set of all points of the form
${(x_0,x_1,\ldots,x_{N-1},1,0,0,\ldots)}$ is the interior of a disk of real
dimension $2N$.  When we consider the coordinate-wise action of $\G$ on
this disk, we see that it is the interior of the $\G$-representation
$D(\cxrep^{-N} \cxuniverse^{N-1})$, with boundary in
$D(\cxrep^{-(N-1)}\cxuniverse^{N-2})$.

Thus, if we define $\omega_N \defeq \cxrep^{-N}\cxuniverse^{N-1} =
\cxrep^{-N}(1\oplus\cxrep\oplus\cdots\oplus\cxrep^{N-1})$,
\label{def::omega} we see that $\cpv{\cxuniverse}$ has a $\G$-CW structure
with one cell $D(\omega_N)$ of real dimension $(|\omega_N^\G|,|\omega_N|) =
(2\floor{\frac{N}p},2N)$ for each integer $N\ge 0$.  These dimensions are
plotted for $p=3$ in \autoref{fig::dimplot}.

\begin{figure}
\[ \xymatrix@R=-0.25pc@C=0.5pc{
{\phantom{\makebox[0pt][l]{$|\alpha^\G|$}}} &   & \makebox[0pt]{$|\alpha|$}  &   &   &   &   &   &   &   &   & \\
{\phantom{\makebox[0pt][l]{$|\alpha^\G|$}}} & . &   & . &   & . &   & . &   & . &   & \\
{\phantom{\makebox[0pt][l]{$|\alpha^\G|$}}} &   & . &   & \makebox[0pt][c]{$\omega_5$} &   & . &   & . &   & . & \\
{\phantom{\makebox[0pt][l]{$|\alpha^\G|$}}} & . &   & . &   & . &   & . &   & . &   & \\
{\phantom{\makebox[0pt][l]{$|\alpha^\G|$}}} &   & . &   & \makebox[0pt][c]{$\omega_4$} &   & . &   & . &   & . & \\
{\phantom{\makebox[0pt][l]{$|\alpha^\G|$}}} & . &   & . &   & . &   & . &   & . &   & \\
{\phantom{\makebox[0pt][l]{$|\alpha^\G|$}}} &   & . &   & \makebox[0pt][c]{$\omega_3$} &   & . &   & . &   & . & \\
{\phantom{\makebox[0pt][l]{$|\alpha^\G|$}}} & . &   & . &   & . &   & . &   & . &   & \\
{\phantom{\makebox[0pt][l]{$|\alpha^\G|$}}} &   & \makebox[0pt][c]{$\omega_2$} &   & . &   & . &   & . &   & . & \\
{\phantom{\makebox[0pt][l]{$|\alpha^\G|$}}} & . &   & . &   & . &   & . &   & . &   & \\
{\phantom{\makebox[0pt][l]{$|\alpha^\G|$}}} &   & \makebox[0pt][c]{$\omega_1$} &   & . &   & . &   & . &   & . & \\
{\phantom{\makebox[0pt][l]{$|\alpha^\G|$}}} & . &   & . &   & . &   & . &   & . &   & \\
\ar@{<.>}[rrrrrrrrrrr] &   & \makebox[0pt][c]{$\omega_0$} &   & . &   & . &   & . &   & . & \makebox[0pt][l]{$|\alpha^\G|$} \\
{\phantom{\makebox[0pt][l]{$|\alpha^\G|$}}} & . &   & . &   & . &   & . &   & . &   & \\
{\phantom{\makebox[0pt][l]{$|\alpha^\G|$}}} &   & \ar@{<.>}[uuuuuuuuuuuuu]  &   &   &   &   &   &   &   &   & \\
} \]
\caption[\lofspace{}The equivariant dimensions of Schubert cells]{The equivariant dimensions of the first few Schubert cell representations $\omega_N$ for $p=3$.  The dimensions of virtual representations are marked with dots; dimensions not corresponding to virtual representations are left blank.}
\label{fig::dimplot}
\end{figure}
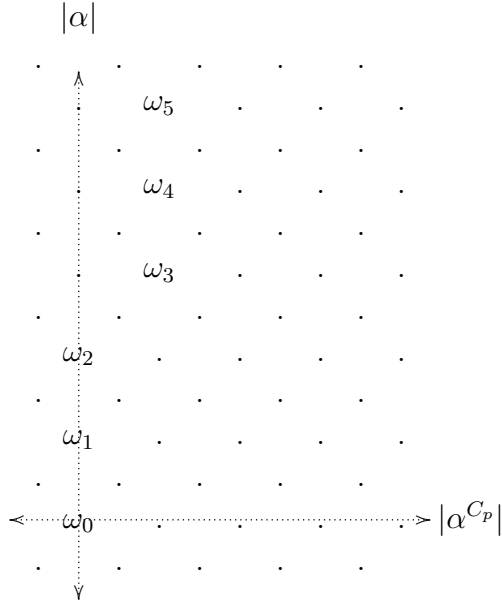

\begin{cor}\label{cor::cxproj}
Let $\cpv{\cxuniverse}$ be the space of complex lines in a complete complex
universe $\cxuniverse$.  Then the $\G$-CW structure on $\cpv{\cxuniverse}$
described above meets the requirements of \autoref{thm::freeness}, so as a
module over $\ulhr^*(S^0;A)$,
\[ \ulhr^*(\cpv{\cxuniverse}_+;A)\cong \bigoplus_{N\ge0} \Sigma^{\omega_N}
\ulhr^*(S^0;A)\]
where $\Sigma^{\omega_N}$ denotes the suspension by the underlying real
representation of the complex representation $\omega_N$.\hfill $\Box$
\end{cor}

\begin{rem}
Notice that choosing a different flag of subrepresentations of
$\cxuniverse$ could result in a different $\G$-CW structure on
$\cpv{\cxuniverse}$, with cells $D(V)$ in \emph{different} dimensions!  The
resulting cohomologies must be isomorphic as $\ulhr^*(S^0)$-modules, of
course, but in general the isomorphism is not induced by a space level map,
because there is no equivariant cellular approximation theorem---a $G$-map
between $G$-CW-complexes need not be $G$-homotopic to a cellular map.
\end{rem}

We also have the following easy consequence of \autoref{thm::freeness}.

\begin{cor}\label{lem::trivspaces}
Let $Y$ be a CW-complex with cells only in even dimensions, viewed as a
$\G$-space with the trivial action.  Then there is an isomorphism of
$\ulhr^*(S^0)$-modules
\[ \ulhr^*({Y}_+) \cong \ulhr^*(S^0) \otimes \neqhr^*(Y_+;\ground).\] 
\end{cor}

\begin{proof}
The notation simply means that $\ulhr^*(Y_+)$ is the free
$\ulhr^*(S^0)$-module on generators in dimensions of the cells of
$Y$, which all have the trivial $\G$ action.
\end{proof}

\section{Tools for Computing Multiplicative Structure}

So far we have not said anything about the multiplicative structure of
$\ulhr^*(X_+)$ for $G$-CW-complexes $X$ satisfying the hypotheses of
\autoref{thm::freeness}.  In this section we will give a result which
partially addresses the issue of multiplicative structure.

To begin with, for CW-complexes with a trivial $\G$ action and cells only
in even dimensions, the multiplicative structure of $\ulhr^*(-)$ is
determined by the usual nonequivariant cup product structure on
$\neqhr^*(-)$.

\begin{prop}\label{lem::trivspacesmult}
Suppose $Y$ is a based CW-complex with cells only in even dimensions such
that $\neqhr^*(Y) \cong \colim \neqhr^*(Y_m)$, the colimit of
the cohomologies of the skeleta, and similarly for $\ulhr^*(Y)$.  Then the
isomorphism of \autoref{lem::trivspaces} is an isomorphism of graded Green
functors.
\end{prop}

{\newcommand{\freegreen}{\sG}
\begin{proof}
We will use the representing spectra of the cohomology theories in
question.  Let $Y$ be a (nonequivariant) CW-complex with cells only in even
dimensions; we will also write $Y$ when viewing it as a $\G$-CW-complex
with trivial $\G$-action.  Let $\repnoneq$ be the representing spectrum for
nonequivariant cohomology $\neqhr^*(-;\ground)$, and $\repeq$ the
representing spectrum for $\ulhr^*(-;A)$.

The category of Mackey functors is tensored over $\ground$-modules, and
genuine $\G$-spectra are tensored over nonequivariant spectra.  Thus, in
the usual way, the isomorphism $\ground\otimes A\to A$ induces a map
$\repnoneq\sm\repeq \to \repeq$.  Combining this with the usual abstract
nonsense provides a map of function spectra
\[ F(Y,\repnoneq)\sm F(S^0,\repeq)\to F(Y\sm S^0,\repnoneq\sm\repeq)\to
F(Y,\repeq).\]
Taking homotopy groups or Mackey functors, as appropriate, then produces a
map of graded Green functors
\[ \varphi\colon \neqhr^*(Y;\ground)\otimes \ulhr^*(S^0) \to
\ulhr^*(Y),\]
natural in $Y$.
We claim that this map is an isomorphism.  This can be proved using
induction, the long exact sequence in cohomology, and the fact that our $Y$
has cells only in even dimensions.  For convenience in what follows, given
a ring $B$, we will write $\freegreen B$ to mean the Green functor
$B\otimes \ulhr^*(S^0)$.  Thus $\varphi$ above is a map
$\freegreen\neqhr^*(Y;\ground)\to \ulhr^*(Y)$.

Write $Y_m$ for the $m$-skeleton of $Y$ as a nonequivariant CW-complex.
Then we have cofiber sequences ${Y_{2m}\to Y_{2m+2}\to Y_{2m+2}/{Y_{2m}}}$
for each $m$, and $Y_{2m+2}/{Y_{2m}}$ is a wedge of copies of $S^{2m+2}$.
The boundary maps in the associated long exact sequence are zero, so we
have a short exact sequence
\[ 0\to \neqhr^n(Y_{2m+2}/{Y_{2m}})\to \neqhr^n(Y_{2m+2})\to
\neqhr^n(Y_{2m}) \to 0\]
for each $n$; $\ground$ coefficients are implicit in the notation.  Since
the arrows are all induced by maps of spaces, we in fact have a short exact
sequence of $\ground$-algebras
\[ 0\to \neqhr^*(Y_{2m+2}/{Y_{2m}})\to \neqhr^*(Y_{2m+2})\to
\neqhr^*(Y_{2m}) \to 0.\]
Since $Y$ has cells only in even dimensions, $\neqhr^*(Y_{2m})$ is a
free $\ground$-module for each $m$, and thus tensoring the above short
exact sequence with the Mackey functor $\ulhr^*(S^0)$ gives another short
exact sequence
\[ 0\to \freegreen\neqhr^*(Y_{2m+2}/{Y_{2m}})\to \freegreen\neqhr^*(Y_{2m+2})\to
\freegreen\neqhr^*(Y_{2m}) \to 0.\]
Since our map $\varphi$ above is natural, we get a map of short exact
sequences of graded Green functors
\[ 
\xymatrix{
0 \ar[r] & \freegreen\neqhr^*(Y_{2m+2}/{Y_{2m}}) \ar[r] \ar_\varphi[d] & 
\freegreen\neqhr^*(Y_{2m+2})  \ar[r] \ar_\varphi[d]& 
\freegreen\neqhr^*(Y_{2m})  \ar[r] \ar_\varphi[d] & 0 \\
0 \ar[r] & \ulhr^*(Y_{2m+2}/{Y_{2m}}) \ar[r] &
\ulhr^*(Y_{2m+2}) \ar[r] &
\ulhr^*(Y_{2m})   \ar[r] & 0 \\
} \]
It is clear from the suspension axiom that
\[{\varphi\colon \freegreen\neqhr^*(Y_{2m+2}/{Y_{2m}})\to
\ulhr^*(Y_{2m+2}/{Y_{2m}})}\]
is an isomorphism.  We may then do an inductive argument; the right-hand
vertical arrow is an isomorphism by the inductive hypothesis, and it
follows from the Five Lemma that the middle arrow is also an isomorphism.
It follows by induction that $\varphi$ is an isomorphism on the
$2m$-skeleton of $Y$ for every $m$.  Since by assumption the cohomology of
$Y$ is the colimit of the cohomologies of the skeleta, it follows that 
\[ \varphi\colon \neqhr^*(Y)\otimes \ulhr^*(S^0)\to \ulhr^*(Y) \]
is an isomorphism.
\end{proof}
}

For a $\G$-CW-complex $X$ satisfying the hypotheses of
\autoref{thm::freeness}, this result allows us to compare the unknown ring
${\ulhr^*(X_+)(\trivo)}$ with more familiar objects, as follows.  The map
$\G_+\sm X_+ \to X_+$ coming from the projection
$\rho\colon\freeo\to\trivo$ induces a homomorphism of Green functors
\[ \rho^*\colon \ulhr^*(X_+) \to \ulhr^*(\G_+\sm X_+).\]
We have previously seen that $\ulhr^*(\G_+\sm X_+) \cong
\ulhr^*(X_+)_\freeo = \ulhr^*(X_+)\boxp A_\freeo$, using the notation for
shifted Mackey functors from page \pageref{def::unnamed} and page
\pageref{def::pshifted}.  The map $\rho^*$ can also be viewed as the map
${\ulhr^*(X_+) \to \ulhr^*(X_+)_\freeo}$ induced by the projection
$\freeo\to\trivo$, and we will switch back and forth between these views in
what follows.  Since $\ulhr^*(X_+)(\freeo) \cong \neqhr^{|*|}(X_+)$, the
target of $\rho^*$ is something we can already compute.

It also follows from the definition of a $\G$-CW-complex and from our
definition of ``even dimensional'' that, if $X$ satisfies the hypotheses of
\autoref{thm::freeness}, then $X^\G$ is a CW-complex with trivial $\G$
action, having cells only in even dimensions.  Thus
\autoref{lem::trivspacesmult} applies, so we know $\ulhr^*(X^\G_+)$ as a
Green functor.  The inclusion $i\colon X^\G\hookrightarrow X$ induces
$\ulhr^*({X}_+)\to \ulhr^*(X^\G_+)$.  Putting these together, we see that
the map
\[ \rho^* \oplus i^* \colon \ulhr^*(X_+) \to \ulhr^*(X_+)_\freeo \oplus
\ulhr^*(X^\G_+) \]
is a homomorphism of Green functors.  

We have the following surprising result.\footnote{Flawed proof given by
Lewis in \cite{lgl}; it has been corrected by the addition of
\autoref{lem::crucial}.}

\begin{thm}\label{thm::multcomparison}
Let $X$ be a $\G$-CW-complex satisfying the hypotheses of
\autoref{thm::freeness}, so that $\ulhr^*(X_+)$ is a free
$\ulhr^*(S^0)$-module.  Further suppose that the ground ring $\ground$ has
no elements of order $p$.  Then for any $\alpha\in RO(\G)$ of even
dimension, the map
\[ \rho^* \oplus i^* \colon \ulhr^\alpha(X_+) \to
\ulhr^\alpha(X_+)_\freeo \oplus \ulhr^\alpha(X^\G_+) \]
is a monomorphism.  Under the same hypotheses, the map
\[ \rho^* \oplus \hat{i}^* \colon \ulhr^\alpha(X_+) \to
\ulhr^\alpha(X_+)_\freeo \oplus \left( \jg{\ground}\boxp \ulhr^\alpha(X^\G_+) \right) \]
becomes a monomorphism after tensoring with $\ground\left[1/p\right]$, for any $\alpha$.
Here $\hat{i}^*$ is the composite
\[ {\ulhr^\alpha(X_+) \xrightarrow{\ i^*\ }\ulhr^\alpha(X^\G_+) = A \boxp
\ulhr^\alpha(X^\G_+) \xrightarrow{1\boxp\id} \jg{\ground}\boxp \ulhr^\alpha(X^\G_+)};\]
the final arrow comes from the map $A\to\jg{\ground}$ taking $\id\mapsto
1\in\ground$.
\end{thm}

The proof depends on the following crucial lemma.

\begin{thm}\label{lem::crucial}
The statement of \autoref{thm::multcomparison} holds when $X_+$ is replaced
by a sphere of the form ${\G_+\sm S^{m}}$ or $S^V$ for a
$\G$-representation $V$.
\end{thm}

Note that this lemma does not immediately follow from the suspension
isomorphism and the statement for $S^0$, because suspension by
representation spheres does not preserve fixed points.  The proof appears
below, after the proof of \autoref{thm::multcomparison}.

By the additivity axiom, it follows from \autoref{lem::crucial} that the
statement of the theorem also holds for wedge sums of spheres $\G_+\sm
S^{2m}$ and $S^V$.  Given this, the proof of \autoref{thm::multcomparison}
is a straightforward induction argument.

\begin{proof}[Proof of \autoref{thm::multcomparison}]
We proceed by induction.  By \autoref{lem::crucial}, both maps of
\autoref{thm::multcomparison} are monomorphisms for finite $\G$-sets.

Now suppose that the result is known for $X_{n-1}$, and consider $X_n$.
Since $X$ satisfies the hypotheses of \autoref{thm::freeness}, we know that
the boundary maps in the long exact sequence associated to the cofibration
\[ X_{n-1}\to X_n \to {X_n}/{X_{n-1}} \]
are all zero.  Since long exact sequences associated to cofibrations are
natural, the map $\rho^*\oplus i^*$ induces the commuting diagram of Mackey
functors below; the columns are exact.
\[ \xymatrix@C=3pc{
  0 \ar[d] & 0 \ar[d] \\
  \ulhr^\alpha({X_n}/{X_{n-1}}) \ar[r]^-{\rho^*\oplus i^*} \ar[d] &
  \ulhr^\alpha({X_n}/{X_{n-1}})_\freeo \oplus \ulhr^\alpha(({X_n}/{X_{n-1}})^\G) \ar[d] \\
  \ulhr^\alpha({X_n}_+) \ar[r]^-{\rho^*\oplus i^*} \ar[d] &
  \ulhr^\alpha({X_n}_+)_\freeo \oplus \ulhr^\alpha({X_n}^\G_+)\ar[d] \\
  \ulhr^\alpha({X_{n-1}}_+) \ar[r]^-{\rho^*\oplus i^*} \ar[d] &
  \ulhr^\alpha({X_{n-1}}_+)_\freeo \oplus \ulhr^\alpha({X_{n-1}}^\G_+) \ar[d] \\
  0 & 0 \\
} \]
If $\alpha$ has even dimension, the bottom horizontal arrow is a
monomorphism by the induction hypothesis, and the top one by
\autoref{lem::crucial}.  Hence, by the Five Lemma, the middle arrow is a
monomorphism as well.  The same holds for the corresponding diagram with
horizontal arrows given by $\rho^*\oplus\hat{i}^*$, for any $\alpha$.

Thus, by induction, the theorem holds for $\ulhr^*(X_n)$, for each $n$.
Since $X$ satisfies the hypotheses of \autoref{thm::freeness}, $\ulhr^*(X)
\cong \colim \ulhr^*(X_n)$, so the theorem must hold for $\ulhr^*(X)$ as
well.
\end{proof}

We now turn to \autoref{lem::crucial}.  Since there are multiple parts to
the argument, we will break the proof into several smaller lemmas.  These
are then used to prove \autoref{lem::crucial} on page
\pageref{pf::crucial}.

\begin{lem}\label{lem::freeolevel}
The assertions of \autoref{thm::multcomparison} hold at the $\freeo$ level,
for any $\G$-space $X$.
\end{lem}

\begin{proof}
At the $\freeo$ level,
\[{\rho^*:\ulhr^*(X_+)(\freeo)\to \ulhr^*(X_+)(\freeo\times\freeo)\cong
\bigoplus_{g\in \G}\ulhr^*(X_+)(\freeo)}\] is the diagonal map, a
monomorphism.  Thus $\rho^*\oplus i^*$ and $\rho^*\oplus\hat{i}^*$ are
monomorphisms as well, and the latter remains a monomorphism after
tensoring with $\ground[1/p]$.
\end{proof}

It remains only to prove the result at the $\trivo$ level.  For the lemmas
below, recall from \autoref{ex::jg} that
\[ \jg{\ground}\boxp M = \jg{M(\trivo)/{(\im\mft)}},\]
and that ${\ulhr^*(X)(\trivo) = \hr^*(X)}$ and ${\ulhr^*(X)(\freeo) =
\hr^*(\G_+\sm X)}$ by definition.

\begin{lem}\label{lem::dim0}
If $X$ is a finite $\G$-set, then the map
\[ \rho^*\oplus i^* \colon \hr^\alpha(X_+) \to
\hr^\alpha(\G_+\sm X_+) \oplus \hr^\alpha(X^\G_+)\]
is a monomorphism for $\alpha$ of even dimension, and
\[ \rho^*\oplus \hat{i}^* \colon \hr^\alpha(X_+) \to
\hr^\alpha(\G_+\sm X_+) \oplus \left( \hr^\alpha(X^\G_+)/{(\im\mft)} \right)\]
becomes a monomorphism for every $\alpha$ after tensoring with
$\ground[1/p]$.
\end{lem}

\begin{proof}
The first assertion holds by inspection, since $i^*$ is an isomorphism when
the set is ${X_+ = S^0}$, and $\rho^*$ is the diagonal map, a monomorphism,
when ${X_+ = \G_+}$.  For the same reason, the second assertion holds when
${X_+ = \G_+}$.  

It thus suffices to consider the second assertion when $X_+ = S^0$.  There
are three cases to consider, based on the dimension of $\alpha$.
\begin{itemize}
\item If $|\alpha|=0$ and $|\alpha^\G|\ne 0$, 
then $\ulhr^\alpha(S^0)$ is one of $\r$, $\l$, $\rminus$, or $\lminus$.
For $\r$, $\l$, and $\rminus$, $\rho^*$ is a 
monomorphism and remains one after inverting $p$. For $\lminus$, $\rho^*$
becomes a monomorphism after inverting $p=2$.  Thus $\rho^*\oplus
\hat{i}^*$ is a monomorphism as well, in each case.  

\item If $\alpha$ has dimension $(0,0)$, we know that $\ulhr^\alpha(S^0) =
\tw{A}{d}$ for some $d$ prime to $p$, and that $\ulhr^\alpha(S^0)_\freeo =
A_\freeo$.  We can also identify $\jg{\tw{A}{d}(\trivo)/{(\im\mft)}} \cong
\ground$.  It follows that $(\rho^*\oplus\hat{i}^*)(\trivo)$ is the map
$\ground\oplus\ground \to \ground\oplus\ground$ given by
$\mymatrix{d&1\\p&0}$.  This is a monomorphism and remains so after
tensoring with $\ground[1/p]$.

\item If $|\alpha|\ne 0$, then $\ulhr^\alpha(S^0)$ is $\jg{\ground}$ or
$\jg{\ground/p}$, so $\im\mft = 0$ in $\ulhr^\alpha(S^0)(\trivo)$ and
$\hat{i}^*$ is an isomorphism.  Thus $\rho^*\oplus\hat{i}^*$ is a 
monomorphism.
\end{itemize}
Putting these results together gives the statement of the lemma.
\end{proof}

\begin{lem}\label{lem::rotsphere}
If $\rotrep$ is the underlying real two dimensional representation of a
nontrivial irreducible complex representation of $\G$,
then the map
\[ \rho^*\oplus i^* \colon \hr^\alpha(S^\rotrep) \to
\hr^\alpha(\G_+\sm S^\rotrep) \oplus \hr^\alpha\left(
(S^\rotrep)^\G_+\right)\]
is a monomorphism for $\alpha$ of even dimension, and
\[ \rho^*\oplus \hat{i}^* \colon \hr^\alpha(S^\rotrep) \to
\hr^\alpha(\G_+\sm S^\rotrep) \oplus \left( \hr^\alpha\left( (S^\rotrep)^\G_+\right)/{(\im\mft)} \right)\]
becomes a monomorphism for every $\alpha$ after tensoring with
$\ground[1/p]$.
\end{lem}

\begin{proof}
Write $X = S^\rotrep$, for brevity.  Then $X^\G = S^0$ and
$S^0\hookrightarrow S^\rotrep$ is the inclusion of the points at $0$ and
$\infty$.  We can give $X$ the structure of a $\G$-CW-complex with trivial
cells, as follows.  We let $X_0 = S^0 = X^\G$; $X_1$ is given by attaching
a cell $\G\times D^1$ to $X_0$ with endpoints the two distinct points of
$X_0$; and $X_2 = X$ is given by attaching a cell $\G\times D^2$ by gluing
in copies of $D^2$ between consecutive copies of $D^1$ in $X_1$.  Although
this $\G$-CW-structure does not have cells only in even
dimensions, we know the cohomology of the cofibers $X_1/X_0 = S^1 \sm \G_+$
and $X_2/X_1 = S^2\sm \G_+ = S^\rotrep\sm \G_+$, and we can use this to
learn about the map induced by $S^0\hookrightarrow S^\rotrep$ on
cohomology.

From here, we need to consider the cases $p=2$ and $p>2$ separately; the
arguments are similar but slightly more complicated for $p=2$.

We will begin by assuming $p$ is an odd prime.  First, consider the
cofibration sequence
\[ S^0\to X_1 \to X_1/S^0 = S^1\sm \G_+ \to S^1\sm S^0 \to \cdots \]
The map $S^1\sm \G_+ \to S^1\sm S^0$ is induced by the projection
$\rho\colon \freeo\to\trivo$, hence agrees with the restriction $\mfr$ on
cohomology.  It follows that, for any $\alpha$, we have an exact sequence
\[ \ulhr^{\alpha-1}(S^0) \xrightarrow{\ \rho^*} \ulhr^{\alpha-1}(\G_+) \to
\ulhr^\alpha(X_1) \to \ulhr^\alpha(S^0) \xrightarrow{\ \rho^*} \ulhr^\alpha(\G_+).\]
This breaks up into a short exact sequence
\begin{equation}\label{eqn::mfrseq}
0\to (\coker\rho^*)^{\alpha-1} \to \ulhr^\alpha(X_1) \to
(\ker\rho^*)^\alpha \to 0.
\end{equation}
At the $\trivo$ level, $\rho^*$ agrees with the map
\[ {\mfr\colon \ulhr^*(S^0)(\trivo) \to \ulhr^*(S^0)(\freeo)\cong
\ulhr^*(\G_+)(\trivo)},\]
so we can use our knowledge of $\ulhr^*(S^0)$ to describe the kernel and
cokernel of $\rho^*$.  Of the Mackey functors appearing in $\ulhr^*(S^0)$
for $p>2$, $\mfr$ fails to be surjective only for $\l$; it fails to be
injective for $\jg{\ground}$, $\jg{\ground/p}$, and $\tw{A}d$.  A pictorial
representation of $(\ker\rho^*)(\trivo)$ and $(\coker\rho^*)(\trivo)$ for
$p>2$ is shown in \autoref{fig::cokerrho}.  For every $\alpha$, at least
one of $(\coker\rho^*)^{\alpha-1}$ and $(\ker\rho^*)^\alpha$ vanishes, so
the short exact sequence of \eqnref{eqn::mfrseq} splits;
${\ulhr^\alpha(X_1)\cong (\coker\rho^*)^{\alpha-1}\oplus
(\ker\rho^*)^\alpha}$, and the map $\ulhr^*(X_1)\to \ulhr^*(S^0)$ induced
by $S^0\hookrightarrow X_1$ is the projection onto $\ker\rho^*$.


\begin{figure}
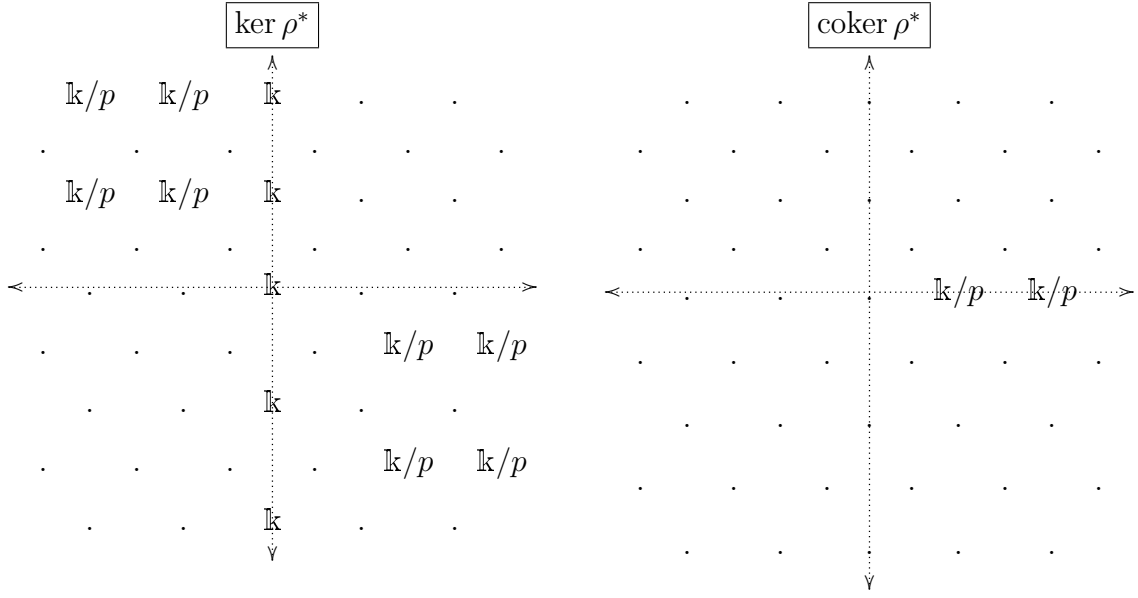

\[ 
\cohomplot{}{\framebox{$\ker\rho^*$}}{\ground}{.}{.}{\ground}{\ground/p}{}{0.7pc}
\cohomplot{}{\framebox{$\coker\rho^*$}}{.}{.}{\ground/p}{.}{.}{\ground/p}{0.7pc}
\]
\caption[\lofspace{}The kernel and cokernel of $\rho$]{A plot showing the kernel and cokernel of
$\rho$ at the $\trivo$ level; $p$ is odd in this picture, and $\ground$ has
no torsion of order $p$.  Compare to \autoref{fig::modp}.
As always, the axes represent $|\alpha^\G|$ and
$|\alpha|$.}
\label{fig::cokerrho}
\end{figure}

Similarly, in the cofibration sequence
\[ X_1 \to S^\rotrep \to S^\rotrep\sm \G_+ \to S^1\sm X_1 \to \cdots \]
the map $S^\rotrep\to S^\rotrep\sm \G_+$ is exactly the transfer associated
to the projection $\rho\colon\freeo\to\trivo$; see, e.g.\ \cite[page 83]{AK}.
Thus the induced map on cohomology is given by ${\ulhr^*(\G_+)(\trivo)\cong
\ulhr^*(S^0)(\freeo) \xrightarrow{\mft} \ulhr^*(S^0)(\trivo)}$.  It follows
as before that the cohomology long exact sequence arising from the
cofibration sequence above breaks up into short exact sequences
\begin{equation}\label{eqn::mftseq}
0\to (\coker\mft)^{\alpha-\rotrep} \to \ulhr^\alpha(X_1) \to
(\ker\mft)^{\alpha-\rotrep+1} \to 0
\end{equation}
where $\mft$ denotes the transfer map on $\ulhr^*(S^0)$.
By assumption $\ground$ has no torsion of order $p$, so
$\ker\mft$ vanishes for every Mackey functor appearing in $\ulhr^*(S^0)$
when $p>2$.  So $\ulhr^\alpha(X_1)\cong (\coker\mft)^{\alpha-\rotrep}$, and
$\ulhr^*(S^\rotrep)\to \ulhr^*(X_1)$ is the projection.\footnote{Note that
this is compatible with ${\ulhr^\alpha(X_1) \cong (\coker\rho^*)^{\alpha-1}
\oplus (\ker\rho^*)^\alpha}$.} A pictorial representation of
$(\coker\mft)(\trivo)$ is shown in \autoref{fig::coker}.

\begin{figure}
\[ 
\cohomplot{|\alpha^\G|}{|\alpha|}{\ground}{\ground/p}{.}{\ground}{\ground/p}{}{0.7pc} 
\]
\caption[\lofspace{}The cokernel of $\mft$]{A plot showing the cokernel of $\mft$ at the $\trivo$ level, again
for $p$ odd.  Compare to $\ulhrnoph^*(S^0)(\trivo)$ in
\autoref{fig::modp}.}
\label{fig::coker}
\end{figure}

Putting all of this together, we see that the inclusion of fixed points
$S^0\hookrightarrow S^\rotrep$ induces the composite of two projections
followed by an inclusion
\begin{align*}
\hr^\alpha(S^\rotrep) & \cong \hr^{\alpha-\rotrep}(S^0) \twoheadrightarrow
(\coker\mft)^{\alpha-\rotrep} \cong \\
& \cong (\coker\rho^*)^{\alpha-1}\oplus
(\ker\rho^*)^\alpha \twoheadrightarrow (\ker\rho^*)^\alpha \hookrightarrow
\hr^\alpha(S^0)
\end{align*}
at the $\trivo$ level of the cohomology Mackey functors.  If we restrict to
representations $\alpha$ of even dimension, or if we tensor everything in
sight with $\ground[1/p]$, then $(\coker\rho^*)^{\alpha-1}$ vanishes, and
the kernel of $i^*\colon \ulhr^*(S^\rotrep)(\trivo)\to \ulhr^*(S^0)(\trivo)$ is
exactly the image of $\mft$ in $\ulhr^*(S^\rotrep)$.  The kernel of
$\rho^*$ is, of course, the kernel of $\mfr$, and so we are interested in
the intersection of the kernel of $\mfr$ and the image of $\mft$.  

However, recall that in any Mackey functor $M$, the composite $\mfr\mft$ is
the trace map of the $\G$-action on $M(\freeo)$.  Since $p>2$ and $\ground$
has no $p$-torsion, this trace map is a monomorphism for every Mackey
functor appearing in $\ulhr^*(S^\rotrep) \cong \ulhr^{*-\rotrep}(S^0)$.
Thus ${(\ker\mfr)}$ and ${(\im\mft)}$ have trivial intersection, and 
the map $\rho^*\oplus i^*$ is a monomorphism for every $\alpha$ of even
dimension.

Similarly, the image of $i^*$ in $\hr^\alpha(S^0)$ is the kernel of $\mfr$,
and thus has empty intersection with the image of $\mft$.  It follows that
$\rho^*\oplus \hat{i}^*$ becomes a monomorphism after inverting $p$, for
every $\alpha$.  This completes the proof for odd primes.

{\renewcommand{\G}{{C_2}}
We next turn to the proof for $p=2$, where the existence of a nontrivial
real one-dimensional representation creates additional dimensions $\alpha$ where
$\ulhr^\alpha(S^0)\ne 0$.  In particular, there are $\alpha$ for which
$(\coker \rho^*)^{\alpha-1}$ and $(\ker \rho^*)^\alpha$ are both nonzero,
and the short exact sequence of \eqnref{eqn::mfrseq} does not split.
Additionally, $\ker \mft\ne 0$ for the Mackey functors $\lminus$ and
$\rminus$, and all three terms in the short exact sequence of
\eqnref{eqn::mftseq} are nonzero for $\alpha=\signrep$.

Fortunately, we also have some more explicit knowledge of the spaces
involved.  We have $\rotrep = \signrep\oplus \signrep$, i.e. two copies of
the one-dimensional sign representation; $X_1 = S^\signrep$, whose
cohomology is known; and the inclusion ${X_1 = S^\signrep \hookrightarrow
S^{2\signrep} = S^\rotrep}$ is the suspension of $S^0\hookrightarrow
S^\signrep$.

The arguments for $p>2$ apply equally well for $p=2$ in dimensions $\alpha
= a_0 + a_1\signrep$ where $a_1$ is even, and in fact for dimensions where
one of the terms in \eqnref{eqn::mfrseq} and also
$(\ker\mft)^{\alpha-\rotrep+1}$ in \eqnref{eqn::mftseq} vanish.  Thus we
can concentrate on the exceptions.  If we generate a picture from
\autoref{fig::mod2} in the same way that \autoref{fig::cokerrho} was
generated from \autoref{fig::modp}, we see that these dimensions are
exactly those $\alpha$ of the form $a + (1-a)\signrep$ for some integer
$a$.  These all have $|\alpha| = 1$, hence are not of even dimension; so we
need only consider $\alpha$ for which
$\ulhr^\alpha(S^\rotrep)(\trivo)\otimes \ground[1/2] \ne 0$.  The only such
$\alpha$ is $\alpha=\signrep$; using the suspension isomorphism, $i^*\colon
\hr^\signrep(S^{2\signrep})\to \hr^\signrep(S^0)$ factors as
\[ \hr^{-\signrep}(S^0) \to \hr^0(S^0)\to \hr^\signrep(S^0).\]
Using \eqnref{eqn::mfrseq}, we see the left-hand map is a monomorphism
with image $\ker\mfr\subset \hr^0(S^0)$.  From \eqnref{eqn::mftseq}, we
see that the kernel of the right-hand map is $\im\mft$.  Since $\ker\mfr
\cap \im\mft = 0$ in $\ulhr^0(S^0)(\trivo) = A(\trivo)$, it follows that
$i^*\colon \ground\to\ground$ is a monomorphism and remains so after
inverting $p=2$.  Since $\im\mft = 0$ in $\hr^\signrep(S^0)$, this
completes the proof of the lemma.  
}
\end{proof}

\begin{lem}\label{lem::repsphere}
Let $V$ be a $\G$-representation with $V^\G = 0$.  Then the kernel of the
map of $\ground$-modules $i^*\colon \hr^*(S^V)\to \hr^*(S^0)$ induced by
$S^0\hookrightarrow S^V$ is given by
\[ (\ker i^*)^\alpha = \begin{cases}
(\im\mft)^\alpha & |\alpha| = |V| \\
\hr^\alpha(S^V) & 0 < |\alpha| < |V| \text{ and } |\alpha^\G| > 0\\
0 = (\im\mft)^\alpha & \text{otherwise.}
\end{cases}
\]

\begin{figure}
\[ 
\shiftedpic{}{\framebox{$\ulhr^*(S^V)$}}{\tw{A}{d}}{\r}{\l}{\jg{\ground}}{\jg{\ground/p}}{\ground/p}
\hspace{10pt}
\kernelpic{}{\framebox{$\ker(\ulhr^*(S^V)(\trivo)\to\ulhr^*(S^0)(\trivo))$}}{\ground}{p\ground}{\ground}{\ground}{\ground/p}{\ground/p}
\]
\caption[\lofspace{}The kernel of $\ulhrcapt^*(S^V)(\trivo)\to\ulhrcapt^*(S^0)(\trivo)$]{$\ulhrnoph^*(S^V)$ and the kernel of
$\ulhrnoph^*(S^V)(\trivo)\to\ulhrnoph^*(S^0)(\trivo)$, shown for $|V| = 4$
and $p>2$.}
\label{fig::kernel}
\end{figure}

\end{lem}

The kernel of $i^*$ is shown pictorially for odd primes $p$ in
\autoref{fig::kernel}.  Note that, for $|\alpha| \ne |V|$, the kernel is
nonzero only in odd dimensions.

\begin{proof}
We know from the proofs of \autoref{lem::rotsphere} that, for a nontrivial
irreducible representation $\rotrep$ and $\alpha\in RO(\G)$,
$S^0\hookrightarrow S^\rotrep$ induces the map
\begin{align*}
\hr^{\alpha+\rotrep}(S^\rotrep) & \cong \hr^{\alpha}(S^0) \twoheadrightarrow
(\coker\mft)^{\alpha} \cong \\
& \cong (\coker\rho^*)^{\alpha+\rotrep-1}\oplus
(\ker\rho^*)^{\alpha+\rotrep} \twoheadrightarrow (\ker\rho^*)^{\alpha+\rotrep} \hookrightarrow
\hr^{\alpha+\rotrep}(S^0),
\end{align*}
whose kernel is as described in the statement of the lemma.  In particular,
if $|\alpha|$ is even, the kernel of the map
$\hr^{\alpha}(S^0)\to \hr^{\alpha+\rotrep}(S^0)$ is $(\im \mft)^{\alpha}$,
and the image is $(\ker \mfr)^{\alpha+\rotrep}$.  If $|\alpha|$ is odd,
then the map is a surjection onto the target.

Now consider any representation $V$ with $V^\G = 0$, so we can write
\[V = \rotrep_1 \oplus \rotrep_2 \oplus \cdots \oplus \rotrep_\ell,\]
where each $\rotrep_i$ is a nontrivial irreducible representation of $\G$.
Then the sequence of inclusions $S^0\hookrightarrow S^{\rotrep_1}
\hookrightarrow \cdots \hookrightarrow
S^{\rotrep_1\oplus\cdots\oplus\rotrep_\ell}$ induces maps 
\[ \hr^\alpha(S^0) \to \hr^{\alpha+\rotrep_{\ell}}(S^0) \to \cdots \to
\hr^{\alpha+\rotrep_{\ell}+\cdots+\rotrep_1}(S^0)
\]
where each arrow is a map of the form described above.  The description of
the kernel then follows from the description above and the fact that
$(\im\mft) \cap (\ker\mfr) = 0$ in $\ulhr^*(S^0)$.
\end{proof}

We are finally in a position to prove the more general statement of
\autoref{lem::crucial}.  As a reminder, the claim is that, when $X = \G_+
\sm S^m$ or $X = S^V$ for a $\G$-representation $V$,
\[ \rho^* \oplus i^* \colon \ulhr^\alpha(X) \to
\ulhr^\alpha(X)_\freeo \oplus \ulhr^\alpha(X^\G) \]
is a monomorphism, and 
\[ \rho^* \oplus \hat{i}^* \colon \ulhr^\alpha(X) \to
\ulhr^\alpha(X)_\freeo \oplus \left( \jg{\ground}\boxp \ulhr^\alpha(X^\G) \right) \]
becomes one after inverting $p$ in $\ground$.

\begin{proof}[Proof of \autoref{lem::crucial}]\label{pf::crucial}
For spheres $\G_+\sm S^{m}$, $\rho^*$ is already a monomorphism, and
remains one after tensoring with $\ground[1/p]$.  The result follows.

Next consider spheres $S^V$.  \autoref{lem::freeolevel} shows that
$\rho^*\oplus i^*$ and $\rho^*\oplus\hat{i}^*$ are monomorphisms at the
$\freeo$ level, so we restrict attention to the $\trivo$ level.
Suspension by a trivial sphere preserves
fixed points, and so we may assume without loss of generality that $V^\G =
0$ and so $(S^V)^\G = S^0$.

We know that, for any $\alpha$, the kernel of $\hr^\alpha(S^V)\to
\hr^\alpha(\G_+\sm S^V)$ is exactly the kernel of $\mfr\colon
\ulhr^\alpha(S^V)(\trivo)\to \ulhr^\alpha(S^V)(\freeo)$; for $|\alpha| \ne
|V|$, $\mfr$ is the zero map on $\ulhr^\alpha(S^V)$.  From
\autoref{lem::repsphere}, we know that the kernel of $i^*$ is the image of
$\mft$ in these dimensions.  As in \autoref{lem::rotsphere}, it follows
that $(\ker \rho^*\oplus i^*)^\alpha = 0$ when $|\alpha| = |V|$, since
$(\im\mft)\cap (\ker\mfr) = 0$ in $\ulhr^*(S^V)$.  In these dimensions,
$\im\mft = 0$ in the image, and so $(\ker\rho^*\oplus \hat{i}^*)^\alpha =
0$ as well.

For $|\alpha|\ne |V|$ of even dimension, $i^*$ is a monomorphism, and so
$\rho^*\oplus i^*$ is as well.  The kernel of $i^*$ disappears for all
$|\alpha| \ne |V|$ after inverting $p$, and the image of $i^*$ is contained
in the kernel of $\mfr$, which does not intersect $\mft$.  It follows that
$\mfr\oplus \hat{i}^*$ becomes a monomorphism in all dimensions after
tensoring with $\ground[1/p]$.
\end{proof}

\section{Cohomology of Complex Projective Spaces}\label{sec::cxprojmult}

We have already seen in \autoref{cor::cxproj} that, as a module over
$\ulhr^*(S^0)$, the cohomology $\ulhr^*({\cpv{\cxuniverse}}_+)$ of the
complex projective space on a complete universe is free on generators
corresponding to the cells $D(\omega_N)$ in the equivariant Schubert cell
decomposition; $\omega_N$ is the complex representation
\[\omega_N = \phi^{-N}(1\oplus\phi\oplus\cdots\oplus\phi^{N-1}),\]
where $\phi$ has underlying real representation $\rotrep_1$.  For ease of
notation, we will also use $\omega_N$ to denote the underlying real
representation of $\omega_N$, relying on context  to determine whether a
real or complex representation is needed.  

In the remainder of this section we will sketch how to use
\autoref{thm::multcomparison} to find the multiplicative structure of
$\ulhr^*(\cpv{\cxuniverse}_+)$, since it will be relevant to the computation of
$\ulhr^*(B_\G O(2)_+)$ in \autoref{ch::computations}.  We will use the
Dress pairing of \autoref{lem::dresspairing} and \autoref{lem::pboxmaps}
and focus on $\ulhr^*(\cpv{\cxuniverse}_+)(\trivo) =
\hr^*(\cpv{\cxuniverse}_+)$, since the nonequivariant cohomology
$\ulhr^*(\cpv{\cxuniverse}_+)(\freeo)$ is known.  Further, since we are
working with unreduced cohomology in this section, we will write
$\h^*(\cpv{\cxuniverse})$ rather than $\hr^*(\cpv{\cxuniverse}_+)$ and
$\h^*(\pt)$ rather than $\hr^*(S^0)$.
Full details of all results presented below, for
$\h^*(\cpv{\cxuniverse})$ and for the more general computation of
$\h^*(\cpv{V})$ for any $\G$-representation $V$, can be found in
\cite{lgl}.

We would first like to describe the maps $\rho^*\oplus i^*$ and
$\rho^*\oplus \hat{i}^*$ of \autoref{thm::multcomparison} explicitly; we
will start by examining $\cpv{\cxuniverse}^\G$.  A point of
$\cpv{\cxuniverse}$ is a line in $\cxuniverse$ through the origin and a
point $x = (x_0,x_1,\ldots,x_{N-1},1,0,\ldots)$.   This line will lie in
$\cpv{\cxuniverse}^{\G}$ if $gx$ is a $\bC$-scalar multiple of $x$.  It is
clear that this happens only when all nonzero coordinates $x_j$ come from
isomorphic representations $\phi^j$.  Since there are $p$ irreducible
complex representations of $\G$, it follows that 
\[ \cpv{\cxuniverse}^\G \simeq \coprod_{N=0}^{p-1} \cpv{(\phi^N)^{\oplus
\infty}} = \coprod^{p} \bC P^\infty.\]
Thus $\h^*(\cpv{\cxuniverse}^\G) \cong \bigoplus_{p}\h^*(\bC
P^\infty)$.  Since finite products and coproducts of
$\ground$-modules agree\footnote{The same is true of Mackey functors.}, the
map ${i^*\colon \h^*(\cpv{\cxuniverse})\to \h^*(\cpv{\cxuniverse}^\G)}$
is determined by $p$ maps ${i_N^*\colon \h^*(\cpv{\cxuniverse})\to
\h^*(\cpv{(\phi^N)^{\oplus\infty}})}$; and so $\rho^*\oplus i^*$ is
determined by $\rho^*$ together with the maps $i^*$.  In this context,
\autoref{thm::multcomparison} tells us that, for even-dimensional $\alpha$,
two elements ${x,y\in \h^\alpha(\cpv{\cxuniverse})}$ are the same if and
only if $\rho^*(x) = \rho^*(y)$ and $i_N^*(x) = i_N^*(y)$ for each $0\le
N\le {p-1}$.  An analogous statement holds for every $\alpha$ and
$\hat{i}^*$ after tensoring with $\ground[1/p]$.

Since we know the multiplicative structures of the targets of $\rho^*$,
$i^*$, and $\hat{i}^*$, we will be able to identify the products of
generators in $\h^*(\cpv{\cxuniverse})$ once we have computed
the values of $\rho^*$ and $i_N^*$ or $\hat{i}_N^*$ on these generators.

\begin{prop}[Lewis]\label{thm::cpvgenerators}
Let $i_N^*$ for $0\le N\le {p-1}$ be the components of ${i^*\colon
\h^*(\cpv{\cxuniverse})\to \h^*(\cpv{\cxuniverse}^\G)}$, as described
above.  Then there are elements ${C\in \h^{\omega_p}(\cpv{\cxuniverse})}$
and, for each $1\le j\le p-1$, $D_j\in
\h^{\omega_j}(\cpv{\cxuniverse})$, satisfying the following
conditions.
\begin{enumerate}
\item If we write $D_0 = 1$, then $\{D_j C^n\}_{0\le j\le p-1,\ n\ge 0}$
gives a complete set of generators of $\h^*(\cpv{\cxuniverse})$ as a
module over $\h^*(\pt)$.

\item Given a chosen generator $z$ of $\ulh^*(\cpv{\cxuniverse})(\freeo)
\cong \neqh^{|*|}(\cpv{\cxuniverse})$,
\[ \rho^*(D_j) = z^j \text{ and } \rho^*(C) = z^p.\]

\item Let $\h^*(\cpv{(\phi^N)^{\oplus\infty}}) \cong \ground[z_N] \otimes
\h^*(\pt)$.  Using brackets $[-]$ to denote the image of an element of
$\h^*(-)$ in $\h^*(-)/{(\im \mft)}$, we have the
following.  For each ${0\le N\le p-1}$, 
\[ \hat{i}_N^* (D_j) = [ d_{N,j} \epsilon_{\omega_j} ] 
\text{ and } \hat{i}_N^* (C) = [ \epsilon_{\omega_{p-1}} z_N ] \]
where $d_{N,j}$ is given by
\[ d_{N,j} = \begin{cases}
0 & N < j \\
1 & N = j \\
\prod_{i=0}^{j-1} d\left(\rotrep_{|i-N|}-\rotrep_{|i-j|}\right) & N > j
\end{cases}
\]
\end{enumerate}
\end{prop}

\begin{proof}
See \cite[Construction 6.1]{lgl}.  The idea is to identify appropriate generators
in the cohomology of certain small finite projective spaces, and then pass
to the projective spaces of interest by attaching cells and looking at the
corresponding long exact sequences in cohomology.
\end{proof}

\autoref{thm::cpvgenerators} puts us in a position to completely describe
the multiplicative structure of $\h^*(\cpv{\cxuniverse})$.  Since the
$\{ D_j C^n\}$ give a complete set of additive generators, dimensional
considerations show that $(D_j C^n) (D_{k} C^{m})$ must be a
$\h^*(\pt)$-linear combination of the elements $D_0 C^{n+m}$
through $D_{p-1} C^{n+m}$ and, if $j+k\ge p$, of $D_0 C^{n+m+1}$ through
$D_{j+k-p} C^{n+m+1}$.  After quotienting by $\im\mft$ and inverting $p$,
the contributions from the $D_j C^{n+m+1}$ vanish.  Using
\autoref{thm::cpvgenerators} and the fact that each $\hat{i}_N^*$ is a map
of algebras, this leaves us with $p$ equations in $p$ unknowns, which can
be used to determine the coefficients of $D_0 C^{n+m}$ through $D_{p-1}
C^{n+m}$.  We can then go through a similar process using $i_N^*$ to
determine the coefficients of $D_0 C^{n+m+1}$ through $D_{j+k-p}
C^{n+m+1}$.

Unfortunately, for odd primes, there does not appear to be a good closed
form for these coefficients, since their values depend on our noncanonical
choices for the $d(\alpha)$, which in turn depend on the ``mod $p$ inverses''
$\text{inv}(j)$.  Even with $\ground = \bQ$, there is some difficulty,
since $\fp$ is not a subfield of $\bQ$.
Thus we will simply say the following.

\begin{thm}[Lewis]\label{thm::multcpv}
As a graded commutative algebra over $\h^*(\pt)$, $\h^*(\cpv{\cxuniverse})$ is
generated by elements $C$ in dimension $\omega_p$ and $D_j$ in dimension
$\omega_j$ for each $1\le j\le p-1$.  $C$ generates a polynomial subalgebra
of $\h^*(\cpv{\cxuniverse})$, and a complete set of additive generators of
$\h^*(\cpv{\cxuniverse})$ is given by the elements $\{ D_j C^n \}_{0\le
j\le p-1,\ n\ge 0}$.  
If $j\le k \le p-1$, then each product $D_j D_k = D_k D_j$ is given by a linear
combination over $\h^*(\pt)$ of the elements $D_j$ through $D_{j+k}$, if
$j+k\le p-1$, and of $D_j$ through $D_{p-1}$ and $D_0 C = C$ through
$D_{j+k-p} C$ if $j+k\ge p$.
\end{thm}


However, when $p=2$, the issues with $d(\alpha)$ vanish, and we have the
following description of the algebra structure.

{\renewcommand{\G}{{C_2}}
\begin{thm}[Lewis]\label{thm::2cpv}
As a graded commutative algebra over $\h^*(\pt)$,
$\h^*(\cpv{\cxuniverse})$ is generated by elements $D_1 \in
\h^{2\signrep}(\cpv{\cxuniverse})$ and $C\in
\h^{2(1+\signrep)}(\cpv{\cxuniverse})$, satisfying the single
relation
\[ D_1^2 = \epsilon^2 D_1 + \xi C.\]
\end{thm}
}

\chapter{Spectral sequences for local coefficients}\label{ch::ss}
\section{The nonequivariant situation}\label{sec::noneq}

\renewcommand{\G}{G}

Let $X$ be a path-connected based space with universal cover $\tilde X$.
Let $\pi = \pi_1(X)$ and let $\pi$ act on the right of $\tilde{X}$ by deck
transformations.  As always, fix a ground ring $\ground$.  Let $M$ be a
left and $N$ be a right module over the group ring $\ground[\pi]$.  Let
$C_*$ be the normalized singular chain complex functor with coefficients in
$\ground$.

\begin{defn}\label{def::stdloc}  Define the homology of $X$ with coefficients in
$M$ by 
\[  \neqh_*(X;M) = \neqh_*(C_*(\tilde X)\otimes_{\ground[\pi]} M). \]
Define the cohomology of $X$ with coefficients in $N$ by \ \ 
\[ \neqh^*(X;N) = \neqh^*(\Hom_{\ground[\pi]}(C_*(\tilde X),N)). \]
\end{defn}

Functoriality in $M$ and $N$ for fixed $X$ is clear.  For a based map
$f\colon X\to Y$, where $\pi_1(Y) = \rho$, and for a left $\ground[\rho]$-module
$P$, we may regard $P$ as \rarticle $\ground[\pi]$-module by pullback along $\pi_1(f)$,
and then, using the standard functorial construction of the universal
cover, we obtain
\[ f_*\colon \neqh_*(X;f^*P)\to \neqh_*(Y;P). \]
Cohomological functoriality is similar.  The definition goes back to
Eilenberg \cite{Eil}, and has the homology of spaces and the homology of
groups as special cases, as discussed below.  It deserves more emphasis
than it is usually given because it implies spectral sequences for the
calculation of homology and cohomology with local coefficients, as we shall
recall. 

\begin{exmp}  If $\pi$ acts trivially on $M$ and $N$, then $\neqh_*(X;M)$ and
$\neqh^*(X;N)$ are the usual homology and cohomology groups of $X$ with
coefficients in $M$ and $N$.  We can identify $C_*(X)$ with
$C_*(\tilde{X})\otimes_{\ground[\pi]}\ground$, where $\ground[\pi]$ acts trivially on $\ground$.
This implies the identifications 
\[C_*(\tilde X)\otimes_{\ground[\pi]} M\cong C_*(X;M)\hspace{12pt}
\text{and}\hspace{12pt} \Hom_{\ground[\pi]}(C_*(\tilde X),N) \cong C^*(X;N). \]
\end{exmp}

\begin{exmp}  If $X = K(\pi,1)$, then $\neqh_*(X;M)$ and $\neqh^*(X;N)$ are the
usual homology and cohomology groups of $\pi$ with coefficients
in $M$ and $N$ since $C_*(\tilde{X})$ is \rarticle $\ground[\pi]$-free resolution of $\ground$.
That is,
\[  \neqh_*(K(\pi,1);M) = \Tor^{\ground[\pi]}_*(\ground,M) \hspace{12pt} 
\text{and}\hspace{12pt} \neqh^*(K(\pi,1);M) = \Ext_{\ground[\pi]}^*(\ground,N). \]
\end{exmp}

\begin{exmp}  If $M =\ground[\pi]\otimes_\ground C$ and $N =
\Hom_\ground(\ground[\pi],C)$ for \rarticle $\ground$-module $C$, then
\[\neqh_*(X;M)\cong \neqh_*(\tilde{X},C)\hspace{12pt} \text{and}\hspace{12pt}
\neqh^*(X;N) \cong \neqh^*(\tilde{X};C). \]
\end{exmp}

\begin{rem}  If we replace $N$ by $M$ 
(viewed as a right $\ground[\pi]$-module) in the cohomology 
case of the previous example, then we are forced to 
impose finiteness restrictions and consider cohomology
with compact supports; compare \cite[3H.5]{Hat}.
\end{rem}

We have spectral sequences that generalize the last two examples.  When
$\pi$ acts trivially on $M$ and $N$, they can be thought of as versions of
the Serre spectral sequence of the evident fibration $\tilde{X}\to
X\to K(\pi,1)$. 

\begin{thm}[Eilenberg Spectral Sequence]\label{thm::CE}  There are spectral sequences
\[  E^2_{p,q} = \Tor^{\ground[\pi]}_{p,q}(\neqh_*(\tilde{X}),M) \Longrightarrow  \neqh_{p+q}(X;M)  \]
and
\[  E_2^{p,q} = \Ext_{\ground[\pi]}^{p,q}(\neqh_*(\tilde{X}),N) \Longrightarrow  \neqh^{p+q}(X;N).  \]
\end{thm}

Up to notation, these are the spectral sequences given by Cartan and
Eilenberg in \cite[p.\ 355]{CE}.

\begin{proof}  In the $E^2$ and $E_2$ terms, $p$ is the homological degree
and $q$ is the internal grading on $\neqh_*(\tilde{X})$.  Let $\varepsilon\colon
P_*\to M$ be \rarticle $\ground[\pi]$-projective resolution of $M$ and form the
bicomplex
\[ C_*(\tilde{X})\otimes_{\ground[\pi]} P_*;\]
the theorem comes from looking at the two spectral sequences associated to
this bicomplex and converging to a common target.

If we filter $C_*(\tilde{X})\otimes_{\ground[\pi]} P_*$ by the degrees of
$C_*(\tilde{X})$, we get a spectral sequence whose $E^0$-term has
differential $\id\otimes d$.  Since $C_*(\tilde{X})$ is a projective
$\ground[\pi]$ module, the resulting $E^1$-term is
$C_*(\tilde{X})\otimes_{\ground[\pi]}M$, the resulting $E^2$-term is $\neqh_*(X;M)$,
and $E^2=E^{\infty}$.  Since $E^\infty$ is concentrated in degree $q=0$,
there is no extension problem; we have identified the target as claimed in
the theorem.

Filtering the other way, by the degrees of $P_*$, we obtain a spectral
sequence whose $E^0$-term has differential $d\otimes \id$.  The resulting
$E^1$-term is $\neqh_*(\tilde{X})\otimes_{\ground[\pi]}P_*$ and the resulting
$E^2$-term is $\Tor^{\ground[\pi]}_{*,*}(\neqh_*(\tilde{X}),M)$.  This gives the
first statement of the theorem.

The argument in cohomology is similar, starting from the bicomplex
\[\Hom_{\ground[\pi]}(C_*(\tilde{X}),I^*)\]
for an injective resolution $\eta\colon N\to I^*$ of $N$.
\end{proof}

We record an immediate corollary.

\begin{cor}\label{cor::CE}
Let $\pi$ be a finite group of order $n$ and $\ground$ be a field of
characteristic prime to $n$.  Then
\[ \neqh_*(X;M) \cong \neqh_*(\tilde{X})\otimes_{\ground[\pi]}M
\hspace{12pt} \text{and}\hspace{12pt} 
\neqh^*(X;N) \cong \Hom_{\ground[\pi]}\left(\neqh_*(\tilde{X}),N\right). \]
\end{cor}

\begin{proof}
Since $\ground[\pi]$ is semi-simple, $E_{p,q}^2 = 0$ and $E_2^{p,q} = 0$ for
$p>0$.  Therefore the spectral sequences collapse to the claimed
isomorphisms.
\end{proof}

If $\pi$ acts trivially on $\neqh_*(\tilde{X})$, so that $\neqh_*(\tilde{X}) \cong
\neqh_*(\tilde{X})\otimes_{\ground[\pi]} \ground$, then the situation simplifies even
further.  For ease of notation, let $M_\pi$ denote the coinvariants
$M/{IM}$, where $I\subset  \ground[\pi]$ is the augmentation ideal, and let
$N^\pi$ denote the fixed points of $N$.

\begin{cor}\label{cor::cor::CE}
Suppose that we are in the situation of \autoref{cor::CE} and that $\pi$
acts trivially on $\neqh_*(\tilde{X})$.  Then

\vspace{0.5ex}
\hfill$\square$
\vspace{-7ex}
\[ \neqh_*(X;M) \cong \neqh_*(\tilde{X};M_\pi) \hspace{12pt} \text{and}\hspace{12pt}
\neqh^*(X;N) \cong \neqh^*(\tilde{X};N^\pi).\]
\end{cor}

It remains to identify the homology and cohomology groups of
\autoref{def::stdloc} with classical (co)homology with local coefficients.
To do this, we first need to reconcile our coefficient $\ground[\pi]$-modules
with the classical definition of a local coefficient system.

In \autoref{def::stdloc}, we took $M$ and $N$ to be left and right modules
over the group ring $\ground[\pi]$ and took $C_*(\tilde{X})$ to be the normalized
singular chains of $\tilde{X}$.  A (left or right) $\ground[\pi]$-module $M$ is
the same as a (covariant or contravariant) functor from $\pi$, viewed as a
category with a single object, to the category of $\ground$-modules.  

As usual, given a space $X$, let $\Pi X$ be the fundamental groupoid of
$X$; this is a category whose objects are the points of $X$ and whose
morphism sets are homotopy classes of paths between fixed endpoints.
By definition, a \defword{local coefficient system} $\sM$ on $X$ is
a functor (covariant or contravariant depending on context, corresponding
to our left and right $\ground[\pi]$-module distinction above) from the
fundamental groupoid $\Pi X$ to the category of $\ground$-modules.   When
$X$ is path-connected with basepoint $x_0$, the category $\pi =\pi_1(X)$
with single object $x_0$ is a skeleton of $\Pi X$.  Therefore a coefficient
system $\sM$ is determined by its restriction $M$ to $\pi$.  

Whitehead \cite[VI.3.4 and 3.4*]{GW} (see
also Hatcher \cite[3H.4]{Hat}) proves the following result and ascribes it 
to Eilenberg \cite{Eil}.  

\begin{thm}[Eilenberg]\label{thm::comparison}
For path-connected spaces $X$ and covariant and contravariant local
coefficient systems $\sM$ and $\sN$ on $X$, the classical homology
and cohomology with local coefficients $\neqh_*(X;\sM)$ and $\neqh^*(X;\sN)$
are naturally isomorphic to the homology and cohomology groups $\neqh_*(X;M)$
and $\neqh^*(X;N)$, where $M$ and $N$ are the restrictions of $\sM$ and $\sN$
to $\pi$.  \end{thm}

Therefore \autoref{thm::CE} gives a way to compute the additive structure of
homology and cohomology with local coefficients.  In particular, if
$f\colon E\to X$ is a fibration with fiber $F$ and path-connected base
space $X$, it gives a means to compute the homology and cohomology with
local coefficients that appear in
\[ E_{*,*}^2 = \neqh_*(X;\sH_*(F;\ground))\hspace{12pt} \text{and}\hspace{12pt} 
E_2^{*,*} = \neqh^*(X;\sH^*(F;\ground)) \]
of the Serre spectral sequences for the computation of $\neqh_*(E ; \ground)$ and
$\neqh^*(E ; \ground)$.

Even the case when $\pi$ is finite of order $n$ and $\ground$ is a field of
characteristic prime to $n$ often occurs in practice.  More generally, the
spectral sequences of \autoref{thm::CE} help make the Serre spectral
sequence amenable to explicit calculation in the presence of non-trivial
local coefficient systems. 

Note that we have not yet addressed the multiplicative structure of the
cohomological Eilenberg spectral sequence; this will be discussed in
\autoref{sec::mult}.  However, even without the multiplicative structure, we
can already do one example.  It will be useful to have the following simple
consequence of \autoref{def::stdloc}.

\begin{prop}\label{prop::hzero}
Let $X$ be a space and $\pi$ its fundamental group.  For any
$\ground[\pi]$-module $M$, $\neqh^0(X;M) \cong M^\pi$, the $\pi$-fixed points of $M$.
If $M$ is \rarticle $\ground[\pi]$-algebra, the isomorphism is as algebras.
\end{prop}

\begin{proof}
We would like to identify the kernel of
\[ \Hom_{\ground[\pi]}(C_0(\tilde{X}),M) \to \Hom_{\ground[\pi]}(C_1(\tilde{X}),M).\]
Since $\tilde{X}$ is connected, $\neqh_0(\tilde{X}) \cong \ground$ with the trivial
$\pi$ action.  Since $\Hom_{\ground[\pi]}$ is left exact, it follows that
$\neqh^0(X;M) \cong \Hom_{\ground[\pi]}(\ground,M) \cong M^\pi$, as claimed.
\end{proof}

In particular, in the Serre spectral sequence, $E_2^{0,t}
\cong \sH^t(F;\ground)^\pi$.  If we are in a situation where $E_2^{s,t}$ vanishes
for $s>0$, then \autoref{prop::hzero} completely describes the
multiplicative structure on the $E_2$ page.

%
%

\begin{exmp}\label{ex::noneq}
Let $\cyctwo$ be the cyclic group of order two, identified with $\{\pm 1\}$
when convenient; we write $\cyctwo$ instead of $C_2$ because we will later
want to distinguish the role of structural groups and groups of
equivariance.  Consider the fibration 
\[B\text{det}\colon BO(2)\to B\cyctwo\] 
with fiber $BSO(2)$.  Take coefficients in a finite field $\ground = \fq$ with
$q$ an odd prime, so that $\ground[\pi]$ is semisimple.  The action of the base
on the fiber is nontrivial; the Serre spectral sequence for this fibration
has
\[ E_2^{s,t} = \neqh^s(B\cyctwo;\sH^t(BSO(2);\fq)) \Longrightarrow \neqh^{s+t}(BO(2);\fq)
\]
We know that $B\cyctwo \simeq \bR P^\infty$ with universal cover $S^\infty
\simeq \pt$, so \autoref{cor::cor::CE} applies, and $BSO(2) \simeq \bC
P^\infty$.  We also know $\neqh^*(\bC P^\infty;\fq) \cong \fq[x]$, a polynomial
algebra on one generator $x$ in degree two, and that the fundamental group
$\pi_1(B\cyctwo) \cong \cyctwo$ acts on $\neqh^*(\bC P^\infty)$ by $x\mapsto
-x$. 

By \autoref{cor::cor::CE}, we thus have
\[ E_2^{s,t} \cong \neqh^s \left(S^\infty;\sH^t(BSO(2);\fq)^{\cyctwo} \right) \]
in the Serre spectral sequence; $\sH^t(BSO(2);\fq)^{\cyctwo}$ is either $0$
or $\fq$, depending on $t$.  $E_2^{s,t}$ thus vanishes for $s>0$, so the Serre
spectral sequence collapses with no extension problems.  Using the
observation after \autoref{prop::hzero}, we see
\[ \neqh^*(BO(2);\fq) \cong \sH^*(BSO(2);\fq)^{\cyctwo} \cong \fq[x^2].\]
The isomorphism is of $\ground$-algebras.  We have thus shown the well-known fact
that $\neqh^*(BO(2);\fq)$ is polynomial on one generator (the Pontrjagin class)
in degree four.  We will later examine an equivariant version of
this example.
\end{exmp}


\section{Equivariant generalizations}\label{sec::eq}

Heading towards an equivariant generalization of \autoref{thm::CE}, we first
rephrase the definition of (co)homology with local coefficients. 
In \autoref{sec::noneq},
we effectively defined homology and cohomology with local coefficients by
restricting a local coefficient system $\sM\colon \Pi X \to \rmod$ to
\rarticle $\ground[\pi]$-module $M\colon \pi \to \rmod$.

Rather than restricting $\sM$ to $\pi$, we could instead redefine
$\tilde{X}$ to be the universal cover functor $\Pi X\op \to \spaces$ that
sends a point $x\in X$ to the space $\tilde{X}(x)$ of equivalence classes
of paths starting at $x$ and sends a path $\gamma$ from $x$ to $y$ to the
map $\tilde{X}(y)\to \tilde{X}(x)$ given by precomposition with $\gamma$.
Since $\pi$ is a skeleton of $\Pi X$, the following definition is
equivalent to \autoref{def::stdloc} when $X$ is connected.  By
\autoref{thm::comparison}, there is no conflict with the classical notation
for homology with local coefficients.  Let $\chr$ denote the category of
chain complexes of $\ground$-modules.

\begin{defn}[Reformulation of \autoref{def::stdloc}]\label{def::altloc} 
Let $\sM\colon \Pi X\to \rmod$ and $\sN\colon \Pi X\op\to \rmod$ be
functors and let $C_*(\tilde{X})\colon \Pi X\op\to \spaces \to \chr$ be the
composite of the universal cover functor with the functor $C_*$.  Define
the homology of $X$ with coefficients in $\sM$ to be \[ \neqh_*(X;\sM) =
\neqh_*(C_*(\tilde{X}) \otimes_{\Pi X} \sM) \] where $\otimes_{\Pi X}$ is the
tensor product of functors (which is given by an evident coequalizer
diagram).  Similarly, define 
\[ \neqh^*(X;\sN) = \neqh^*\left(\Hom_{\Pi X}(C_*(\tilde{X}),\sN)\right) \]
where $\Hom_{\Pi X}$ is the hom of functors (also known as natural
transformations; alternatively, given by an evident equalizer diagram).
\end{defn}


Note that our distinctions between left and right and between covariant and
contravariant are purely semantic above, since we are dealing with groups
and groupoids.  However, we are about to consider (Bredon) equivariant
homology and cohomology. Here the fundamental ``groupoid'' is only an
EI-category (endomorphisms are isomorphisms) and the distinction is
essential. There is an equivariant Serre spectral sequence, due to Moerdijk
and Svensson \cite{MS}, but it has not yet had significant calculational
applications. The essential reason is the lack of a way to compute its
$E^2$-terms.  However, the results of \autoref{sec::noneq} generalize nicely
to compute Bredon homology and cohomology with local coefficients.


\autoref{def::altloc} generalizes directly to the equivariant case.  Recall
our discussion of Bredon cohomology in \autoref{sec::bredon}; for this
subsection, we will think about integer-graded Bredon cohomology, which has
coefficient systems rather than Mackey functors as coefficients.  From now
on, let $X$ be a $G$-space; as in the rest of this \paper{}, we take $G$ to
be a finite group.\footnote{The results in this subsection apply equally
well to discrete groups, and with a little more detail, we could generalize
to topological groups.}  Following tom Dieck \cite{tD}, we can define the
\defword{fundamental EI-category} $\fund{X}$ to be the category whose
objects are pairs $(K,x)$, where $x \in X^K$; a morphism from $(K,x)$ to
$(L,y)$ consists of a $G$-map $\alpha\colon G/K\to G/L$, determined by
$\alpha(eH) = gK$, together with a homotopy class rel endpoints $[\gamma]$
of paths from $x$ to $\alpha^*(y) = gy$.  Here $\alpha^*\colon X^L\to  X^K$ is
the map given by $\alpha^*(z)=gz$, which makes sense because $g^{-1}Hg\subset
L$.

Likewise, we follow tom Dieck in defining the equivariant universal cover
$\tilde{X}$ to be the functor $\tilde{X}\colon (\fund{X})\op \to \spaces$
which sends $(K,x)$ to $\widetilde{X^K}(x)$, the space of equivalence
classes of paths in $X^K$ starting at $x$.  For a morphism $\fmap{}\colon
(K,x)\to (L,y)$, $\tilde{X}\fmap{}\colon \tilde{X}(L,y)\to \tilde{X}(K,x)$
takes a class of paths $[\beta]$ starting at $y\in X^L$ to the class of the
composite $(\alpha^*\beta)\gamma$.

We can now define equivariant (co)homology with local coefficients.  In fact, 
\autoref{def::altloc}
applies almost verbatim: we need only add $G$ to the notations.  We repeat
the definition for emphasis. 

\begin{defn}[Equivariant generalization of \autoref{def::altloc}]\label{def::eqloc}
Let $X$ be a $G$-space and write $\Pi = \fund{X}$.  Let $\sM\colon \Pi\to
\rmod$ and $\sN\colon \Pi\op\to \rmod$ be functors and let
$\eqchains{\tilde{X}}\colon \Pi\op\to \spaces \to \chr$ be the composite of
the equivariant universal cover functor with the functor $C_*$.  Define the
homology of $X$ with coefficients in $\sM$ to be
\[ \hl_*(X;\sM) = H_*(\eqchains{\tilde{X}} \otimes_\Pi \sM) \]
and the cohomology of $X$ with coefficients in $\sN$ by
\[ \h^*(X;\sN) = H^*\left(\Hom_\Pi(C^G_*(\tilde{X}),\sN)\right). \]
\end{defn}

Note that we could also take $\Pi$ to be a skeleton
$\text{skel}(\fund{X})$. 

Inserting $G$ into the notations, the proofs in Whitehead or Hatcher
\cite{GW, Hat} apply to show that this definition of Bredon (co)homology
with local coefficients is naturally isomorphic to the Bredon (co)homology
with local coefficients, as defined in Mukherjee and Pandey \cite{MP},
which they in turn show is naturally isomorphic to the (co)homology with
local coefficients, as defined and used by Moerdijk and Svensson in
\cite{MS} to construct the equivariant Serre spectral sequence of a
$G$-fibration $f\colon E\to B$.

We quickly review the homological algebra needed for the equivariant
generalization of \autoref{thm::CE}. Since $\rmod$ is an abelian category, the
categories $[\fundx{},\rmod]$ and $[\fundx{}\op,\rmod]$ of functors from
$\fundx{}$ to $\rmod$ are also abelian, with kernels and cokernels defined
levelwise.  These categories have enough projectives and injectives, which
by the Yoneda lemma are related to the represented functors.

Specifically, let $\ground-$ denote the free $\ground$-module functor
$\textbf{Set}\to\rmod$.  Given an object $(K,x)\in \Pi$, let
$\covproj{K,x}$ be the covariant represented functor $\Pi\to\rmod$
given on objects by
\[ \covproj{K,x}(L,y) = \ground\fundhom{(K,x)}{(L,y)}.\]
By the Yoneda lemma, each $\covproj{K,x}$ is projective.  Therefore, given
a functor $\sM$, we can construct an epimorphism $\sP\to\sM$ with $\sP$
projective by taking $\sP$ to be a direct sum of representables
\[ \sP = \bigoplus_{(K,x)}\bigoplus_{\sM(K,x)}\covproj{K,x},\]
one for each element of each $\ground$-module $\sM(K,x)$.  Similarly, there are
contravariant represented functors $\contraproj{K,x}\colon \Pi\op\to\rmod$
given by
\[ \contraproj{K,x}(L,y) = \ground\fundhom{(L,y)}{(K,x)}. \]
The same argument shows that these are projective and that $[\Pi\op,\rmod]$
has enough projectives.

The construction of the injective objects is dual but perhaps less
familiar.  Given \rarticle $\ground$-module $C$ and $(K,x)\in\Pi$, we define a functor
$\covinj{K,x,C}\colon \Pi\to\rmod$ by
\[ \covinj{K,x,C}(L,y) = \Hom_\ground(\contraproj{K,x}(L,y),C).\]
Whenever $C$ is an injective $\ground$-module, $\covinj{K,x,C}$ is an injective
object in $[\Pi,\rmod]$.  This comes from a more general fact.  
For any coefficient system $\sA\colon \Pi\to\rmod$, there is a tensor-hom
adjunction
\[ [\Pi,\rmod](\sA,\covinj{K,x,C}) \cong \rmod(\sA\otimes_\Pi
\contraproj{K,x},C) \]
where again $\otimes_\Pi$ is the tensor product of functors.  The tensor
product of any functor with a representable functor $\contraproj{K,x}$ is
given by evaluation at $(K,x)$.  Putting these two facts together, we have
that a natural transformation from $\sA$ to $\covinj{K,x,C}$ is given by
the same data as a homomorphism of $\ground$-modules from $\sA(K,x)$ to $C$.  It
is then clear that, if $C$ is an injective $\ground$-module, $\covinj{K,x,C}$
must be an injective object of $[\Pi,\rmod]$, as desired.  Given any
$\sN\colon \Pi\to\rmod$, we can construct an injective coefficient system
$\sI$ and a monomorphism $\sN\to\sI$ as follows.  Choose monomorphisms
$\sN(K,x) \hookrightarrow C_{K,x}$ for each $(K,x)$ with $C_{K,x}$
injective, and define $\sI$ to be the product of injective functors
\[ \sI = \prod_{(K,x)} \covinj{K,x,C_{K,x}}. \]
It can be checked that the evident map $\sN\to\sI$ is a monomorphism.
Thus $[\Pi,\rmod]$ has enough injectives.  The functors 
$\contrainj{K,x,C} = \Hom_\ground(\covproj{K,x}(-),C)$ show that $[\Pi\op,\rmod]$
has enough injectives as well.

Finally, we define $\Tor^\fundx{}(\sN,\sM)$ in the obvious way.  It is the
homology of the complex of $\ground$-modules that is obtained by taking the
tensor product of functors of $\sN$ with a projective resolution of the
functor $\sM$.  We define $\Ext_\fundx{}(\sN_1,\sN_2)$ similarly, taking
the hom of functors of $\sN_1$ with an injective resolution of
$\sN_2$.\footnote{Alternatively, we could define $\Ext_{\fundx{}}(\sN_1,\sN_2)$
by taking a projective resolution of $\sN_1$; however, it is the definition
above which gives rise to the desired spectral sequence.}

The following equivariant analogue of the nonequivariant statement that
$C_*(\tilde{X})$ is a free $\ground[\pi]$-module should be a standard first
observation in equivariant homology theory, but the author has not seen it
in the literature.  The nonequivariant assertion, while obvious, is the
crux of the proof of \autoref{thm::CE}.  Let $\contraproj{K,x} =
\ground\fundhom{-}{(K,x)}$, as above.

\begin{lem}\label{lem::technical}
With $\Pi = \fund{X}$, each functor $C_n^G(\tilde{X})\colon \Pi\op \to
\rmod$ is a direct sum of representable functors
$\bigoplus_{(K_i,x_i)}\contraproj{K_i,x_i}$. 
\end{lem}

Granting this result for the moment, we can prove the equivariant
generalization of \autoref{thm::CE}.

\begin{thm}[Equivariant Eilenberg Spectral Sequence]\label{thm::eqCE}
With $\Pi = \fund{X}$, there are spectral sequences
\[ E_{p,q}^2 = \Tor_{p,q}^\fundx{}(\coh_*(\tilde{X}),\sM) \Longrightarrow
\hl_{p+q}(X;\sM) \]
and
\[  E_2^{p,q} = \Ext_{\fundx{}}^{p,q}(\coh_*(\tilde{X}),\sN) \Longrightarrow  \h^{p+q}(X;\sN).  \]
\end{thm}

Here the functor $\coh_*(\tilde{X})\colon \fundx{}\op\to \rmod$ is the
homology of the chain complex functor $\eqchains{\tilde{X}}$; that is,
$\coh_*(\tilde{X})(K,x)$ is the homology of the chain complex
$C_*(\tilde{X})(K,x)$.

\begin{proof}
Let $\varepsilon\colon \sP_*\to\sM$ be a projective resolution of $\sM$.  As in
the nonequivariant theorem, form the bicomplex of $\ground$-modules
$C_*(\tilde{X}) \otimes_\fundx{} \sP_*$.  Since the tensor product of a
functor with a representable functor is given by evaluation,
\[ \contraproj{K,x} \otimes_\fundx{} \sM \cong \sM(K,x),\]
tensoring with such projective modules is exact.

In particular, if we filter our bicomplex by degrees of $C_*(\tilde{X})$,
then $d^0 = \id\otimes d$.  By \autoref{lem::technical}, each
$C_n^G(\tilde{X})$ is projective, and so we get a spectral sequence with
$E^1$-term $C_*(\tilde{X})\otimes_\fundx{} \sM$.  Thus the resulting $E^2 =
E^\infty$ term is $\hl_*(X,\sM)$, exactly as in the nonequivariant case.

If we instead filter by degrees of $\sP_*$, so $d^0 = d\otimes\id$, then
the $E^1$ term is $\coh_*(\tilde{X})\otimes_\fundx{}\sP_*$ and the $E^2$
term is $\Tor_{*,*}^\fundx{}(\coh_*(\tilde{X}),\sM)$, as desired.

The construction of the second spectral sequence is similar, starting from
an injective resolution $\eta\colon \sN \to \sI^*$.
\end{proof}

\begin{proof}[Proof of \autoref{lem::technical}]
The proof is analogous to that of the nonequivariant result, but more
involved.  We may identify $C_n^G(\tilde{X})(K,x)$ with the free $\ground$-module
on generators given by the nondegenerate singular $n$-simplices
$\sigma\colon \Delta^n\to \tilde{X}(K,x)$.  We must show that these free
$\ground$-modules piece together appropriately into a free functor.  More
specifically, by the Yoneda lemma, each $\sigma\colon \Delta^n\to
\tilde{X}(L,y)$ determines a natural transformation
\[\iota_\sigma\colon \contraproj{L,y}\to C_n^G(\tilde{X})  \]
that takes $\id \in \fundhom{(L,y)}{(L,y)}$ to $\sigma$.  We thus obtain a
natural transformation 
\[ \bigoplus_{\{\tau\}} \contraproj{L_\tau,y_\tau} \to C_n^G(\tilde{X}) \]
from any set of nondegenerate $n$-simplices $\{\tau\colon \Delta^n\to
\tilde{X}(L_\tau,y_\tau)\}$. We must show that there is a set $\{\tau\}$
such that the resulting natural transformation is a natural isomorphism,
that is, a levelwise isomorphism.  This amounts to showing that the
following statements hold for our choice of generators $\tau$ and each
object $(K,x)$.  
\begin{enumerate}
\item (Injectivity) For any arrows $\fmap{1}$ and $\fmap{2}$ in $\fundx{}$
with source $(K,x)$ and any generators $\tau_{1}$ and $\tau_{2}$,
$\fmap{1}^*\tau_{1} = \fmap{2}^*\tau_{2}$ must imply that both $\fmap{1} =
\fmap{2}$ and $\tau_{1} = \tau_{2}$.  

\item (Surjectivity) For every $\sigma\colon \Delta^n\to \tilde{X}(K,x)$,
there must be a generator $\tau$ and an arrow $\fmap{}$ such that $\sigma =
\fmap{}^* \tau$.
\end{enumerate}

Fixing $n$, define the generating set as follows.  Regard the initial
vertex $v$ of $\Delta^n$ as a basepoint.  Recall that $\tilde{X}(L,y)$ is
the universal cover of $X^L$ defined with respect to the basepoint $y\in
X^L$, so that the equivalence class of the constant path $c_{L,y}$ at $y$
is the basepoint of $\widetilde{X^L}$.  In choosing our generating set, we
restrict attention to based maps $\sigma\colon \Delta^n\to
\tilde{X}(L_\sigma,y_\sigma)$ that are non-degenerate $n$-simplices of
$\widetilde{X^{L_\sigma}}$.  Such maps $\sigma$ are in bijective
correspondence with based nondegenerate $n$-simplices $\sigma_0\colon
\Delta^n\to X^{L_\sigma}$.  The correspondence sends $\sigma$ to its
composite with the end-point evaluation map $p\colon
\tilde{X}(L_\sigma,y_\sigma)\to X^{L_\sigma}$ and sends $\sigma_0$ to the
map $\sigma\colon \Delta^n\to \tilde{X}(L_\sigma,y_\sigma)$ that sends a
point $a\in \Delta^n$ to the image under $\sigma_0$ of the straight-line
path from $v$ to $a$.  Restrict further to those $\sigma$ that cannot be
written as a composite 
\[ \xymatrix@1{ \Delta^n \ar[r]^-{\rho} & \tilde{X}(L',y')
\ar[rr]^-{(\alpha,\gamma)^*} & & \tilde{X}(L_\sigma,y_\sigma)\\} \]
for any non-isomorphism $(\alpha,\gamma)\colon (L',y')\to (L_\sigma,y_\sigma)$ in
$\Pi$.  Note that, for each such $\sigma$, we can obtain another such
$\sigma$ by composing with the isomorphism $(\xi,\delta)^*$ induced by an
isomorphism $(\xi,\delta)$ in $\Pi$.  We say that the resulting maps $\sigma$
are equivalent, and we choose one $\tau$ in each equivalence class of such
based singular $n$-simplices $\sigma$.

It remains to verify that the natural transformation defined by this set
$\{\tau\}$ is an isomorphism.  This is straightforward but somewhat tedious
and technical.

For the injectivity, suppose that $\fmap{1}^* \tau_{1} = \fmap{2}^*
\tau_{2}$, where $\tau_1,\tau_2$ are in our generating set and 
\begin{align*}
\tau_1\colon \Delta^n\to \tilde{X}(L_1,y_1),  & \hspace{12pt}   \fmap{1} \in \fundhom{(K,x)}{(L_1,y_1)} \\
\tau_2\colon \Delta^n\to \tilde{X}(L_2,y_2),  & \hspace{12pt} \fmap{2} \in \fundhom{(K,x)}{(L_2,y_2)}.
\end{align*}
Since $\tau_i(v) = c_{(L_i,y_i)}$ for $i = 1,2$, we see that $\fmap{i}^*
\tau_i$ must take $v$ to $[\gamma_i]$.  Since $\fmap{1}^* \tau_{1} =
\fmap{2}^* \tau_2$, this means that $[\gamma_1] = [\gamma_2]$; call this
path class $[\gamma]$.  In turn, this implies that $\alpha_1^* y_1 =
\alpha_2^* y_2$; call this point $z\in X^K$, so that $[\gamma]$ is a path from
$x$ to $z$.  Since $\fmap{i} = (\alpha_i,[c_{z}])\circ (\id,[\gamma])$ and
$(\id,[\gamma])$ is an isomorphism in $\fundx{}$, we must have 
\[ (\alpha_1,[c_z])^* \tau_1 = (\alpha_2,[c_z])^* \tau_2. \]
In particular, if we compose each side of this equation with $p$, we obtain 
\[ p \circ (\alpha_1,[c_z])^* \tau_1 = p\circ (\alpha_2,[c_z])^* \tau_2\]
as maps $\Delta^n \to X^K$.  Since we have commutative diagrams 
\[ \xymatrix{ \Delta^n \ar[r]^-{\tau_i} & \tilde{X}(L_i,y_i)
\ar[r]^-{(\alpha,[c_z])^*} \ar[d]_{p} & \tilde{X}(K,z) \ar[d]_{p} \\ &
X^{L_i} \ar[r]^{\alpha_i^*} & X^K }\]
for each $i$, this implies that we have a commutative square 
\[ \xymatrix{ \Delta^n \ar[r]^-{p\circ\tau_1} \ar[d]_{p\circ\tau_2} &
X^{L_1} \ar[d]^{\alpha_1^*} \\ X^{L_2} \ar[r]^{\alpha_2^*} & X^K }\]
If the maps $\alpha_i \colon G/K \to G/{L_i}$ are defined by elements
$g_i\in G$, this implies that the common composite $\Delta^n \to X^K$
factors through the fixed-point sets $X^{g_i L_i g_i^{-1}}$ for each $i$,
and hence through $X^L$, where $L$ is the smallest subgroup containing $g_1
L_1g_1^{-1}$ and $g_2 L_2 g_2^{-1}$.  Since $L_i \subset g_i^{-1} L g_i$,
the maps $\alpha_i\colon G/K\to G/L_i$ factor through the maps $\beta_i\colon
G/L_i\to G/L$ specified by $\beta_i(eK_i) = g_iL$ and there result
factorizations of the $\tau_i$ as 
\[ \xymatrix@1{ \Delta^n \ar[r] &  \tilde{X}(g_i^{-1} L g_i, y_i)
\ar[rr]^-{(q,[c_z])^*} & &  \tilde{X}(L_i,y_i),\\}  \]
where $q$ denotes either quotient map $G/L_i\to G/g_i^{-1} L g_i$.  By
our choice of the generators $\tau$, this can only happen if $g_i^{-1} L
g_i = L_i$, giving $g_1 L_1 g_1^{-1} = g_2 L_2 g_2^{-1}$.  In terms of
$g_1$ and $g_2$, we see that our equation $(\alpha_1,[c_z])^* \tau_1 =
(\alpha_2,[c_z])^* \tau_2$ says that $g_1 \tau_1 = g_2 \tau_2$, that is,
$\tau_2 = g_2^{-1} g_1 \tau_1$.  Since $g_2^{-1} g_1$ defines an
isomorphism $G/{L_2}\to G/{L_1}$, we again see by our choice of the
generators $\tau$ that $\tau_1 = \tau_2$ and that $g_2^{-1} g_1 \in L_1 =
L_2$.  This in turn implies that the maps $\alpha_i \colon G/K \to G/L_i$
defined by the $g_i$ are identical.  The conclusion is that $\fmap{1}^*
\tau_1 = \fmap{2}^* \tau_2$ implies $\tau_1 = \tau_2$ and $\fmap{1} =
\fmap{2}$, as desired.

It only remains to show that we have accounted for all elements of
$C_n(\tilde{X})(K,x)$.  For any map $\sigma\colon \Delta^n \to
\tilde{X}(K,x)$, $\sigma(v)$ is a homotopy class of paths from $(K,x)$ to
$(K,x')$ in $X^K$.  Call this class $[\gamma]$.  Then $(\id,[\gamma])$ is
an isomorphism with inverse $(\id,[\gamma^{-1}])$, and $\sigma' =
(\id,[\gamma^{-1}])^*\sigma$ takes $v$ to the homotopy class of the
constant path at $(K,x')$; it follows that $\sigma =
(\id,[\gamma])^*\sigma'$.  Similarly, if $\sigma'$ factors through
$\tilde{X}(L,y)$ for some $L$ properly containing a conjugate of $K$, then
by definition $\sigma = \fmap{}^*\tau$ for some $\tau\colon \Delta^n \to
\tilde{X}(L,y)$.  We can choose a $\tau$ that does not itself factor and is
in our chosen set of generators.
\end{proof}

\section{Some remarks on the Serre Spectral Sequence}


We say that a map $f\colon E\to X$ is a \defword{$G$-fibration} if $f^K
\colon E^K \to X^K$ is a fibration for every subgroup $K$ of $G$.  Observe
that, if $x\in X^K$, then its preimage $f^{-1}(x)\subset E$ is necessarily
a $K$-space.  As previously mentioned, Moerdijk and Svensson in \cite{MS}
develop an equivariant Serre spectral sequence for $G$-fibrations of
$G$-spaces, which we will now describe.  They use integer-graded Bredon
cohomology.

Given a coefficient system $M\colon \sO_G\op\to \rmod$ and a subgroup $K <
G$, there is a restricted coefficient system $M|_K \colon \sO_K\op \to \rmod$
given by $M|_K (K/L) := M(G/L)$.  We may thus define a local coefficient
system
\[\sh^q(f;M )\colon \fund{X}\op\to\ab \]
to be the functor which acts on objects by
\[ (K,x) \longmapsto \h^q(G\times_K f^{-1}(x);M ) \cong {\renewcommand{\G}{K} \h^q(f^{-1}(x);M |_K)}.\]
It is defined on morphisms via lifting of paths.

The main result of \cite{MS} is the following spectral sequence.

\begin{thm}[Moerdijk and Svensson]
For any $G$-fibration $f\colon E\to X$ and any coefficient system $M\colon
\sO_\G\op\to \rmod$, there is a natural spectral sequence 
\[ E_2^{s,t}(M) = \h^s(X;\sh^t(f;M)) \Longrightarrow \h^{s+t}(E;M).\]
Further, this spectral sequence carries a product structure, in the sense
that there is a natural pairing of spectral sequences
\[ E_r^{s,t}(M) \otimes E_r^{s',t'}(N) \to
E_r^{s+s',t+t'}(M\otimes N) \]
converging to the standard pairing
\[ \h^*(E;M) \otimes \h^*(E;N) \xrightarrow{\smile}
\h^*(E;M\otimes N). \]
On the $E_2$ page, this pairing agrees 
with the standard pairing
\[ \h^s(X;\sh^t(f,M))\otimes \h^{s'}(X;\sh^{t'}(f,N)) \to
\h^{s+s'}(X;\sh^{t+t'}(f,M\otimes N)).\footnote{As usual in the Serre
spectral sequence, the standard pairing on the $E_2$ page incorporates a
sign $(-1)^{s' t}$ relative to the cup product pairing.}\]
\end{thm}

The tensor product $M\otimes N$ above is a levelwise tensor product,
$(M\otimes N)(G/K) = M(G/K)\otimes N(G/K)$, which is distinct from the box
product $\boxp$ discussed in \autoref{sec::mackey}.

Although we have been using the integer-graded equivariant cohomology
originally defined by Bredon in this chapter, equivariant cohomology is more naturally
graded on $RO(G)$, as we discussed in \autoref{ch::background}.  As
we mentioned there, for any coefficient system $M$ which can be
extended to a Mackey functor, the Bredon cohomology theory $\h^*(-;M)$ can
be extended to an $RO(G)$-graded theory.  That is, for every virtual
representation $\omega$, we have a functor $\h^{\omega}(-;M)$.  We
pictured these functors as lying in a two-dimensional plane in
\autoref{ch::background}, but another way of visualizing this extra data is
to say that we have one integer-graded theory $\{\h^{V+n}\}_{n\in\bZ}$ for
each representation $V$ containing no trivial subrepresentations.  Each of
these theories $\h^{V+*}(-;M)$ can be used to define local coefficient
systems $\sh^{V+t}(f,M)$.  Kronholm shows the following in his thesis
\cite{bill}.

\begin{thm}[Kronholm]\label{thm::bill}
For each real representation $V$, there is a natural spectral sequence
\[ E_2^{s,t}(M,V) = \h^s(X;\sh^{V+t}(f,M)) \Longrightarrow
\h^{V+s+t}(E;M). \]
Further, for each $V,V'\in RO(G)$, there is a pairing
\[ E_r^{s,t}(M,V) \otimes E_r^{s',t'}(N,V') \to
E_r^{s+s',t+t'}(M\otimes N,V+V')\]
converging to the standard pairing on $E_\infty$ and agreeing 
with the standard pairing on $E_2$.
\end{thm}

We have an analogue of \autoref{prop::hzero}, as well.

\begin{prop}\label{prop::eqhzero}
$\h^0(X;\sM) \cong \Hom_{\Pi}(\underline{\ground},\sM)$, where $\underline{\ground}$ is
the constant functor.
\end{prop}

\begin{proof}
As in \autoref{prop::hzero}, this comes from identifying $\coh_0(\tilde{X})
\cong \underline{\ground}$ and from the left exactness of $\Hom_\Pi$.
\end{proof}

\begin{prop}\label{prop::gconnected}
Suppose that $X$ is $G$-connected, in the sense that each $X^K$ is nonempty
and connected, and let $x \in X^G$.  Then
$\Hom_{\Pi}(\underline{\ground},\sM)$ is isomorphic to a sub-$\ground$-module of
$\sM(G,x)$.
\end{prop}

\begin{proof}
Since $X$ is $G$-connected, $(G,x)$ is a weakly terminal object in
$\fund{X}$, i.e.\ for every $(K,y)$ there is a map $(K,y)\to (G,x)$.
It follows that an element of $\Hom_{\Pi}(\underline{\ground},\sM)$ is determined
by the map of $\ground$-modules $\ground\to \sM(G,x)$.
\end{proof}

\section{Equivariant classifying spaces}\label{sec::eqclassifying}

For the remainder of this  \paper{}, with the exception of classical structure
groups, groups named with Greek letters will be viewed as structure groups
and those named with Latin letters will be viewed as ambient groups of
equivariance.  Suppose that we are given a structure group $\Gamma$ and a
group of equivariance $G$.  Then, as discussed in \cite{AK} and many other
places, there is a notion of a principal $(G,\Gamma)$-bundle, namely a
projection to $\Gamma$-orbits $E\to B=E/\Gamma$ of a $\Gamma$-free
$(G\times\Gamma)$-space $E$.  Such equivariant bundles are classified by
universal principal bundles $E_G\Gamma\to B_G\Gamma$, where $E_G\Gamma$ is a space
whose fixed point sets $(E_G\Gamma)^\Lambda$ are empty when $\Lambda \subset
G\times\Gamma$ intersects $\Gamma = \triv\times \Gamma$ nontrivially and
contractible when $\Lambda \cap \Gamma = \triv$. 

As should be expected, for a fixed group $G$, the equivariant classifying
space construction can be made functorial; that is, there are functors
\begin{align*}
E_G &\colon \gset{Grp} \to \gtop\\
B_G &\colon \gset{Grp} \to \gtop.
\end{align*}
It will be helpful to pick particular functors $E_G$ and $B_G$, using the
categorical two-sided bar construction.  Given any groups $G$ and $\Gamma$, we
may take 
\begin{align*}
E_G \Gamma &:= B(T_\Gamma,\sO_{G\times\Gamma},O_{G\times\Gamma}) \\
B_G \Gamma &:= (E_G\Gamma)/\Gamma
\end{align*}
where $\sO_{G\times\Gamma}$ is the orbit category,
$O_{G\times\Gamma}\colon \sO_{G\times\Gamma}\to\spaces$ is given by viewing an
orbit $(G\times\Gamma)/\Lambda$ as a topological space, and $T_\Gamma\colon
\sO_{G\times\Gamma}\op \to\spaces$ is the functor which takes 
\[(G\times\Gamma)/\Lambda \longmapsto 
\begin{cases}
\pt & \text{if }\Lambda\cap\Gamma=\triv\\
\emptyset & \text{otherwise.}
\end{cases}\]
Since the functor $O_{G\times\Gamma}$ lands in $(G\times\Gamma)\textbf{-Top}$,
$E_G\Gamma$ is a $(G\times\Gamma)$-space, and it is easy to check that it has the
correct fixed points.

The bar construction $B(-,-,-)$ is a functor from the category of triples
$(T,\sC,S)$ to $\spaces$.  Here $\sC$ is a category and $S,T$ are
respectively a covariant and a contravariant functor $\sC\to\spaces$.  A
morphism $(T_1,\sC_1,S_1)\to (T_2,\sC_2,S_2)$ in the category of triples
consists of a functor $F\colon \sC_1\to\sC_2$ together with natural
transformations $S_1 \to S_2\circ F$ and $T_1\to T_2\circ F\op$.  It
follows that, for fixed $G$, we can make $E_G(-)$ into a functor
$\textbf{Grp}\to\spaces$ as follows.  Given a homomorphism $\varphi\colon
\Gamma_1\to \Gamma_2$, we apply $B(-,-,-)$ to the morphism of triples given by
the functor
\[ F \colon \sO_{G\times\Gamma_1} \to \sO_{G\times\Gamma_2} \colon (G\times \Gamma_1)/\Lambda
\longmapsto (G\times \Gamma_2)/{\left( (\id\times\varphi)(\Lambda) \right)},\]
with the obvious natural transformations $O_{G\times\Gamma_1} \to
O_{G\times\Gamma_2}\circ F$ and $T_{\Gamma_1}\to T_{\Gamma_2}\circ F\op$ (for the
latter, note that if $\Lambda\cap\Gamma_1 = \triv$, then also
$(\id\times\varphi)(\Lambda)\cap\Gamma_2 = \triv$).  Since the morphism
$E_G\Gamma_1 \to E_G\Gamma_2$ induced by $\Gamma_1\to\Gamma_2$ is $\Gamma_1$-equivariant,
there is an induced map $B_G\Gamma_1\to B_G\Gamma_2$, making $B_G(-)$ a functor.

The following result will be useful later.

\begin{prop}\label{prop::pullback}
Fix a group $G$.  Corresponding to any short exact sequence of structure
groups
\[ 1\to \Gamma \xrightarrow{\varphi} \Upsilon\xrightarrow{\psi} \Sigma \to 1 \]
there is a pullback square in the category of $G$-spaces
\[ \xymatrix@C=3pc{
  B_G \Gamma \ar[r] \ar^{/\Sigma}[d] 
  & E_G \Sigma \ar^{/\Sigma\phantom{=/\Upsilon}}[d] \\
  B_G \Upsilon \ar[r]^{B_G\psi} & B_G \Sigma \\
}\]
\end{prop}

\begin{proof}
By functoriality of $E_G(-)$ and the definition of $B_G(-)$, the map $\psi$
induces a commutative diagram
\[ \xymatrix@C=3pc{
  E_G \Upsilon \ar[r]^{E_G\psi} \ar[d]^{/\Upsilon} 
  & E_G \Sigma \ar[d]^{/\Upsilon = /\Sigma} \\
  B_G \Upsilon \ar[r]^{B_G\psi} & B_G\Sigma \\
}\]
Further, the projection $E_G \Upsilon\to B_G \Upsilon = (E_G \Upsilon)/\Upsilon$
factors as 
\[ E_G \Upsilon \to (E_G \Upsilon)/\Gamma \xrightarrow{/\Sigma} (E_G \Upsilon)/\Upsilon \]
and since the action of $\Gamma$ on $E_G \Sigma$ is trivial, $E_G \Upsilon\to
E_G \Sigma$ also factors through $(E_G \Upsilon)/\Gamma \cong  B_G \Gamma$.  We
thus get the commutative square described in the proposition.  Viewing
$E_G\Sigma$ as a $G$-space via the inclusion $G\hookrightarrow
G\times\Sigma$, this is a diagram in the category of $G$-spaces.  The
square induces a homeomorphism on the fibers of the vertical maps, and
hence is a pullback in the category of $G$-spaces, as desired.
\end{proof}

\renewcommand{\G}{{C_{\!p}}}

\begin{exmp}
We have already studied one example of a classifying space: the complex
projective space $\cpv{\cxuniverse}$ on a complete complex universe
$\cxuniverse$ is a model for $B_G SO(2)$, the equivariant classifying space
of the circle group.  \autoref{cor::cxproj}, \autoref{thm::cpvgenerators},
and \autoref{thm::multcpv} established the structure of
$\ulh^*(\cpv{\cxuniverse})$ as an algebra over $\ulh^*(\pt)$ for $G=\G$,
provided the ground ring $\ground$ has no torsion of order $p$.  In the
next chapter, we will use this result to compute the cohomology of $B_\G
O(2)$ for $\ground=\fq$.
\end{exmp}




\chapter[The cohomology of $B_\G O(2)$]{Example: the cohomology of $B_\G O(2)$ }\label{ch::computations}

As before, let $\G$ be the cyclic group of order $p$.  In this
chapter, we will calculate the $RO(\G)$-graded Bredon cohomology
\[\h^*(B_\G O(2);A)\] of the equivariant classifying space $B_\G
O(2)$ for ground ring $\ground=\fq$ and for an odd prime $p\ne q$.  By way
of motivation, we know from \autoref{sec::cxprojmult} what $\h^*(B_\G S^1)$
is, and historically $O(2)$ is often the first test case to try after
$S^1$.  Let $D_1$ through $D_p$ and $C$ be the elements described in
\autoref{thm::cpvgenerators} and \autoref{thm::multcpv}.

\begin{thm}\label{thm::mythm}
Let $p,q$ be distinct odd primes.  Then, as an algebra over $\h^*$,
$\h^*(B_\G O(2);A)$ is isomorphic to the subalgebra of
$\h^*(\cpv{\cxuniverse};A)$ generated by the elements
$D_2,D_4,\ldots,D_{p-1},D_1 C,\ldots,D_{p-2}C,C^2$.
\end{thm}

By analogy with \autoref{ex::noneq}, we will approach \autoref{thm::mythm} via
the short exact sequence of structure groups
\[ 1\to SO(2) \to O(2) \to \cyctwo \to 1 \]
and the induced fibration $f\colon B_\G O(2) \to B_\G \cyctwo$, where again
$\cyctwo$ is the cyclic group of order 2.  We will first identify the
$\G$-action on the fibers of $f$, then explicitly describe the coefficient
systems $h^{V+t}(f,A)$ and $\coh_*(\widetilde{B_\G\cyctwo})$, and finally
prove the theorem.

\section{Identifying the fibers of $f\colon B_\G O(2)\to B_\G\cyctwo$}

We will begin by identifying models for the equivariant classifying spaces
under consideration.  Recall that nonequivariantly, the universal bundle
$E\cyctwo \to B\cyctwo$ has as a model $S^\infty \to \bR P^\infty$, where
$S^\infty = S(\bR^\infty)$ is the unit sphere in $\bR^\infty$ and $\bR
P^\infty = \bR P(\bR^\infty)$ is the infinite-dimensional real projective
space.   
In general, for any group $G$, let $\runiverse$ be a direct sum containing
countably infinitely many copies of each real representation of $G$.  Then
$S(\runiverse)\to \bR P(\runiverse)$ is a model for $E_G\cyctwo\to
B_G\cyctwo$.  The $G\times\cyctwo$ action on $S(\runiverse)$ comes from the
$G$ action on $\runiverse$ and the $\cyctwo$ action by multiplication by
$-1$.  When $G = \G$ is cyclic of prime order, however, we can choose a
simpler model. 

\begin{lem}\label{lem::rpinfty}
If $p$ is an odd prime, then $E_\G\cyctwo \to B_\G\cyctwo$ has as a model
$S^\infty\to\bR P^\infty$ with the trivial $\G$-action on both spaces.
\end{lem}

\begin{proof}
To verify this claim, it suffices to check the fixed-point sets of
$S^\infty$.  Since $\cyctwo$ acts freely on $S^\infty$, the fixed
points $(S^\infty)^\Lambda$ are certainly empty when $\Lambda\cap\cyctwo$
is nontrivial.  So we need only check that the fixed point set is
contractible whenever $\Lambda\cap\cyctwo$ is trivial.

Note that the subgroups $\Lambda\subset \G\times \cyctwo$ which intersect
$\cyctwo$ trivially are the ``twisted diagonal subgroups''
$\Lambda = \Delta_{\rho,K} = \{ (h,\rho(h)) | h\in K \}$, for $K$ a
subgroup of $\G$ and $\rho\colon K\to\cyctwo$ a homomorphism.  However,
since $p$ is an odd prime, the only homomorphism $K\to\cyctwo$ is the
trivial homomorphism, and so $\Delta_{\rho,K} = K\times\triv$.  This acts
trivially on $S^\infty$, so $(S^\infty)^\Lambda = S^\infty \simeq \pt$,
as desired.  
\end{proof}

Similarly, $SO(2)$ is the circle $\bT$.  Letting $\cxuniverse$ again be the
direct sum of countably infinitely many copies of each irreducible complex
representation of $\G$, an analysis of the fixed-point sets of
$S(\cxuniverse)$ gives the following well-known result.

\begin{lem}\label{lem::cpinfty}
For any prime $p$, $E_\G SO(2) \to B_\G SO(2)$ has as a model
$S(\cxuniverse) \to \cpv{\cxuniverse}$.  The $\G\times SO(2)$ action on
$S(\cxuniverse)$ comes from the $\G$ action on $\cxuniverse$ and the usual
circle action on the complex plane. \hfill $\Box$
\end{lem}

We are now in a position to use \autoref{prop::pullback} for the short exact
sequence 
\[ 1\to SO(2) \to O(2) \xrightarrow{\text{det}} \cyctwo \to 1. \]
We have a pullback square in the category of $\G$-spaces
\[ \xymatrix@C=3pc{
  B_\G SO(2) \ar[r] \ar[d] & E_\G \cyctwo \ar[d] \\
  B_\G O(2) \ar[r] & B_\G \cyctwo \\
}\]
By \autoref{lem::rpinfty}, the $\G$-actions on $E_\G\cyctwo$ and $B_\G
\cyctwo$ are trivial.  It follows that $E_\G\cyctwo$ is $\G$-contractible,
and so we have proved the following about our map $f\colon B_\G O(2)\to
B_\G\cyctwo$.

\begin{lem}\label{lem::fibers}
For each point $x\in B_\G\cyctwo$, the fiber $f^{-1}(x)$ is $\G$-homotopy
equivalent to $B_\G SO(2) = \cpv{\cxuniverse}$.\hfill $\Box$
\end{lem}

\section{The local coefficient system $\sh^{V+t}(f,A)$}
Recall from \autoref{thm::bill} that the equivariant Serre spectral sequence
for a $\G$-fibration $f\colon E\to X$ has
\[ E_2^{s,t}(M,V) = \h^s(X;\sh^{V+t}(f,M)) \Longrightarrow
\h^{V+s+t}(E;M). \]
Choose the coefficient ring $\ground = \fq$, the finite field with $q$ elements,
for an odd prime $q\ne p$.  As in \autoref{thm::mythm}, let $M = A$.
We must first analyze the local coefficient systems
\[ \sh^{V+t}(f,A)\colon (K,x)\longmapsto \h^{V+t}(\G\times_K f^{-1}(x)). \]
We may start by taking a skeleton $\Pi$ of the category $\fund{B_\G\cyctwo}$.  
Again using \autoref{lem::rpinfty}, we see that each fixed-point set
$(B_\G\cyctwo)^K$ is nonempty and connected with fundamental group
$\cyctwo$, so $\Pi$ has two objects $(\G,x_0)$ and $(\triv,x_0)$.
Recall from \autoref{sec::modp} that we denoted the orbit $\G/\triv$ by
$\freeo$ and $\G/\G$ by $\trivo$.  We will continue that convention in this
section.  For convenience, and to bring out the parallels, we will also
write $(\trivo,x_0)$ and $(\freeo,x_0)$ in place of $(\G,x_0)$ and
$(\triv,x_0)$, respectively.

Recalling that a map $(K,x)\to (L,y)$ consists of a $\G$-map $\alpha\colon
\G/K\to \G/L$ and a homotopy class of paths $[\gamma]$ from $x$ to
$\alpha^* y$, we see that there are two endomorphisms of $(\trivo,x_0)$. One of
these is the identity, and the other, $\kappa$, squares to the identity.
Similarly, there are two morphisms $(\freeo,x_0)\to (\trivo,x_0)$, and
composition with $\kappa$ exchanges them.  Finally, the endomorphisms of
$(\freeo,x_0)$ are in bijection with $\G\times\cyctwo$; when precomposing
with the two morphisms $(\freeo,x_0)\to (\trivo,x_0)$, only the $\cyctwo$ factor
has an effect.  We may visualize $\Pi$ as follows:
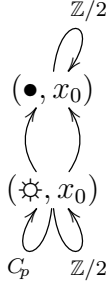
\[ \xymatrix@1{
(\trivo,x_0) \ar@(u,ur)[]^\cyctwo \\
(\freeo,x_0) \ar@(d,dr)[]_{\cyctwo} \ar@(d,dl)[]^{\G} \ar@/^2ex/[u] \ar@/_2ex/[u] \\
}\]

We can then explicitly describe the coefficient system $\sh^{V+t}(f,A)$.
Let $\rho$ be the projection $\freeo\to\trivo$.

\begin{prop}\label{prop::h}
The functor $\sh^{V+t}(f,A)\colon \Pi\to\rmod$ takes
\begin{align*}
(\trivo,x_0) & \longmapsto \h^{V+t}(\cpv{\cxuniverse}; A) =
\ulh^{V+t}(\cpv{\cxuniverse})(\freeo)\\
(\freeo,x_0) & \longmapsto \neqh^{|V|+t}(\cpv{\cxuniverse};A(\freeo)) =
\ulh^{V+t}(\cpv{\cxuniverse})(\trivo)
\end{align*}
The functor is determined on morphisms by the following: 
\begin{enumerate}
\item \label{item::down} The image of $(\rho,[c_{(\freeo,x_0)}])$ is the image
of $\rho$ in the underlying contravariant coefficient system of
the Mackey functor $\ulh^{V+t}(\cpv{\cxuniverse})$;

\item \label{item::endobottom} For each map $g\colon \freeo \to \freeo$,
the image of $(g,[c_{(\freeo,x_0)}])$ is the image of $g$ in the underlying
contravariant coefficient system of $\ulh^{V+t}(\cpv{\cxuniverse})$;

\item \label{item::endotop} Let $D_1$ through $D_{p-1}$ and $C$ be the
algebra generators described in \autoref{thm::multcpv}.  The nontrivial
automorphism of $(\trivo,x_0)$ acts by
the identity on the $D_{2k}$ and by multiplication by $-1$ on the
$D_{2k-1}$ and $C$.
\end{enumerate}
\end{prop}

This may be visualized by the diagram below.
\[ \xymatrix@1{
\h^{V+t}(\cpv{\cxuniverse}; A) \ar@/^2ex/[d] \ar@/_2ex/[d] \ar@(u,ur)[]^{\cyctwo} \\
\neqh^{|V+t|}(\cpv{\cxuniverse};A(\freeo)) \ar@(d,dl)[]^{\G} \ar@(d,dr)[]_{\cyctwo} \\
} \]
Note that the second downward arrow is given by composing the maps of
(\ref{item::down}) and (\ref{item::endotop}).

\begin{proof}
\autoref{lem::fibers} identifies the value of $\sh^{V+t}(f;A)$ on objects.

For item (\ref{item::down}), 
the downward arrow induced by the morphism
$(\rho,[c_{(\freeo,x_0)}])$ is simply the map on
cohomology induced by the space-level map
\[\G\times f^{-1}(x) \to f^{-1}(x).\]
This is the same as the map $\ulh^{V+t}(\cpv{\cxuniverse})(\trivo) \to
\ulh^{V+t}(\cpv{\cxuniverse})(\freeo)$ induced by the span
\[ \cospan{\freeo}{\trivo}{} \]
A similar argument identifies the map in item (\ref{item::endobottom}).

It remains to identify the $\cyctwo$ action.  Recall that any $B_G\Gamma$ is of the
nonequivariant homotopy type of the classifying space $B\Gamma$.  In
particular, our fibration $B_\G O(2)\to B_\G \cyctwo$ corresponds to the
nonequivariant fiber sequence
\[ BSO(2)\to BO(2)\xrightarrow{B\text{det}} B\cyctwo \]
We know that $\neqh^*(BSO(2);\fq) \cong \fq[x]$, a polynomial algebra on a
generator $x$ in degree $2$, and that $\pi_1 B\cyctwo$ acts by $-1$ on $x$.
This determines the $\cyctwo$ action at the $\freeo$ level in the Mackey
functor $\ulh^{V+t}$, and thus at the $\freeo$ level in $\sh^{V+t}(f;A)$.
The algebra generators of $\ulh^*(\cpv{\cxuniverse})$ at the $\trivo$ level
are $D_1$ through $D_{p-1}$ and $C$, where $D_j$ restricts to $x^j$ at the
$\freeo$ level and $C$ restricts to $x^p$.  It follows that the action of
$\cyctwo$ at the $\trivo$ level must be by $-1$ on the elements $D_{2k-1}$
and $C$, and by the identity on the $D_{2k}$.
\end{proof}

\section{The local coefficient system $\coh_*(\widetilde{B_\G\cyctwo})$}

We will continue to write $\Pi$ for $\Pi_{\G} B_{\G}\cyctwo$.
Recall that $\coh_*(\tilde{X})$ is the coefficient system ${\Pi_G X\op\to\rmod}$
which takes ${(K,x)\longmapsto \neqh_*(\widetilde{X^K}(x))}$.  In our case, since
$B_\G\cyctwo \cong \bR P^\infty$ with trivial $\G$-action, it follows that 
$\coh_*(\widetilde{B_\G\cyctwo})$ is the constant functor at $\neqh_*(S^\infty)
\cong \neqh_*(\pt)$.  We will continue to take our coefficient ring
$\ground=\fq$ for $q\ne p$, so $\neqh_*(\pt) \cong \fq$ concentrated in dimension
0.

One reason that $\ground=\fq$ is such a convenient choice in this example is that 
the constant functor at $\neqh_*(\pt)$ is projective.

{\renewcommand{\ground}{\fq}
\begin{prop}
The constant functor $\underline{\ground}\colon \Pi\op \to\rmod$ is a direct
summand of a representable functor and hence a projective object in the
category $[\Pi\op,\rmod]$.
\end{prop}

\begin{proof}
Consider the represented functor $\ground \Pi(-,(\trivo,x_0))$.  We see by
inspection that $\ground \Pi((K,x_0),(\trivo,x_0)) \cong \ground[\cyctwo]$ for both
possible values of $K$.  For an appropriate choice of basis, we can display this as
\[ \xymatrix@R=3pc@C=3pc{
\ground \oplus \ground
\ar@/_2ex/[d]_{\mymatrix{1&0\\0&1}} \ar@/^2ex/[d]^{\mymatrix{0&1\\1&0}}
\ar@(u,ur)[]^{\mymatrix{0&1\\1&0}} \\
\ground \oplus \ground
\ar@(d,dr)[]_{\mymatrix{0&1\\1&0}} \ar@(d,dl)[]^{\text{triv}} \\
}\]
That is, the nontrivial element of $\cyctwo$ acts by interchanging the
basis elements, on both the top and the bottom.  The action of $\G$ on the bottom
is trivial, and the downward maps behave as shown.  If we take the new
basis given by the change-of-coordinates matrix $\mymatrix{1&\phantom{-}1\\1&-1}$
(using the fact that $q\ne 2$), we see that $\ground \Pi(-,(\trivo,x_0))$
breaks up as the direct sum of two functors, one of which is our constant
functor $\underline{\ground}$.
\[ \xymatrix@R=0.5pc@C=1pc{
  & \ground \ar@(u,ur)[]^{\id} \ar@/_2ex/[dd]_{\id} \ar@/^2ex/[dd]^{\id} & & 
    \ground \ar@(u,ur)[]^{-1} \ar@/_2ex/[dd]_{\id} \ar@/^2ex/[dd]^{-1} \\
\ground \Pi(-,(\trivo,x_0))\hspace{10pt}\cong\hspace{10pt}  & & \hspace{10pt}\oplus\hspace{10pt} & \\
  & \ground \ar@(d,dr)[]_{\id} \ar@(d,dl)[]^{\text{triv}} & &
    \ground \ar@(d,dr)[]_{-1} \ar@(d,dl)[]^{\text{triv}} \\
}\]
\end{proof}
}

\section{The calculation of $\h^*(B_\G O(2);A)$}

We are now prepared to prove \autoref{thm::mythm}.  Fix odd primes $p\ne
q$, and continue to take $\ground=\fq$.  We will also continue to make
heavy use of the identification of $B_\G \cyctwo$ in
\autoref{lem::rpinfty}.

\begin{proof}[Proof of \autoref{thm::mythm}]
We will use the equivariant Eilenberg spectral sequence to identify the
$E_2$ page of the Serre spectral sequence for $f\colon B_\G O(2)\to
B_\G\cyctwo$ and then show that the Serre spectral sequence collapses with
no extension problems.

Since $\widetilde{B_\G\cyctwo}$ is the constant functor at $S^\infty$, the relevant
equivariant Eilenberg spectral sequence in this case is
\[ \Ext_\Pi^{u,v} \left( \coh_*(S^\infty),\sh^{V+t}(f;A) \right) \Longrightarrow
\h^{u+v}(B_\G\cyctwo;\sh^{V+t}(f;A)). \]
As before, in the $E_2$ term, $u$ is the homological degree and $v$ is the
internal grading on $\coh_*$.  Since $\coh_v(S^\infty)$ is either $0$ or
$\underline{\fq}$, both of which are projective,
$\Hom_\Pi(\coh_v(S^\infty),-)$ is exact, and so all $\Ext$ terms with $u>0$
vanish.  It follows that the spectral sequence collapses at $E_2$ with no
extension problems, and so the $E_2$ terms of the Serre spectral sequence
are given by
\[ \h^s(B_\G\cyctwo;\sh^{V+t}(f;A)) \cong
\Hom_\Pi \left( \coh_s(S^\infty),\sh^{V+t}(f;A) \right). \]
The homology of $S^\infty$ vanishes for $s>0$, so in fact the Serre
spectral sequence also collapses with no extension problems.

$B_\G\cyctwo$ has a trivial $\G$ action and is $\G$-connected, meaning that
\autoref{prop::eqhzero} and \autoref{prop::gconnected} apply.  Thus
we may identify
\[ \h^{V+t}(B_\G O(2);A) \cong \h^0(B_\G\cyctwo;\sh^{V+t}(f;A))
\hookrightarrow \h^{V+t}(B_\G SO(2);A)\]
as algebras over the cohomology of a point.

More specifically, we have
\[ \h^{V+t}(B_\G O(2);A) \cong \Hom_\Pi(\underline{\fq},\sh^{V+t}(f;A)). \] 
For any $\sN$ and
any element ${\eta\nobreak\in\Hom_\Pi(\underline{\fq},\sN)}$, $\eta$ factors through
the ``fixed subfunctor of $\sN$,''  i.e.\ the subfunctor
\[ \xymatrix@1{
\sN(\trivo,x_0)^{\cyctwo} \ar[d] \\
\sN(\freeo,x_0)^{\G\times\cyctwo}.
}\]
Note the two downward arrows must give the same map, and so a single arrow has
been drawn above.  As already observed, our $\Pi$ has a weakly terminal
object, and so a map in $\Hom_\Pi(\underline{\fq},\sN)$ is in fact
determined by choosing an element of $\sN(\G,x_0)^{\cyctwo}$.  By
inspection of the structure of $\Pi = \fund{B_\G\cyctwo}$, we see that
every element of $\sN(\G,x_0)^{\cyctwo}$ defines a map in
$\Hom_\Pi(\underline{\fq},\sN)$, as well.  
In other words, we have demonstrated that 
\[ \h^{V+t}(B_\G O(2);A) \cong \h^{V+t}(B_\G SO(2);A)^{\cyctwo}\]
for each $V$ and $t$, and hence
\[ \h^{*}(B_\G O(2);A) \cong \h^{*}(B_\G SO(2);A)^{\cyctwo}.\]
By \autoref{thm::multcpv} and \autoref{prop::h}, it follows that
$\h^{*}(B_\G O(2);A)$ is the subalgebra of $\h^{*}(B_\G SO(2);A)$ generated
by the elements $D_{2k}$, $D_{2k-1} C$, and $C^2$, the generators
restricting to an even power of the nonequivariant generator $x$ of
$\neqh^*(\mathbb{C}P^\infty)$.
\end{proof}

In fact, closer examination shows that the Green functor $\ulh^*(B_\G
O(2))$ is a sub-Green functor of $\ulh^*(\cpv{\sU}) = \ulh^*(B_\G SO(2))$,
again on the generators $D_{2k}$, $D_{2k-1} C$, and $C^2$.

\chapter{Multiplicative structure of the spectral sequence}\label{ch::mult}
\section{Multiplicative structure}\label{sec::mult}

The nonequivariant Serre spectral sequence for the computation of
$\neqh^*(E;\ground)$ is a spectral sequence of algebras, and the same is
true of its equivariant analogue.  The Eilenberg spectral sequence would
thus be much more powerful if it could be used to find the multiplicative
structure of the $E_2$ page of the Serre spectral sequence.  An elaboration
of Whitehead's proof of \autoref{thm::comparison} shows that when $\sN$ is
$\ground$-algebra valued, $\neqh^*(X;\sN)$ and $\neqh^*(X; N)$ are isomorphic as
$\ground$-algebras, so we need only show that the Eilenberg spectral
sequence is multiplicative.

\subsection{The nonequivariant spectral sequence}\label{subsec::noneqmult}
Consider a fibration $f\colon E\to X$ with fiber $F$ and $\pi := \pi_1 X$.
Return for the moment to $\ground[\pi]$-module notation in the nonequivariant
case.  The Serre spectral sequence has
\[ E_2^{s,t} = \neqh^s(X;\sH^t(F;\ground)) \]
where $\sH^t(F;\ground)$ is \rarticle $\ground[\pi]$-module.  This has the expected pairing
\[ \neqh^{s_1}(X;\sH^{t_1}(F;\ground)) \otimes \neqh^{s_2}(X;\sH^{t_2}(F;\ground)) \to \neqh^{s_1+s_2}(X;\sH^{t_1+t_2}(F;\ground)). \]
We know from \autoref{thm::CE} that, for each fixed $t$, there is a spectral
sequence
\[ E_2^{u,v} = \Ext_{\ground[\pi]}^{u,v}(\neqh_*(\tilde{X}),\sH^t(F;\ground))
\Longrightarrow \neqh^{u+v}(X;\sH^t(F;\ground)). \]
It is convenient to view the collection of all of these spectral sequences
as a single trigraded spectral sequence.

\begin{lem}[Trigraded Eilenberg Spectral Sequence]\label{thm::triCE}
Let $X$ be a path-connected based space with universal cover $\tilde{X}$
and $\pi = \pi_1 X$ (as in \autoref{sec::noneq}).  Let $N^*$ be a graded
$\ground[\pi]$ module.  Then there is a spectral sequence
\[ E_2^{t,u,v} = \Ext_{\ground[\pi]}^{u,v}(\neqh_*(\tilde{X}),N^t)
\Longrightarrow \neqh^{u+v}(X;N^t).\]
The differentials $d_r$ change tridegrees by $(t,u,v)\mapsto
(t,u+r,v-r+1)$.
\end{lem}

\begin{proof}
For each $t$, choose an injective resolution $N^t \to I^{t,*}$.  Then we
have a trigraded complex $\Hom_{\ground[\pi]}(C_*(\tilde{X}),I^{*,*})$.  For each
fixed $t$, we have a bicomplex as in \autoref{thm::CE} and thus a spectral
sequence.  As $t$ varies, these fit together to form the above trigraded
spectral sequence.
\end{proof}

Now suppose that $N^*$ is \rarticle $\ground[\pi]$-algebra.  A multiplicative structure
on \[\Hom_{\ground[\pi]}(C_*(\tilde{X}),I^{*,*})\] which induces the expected
multiplicative structure on $\neqh^*(X;N^*)$ is sufficient to make the
trigraded Eilenberg spectral sequence into a spectral sequence of algebras.

To find such a structure, it is enough to show the existence of a product
map $\varphi\colon I^{*,*}\otimes I^{*,*}\to I^{*,*}$ compatible with the
multiplication $N^*\otimes N^* \to N^*$.  We then have the composite
\[
\xymatrix@R=1pc@C=1pc{ 
\makebox[20pt]{} &  \makebox[60pt]{$\Hom_{\ground[\pi]}(C_{v_1}(\tilde{X}),I^{t_1,u_1})\otimes \Hom_{\ground[\pi]}(C_{v_2}(\tilde{X}),I^{t_2,u_2})$} \ar[dr] & & \\
 &  & \makebox[60pt]{$\Hom_{\ground[\pi]}(C_{v_1}(\tilde{X})\otimes C_{v_2}(\tilde{X}),I^{t_1,u_1}\otimes I^{t_2,u_2})$} \ar[dr]  & \\
 &  & & \makebox[60pt]{$\Hom_{\ground[\pi]}(C_{v_1+v_2}(\tilde{X}),I^{t_1+t_2,u_1+u_2})$} \\}
\]
where the first arrow takes the tensor product of a pair of maps, and the
second is induced by $\varphi$ and the usual Alexander-Whitney map
$C_*(\tilde{X}) \to C_*(\tilde{X})\otimes C_*(\tilde{X})$.

In general, it is difficult to come up with a suitable product on
$I^{*,*}$.  However, there are some situations in which this product
exists, and then the cohomological Eilenberg spectral sequence becomes a
spectral sequence of algebras.

\begin{thm}\label{thm::easymult}
Suppose that $\pi$ is a finite group of order $n$, $\ground$ is a field of
characteristic prime to $n$, and $N^*$ is \rarticle $\ground[\pi]$-algebra.  Then the
spectral sequence of  \autoref{thm::triCE} is a spectral sequence of algebras,
where the multiplication on the $E_2$ page comes from the diagonal map on
$\neqh_*(\tilde{X})$ and the multiplication of $N^*$.
\end{thm}

\begin{proof}
Since $\ground[\pi]$ is semisimple, every $\ground[\pi]$-module is injective and
projective.  Hence we can choose each $I^{t,*}$ to be the resolution $0\to
N^t \to 0 \to \cdots$.  Since $N^*$ is an algebra, it is clear that we have
$I^{*,*}\otimes I^{*,*} \to I^{*,*}$ as desired.
\end{proof}

In fact, it suffices to assume that each $N^t$ has an underlying $\ground$-module
which is finitely generated and projective.

\begin{thm}\label{thm::multstr}
Suppose that $\pi = \pi_1 X$ is a finite group and that $N^*$ is a graded
$\ground[\pi]$-module such that the underlying $\ground$-module of each $N^t$ is
finitely generated projective.  Then the spectral sequence of
\autoref{thm::triCE} is a spectral sequence of algebras; the multiplication on
$E_2$ is as in \autoref{thm::easymult}.
\end{thm}

To prove \autoref{thm::multstr}, we will need to review some terminology and
results from \cite[Chapter VI]{Brown}.  In the following, let $\ground$ be any
ring and $\pi$ be a finite group.  (The definitions work for general $\pi$,
but the propositions require finite groups.)  Brown takes $\ground=\bZ$
throughout, but all proofs work for general $\ground$ as well.

\begin{defn}
An injection $M_1 \hookrightarrow M_2$ of $\ground[\pi]$-modules is an
\defword{admissible injection} if it is split when regarded as an injection
of $\ground$-modules.  A long exact sequence of $\ground[\pi]$-modules is
\defword{admissible} if it is contractible when viewed as a chain complex
of $\ground$-modules.
\end{defn} 

\begin{defn}
We say that \rarticle $\ground[\pi]$-module $Q$ is \defword{relatively injective} if the
functor $\Hom_{\ground[\pi]}(-,Q)$ takes admissible injections of
$\ground[\pi]$-modules to surjections of $\ground$-modules.  If $M$ is any
$\ground[\pi]$-module, a \defword{relatively injective} resolution of $M$
consists of a complex $0\to Q^0 \to Q^1 \to \cdots$ such that each $Q^i$ is
relatively injective, together with a weak equivalence $M\to Q^*$ such that
$0\to M\to Q^0 \to Q^1\to \cdots$ is admissible.
\end{defn}

In particular, injective modules are relatively injective and injective
resolutions are relative injective resolutions.  The following proposition
shows that it suffices to consider relative injective resolutions when
computing Tor and Ext. 

\begin{prop}\label{prop::brown}
Let $\pi$ be a finite group, and let $M$ be a left $\ground[\pi]$-module.
\begin{enumerate}
\item Any two relative injective resolutions of $M$ are canonically
homotopy equivalent.

\item If $M$ is projective as \rarticle $\ground[\pi]$-module, then $M$ is relatively
injective.

\item There is \rarticle $\ground[\pi]$-module isomorphism $\Hom_\ground(M,\ground) \cong
\Hom_{\ground[\pi]}(M,\ground[\pi])$.
\end{enumerate}
\end{prop}

\begin{proof}
These are propositions VI.2.5, VI.2.3, and VI.3.4 in \cite{Brown}.
\end{proof}

These facts are sufficient to prove the following useful lemma.

\begin{lem}\label{lem::injres}
Let $M_1,M_2$ be $\ground[\pi]$-modules whose underlying $\ground$-modules are finitely
generated and projective.  Then, for suitable choices of relatively
injective resolutions $M_1 \to I_1^*$ and $M_2\to I_2^*$, $I_1^*\otimes
I_2^*$ is a relatively injective resolution of $M_1\otimes M_2$ and
\[ 0 \to M_1\otimes M_2 \to (I_1^*\otimes I_2^*)^0 \to (I_1^*\otimes
I_2^*)^1 \to\cdots \]
is contractible when viewed as a complex of $\ground$-modules.
\end{lem}

\begin{proof}
Since each $M_i$ is finitely generated and projective as \rarticle $\ground$-module, the
same is true of $\overline{M}_i := \Hom_\ground(M_i,\ground) \cong
\Hom_{\ground[\pi]}(M_i,\ground[\pi])$; additionally, $M_i \cong
\overline{\overline{M}}_i = \Hom_\ground(\overline{M}_i,\ground) \cong
\Hom_{\ground[\pi]}(\overline{M}_i,\ground[\pi])$.

Choose $\ground[\pi]$-projective finitely generated resolutions $P_{i,*} \to
\overline{M}_i$.  Since $\overline{M}_i$ is projective as \rarticle $\ground$-module,
$P_{i,*}\to M_i$ is a homotopy equivalence of complexes of $\ground$-modules.  If
we then dualize, defining $I_i^n = \overline{P}_{i,n}$, it follows that
$M_i \cong \overline{\overline{M}}_i \xrightarrow{\eta_i} I_i^*$ is also a
homotopy equivalence of complexes over $\ground$.  Since the $P_{i,*}$ are
projective $\ground[\pi]$-modules, the same is true of the $I_i^*$, and so we
have relatively injective resolutions of $M_1$ and $M_2$ by
\autoref{prop::brown}.  

Since $\eta_1$ and $\eta_2$ are homotopy equivalences of chain
complexes of $\ground$-modules, it follows that $\eta_1\otimes\eta_2 \colon
M_1\otimes M_2 \to I_1^* \otimes I_2^*$ is a homotopy equivalence (over
$\ground\otimes \ground = \ground$) as well; see, e.g.\ \cite[Chapter I]{Brown}.  This
immediately implies the final conclusion of the theorem.  In addition,
since
\[ 0\to M_1\otimes M_2 \to (I_1^*\otimes I_2^*)^0 \to
(I_1^*\otimes I_2^*)^1\to\cdots\]
is acyclic as a complex of $\ground$-modules, it must also be acyclic as a
complex of $\ground[\pi]$-modules.  Using \autoref{prop::brown} and the fact that
the tensor product of projectives is projective, we see that $M_1\otimes
M_2\to I_1^*\otimes I_2^*$ is a relatively injective resolution as claimed.
\end{proof}

\begin{proof}[Proof of \autoref{thm::multstr}]
We would like to construct the trigraded Eilenberg spectral sequence in
such a way that there is a multiplication on the $E_0$ page.
By \autoref{prop::brown}, we may use the relatively injective resolutions of 
\autoref{lem::injres} when defining the trigraded Eilenberg spectral
sequence.  Thus, for each $t$, we have a relatively injective resolution
$N^t \to I^{t,*}$, and for each $t_1,t_2$, $N^{t_1}\otimes N^{t_2} \to
I^{t_1,*} \otimes I^{t_2,*}$ is also a relatively injective resolution.

Since $N^*$ is an algebra, there are maps $N^{t_1}\otimes N^{t_2}\to
N^{t_1+t_2}$ for each $t_1$ and $t_2$.
We would then like to fill in the
dotted arrows in the diagram below.
\[ \xymatrix@R=1.5pc@C=1pc{ 
  0 \ar[r] & N_{t_1}\otimes N_{t_2} \ar[r] \ar[d] & 
  (I^{t_1,*}\otimes I^{t_2,*})^0 \ar[r] \ar@{.>}[d] & 
  (I^{t_1,*}\otimes I^{t_2,*})^1 \ar[r] \ar@{.>}[d] & \cdots \\
  0 \ar[r] & N^{t_1+t_2} \ar[r] & I^{t_1+t_2,0} \ar[r] & I^{t_1+t_2,1} \ar[r] & \cdots}
\]
By \autoref{lem::injres}, the top row is an admissible exact sequence.  The
bottom row is a relatively injective resolution of $N^{t_1+t_2}$.  We may
thus use the universal property of relatively injective modules to
successively fill in the dotted arrows, as usual.

These dotted arrows give us maps $I^{t_1,u_1}\otimes I^{t_2,u_2} \to
I^{t_1+t_2,u_1+u_2}$ for each $t_1,t_2,u_1,u_2$, and hence
$\varphi\colon I^{*,*}\otimes I^{*,*} \to I^{*,*}$.  This $\varphi$,
together with the Alexander-Whitney map $C_*(\tilde{X})\to
C_*(\tilde{X})\otimes C_*(\tilde{X})$, then gives us a multiplication on
\[\Hom_{\ground[\pi]}(C_*(\tilde{X}),I^{*,*}),\]
the $E_0$ page of our spectral sequence, and hence makes the trigraded
Eilenberg spectral sequence into a spectral sequence of algebras.
\end{proof}

\subsection{The equivariant spectral sequence}
As in the previous section, we may view the equivariant Eilenberg spectral
sequence as a trigraded spectral sequence.  Unfortunately, the tricks of
\autoref{subsec::noneqmult} do not immediately generalize to the equivariant
situation.  As a result, the following remains a conjecture.

\begin{conj}\label{conj::eqmult}
In some cases of interest, the trigraded equivariant Eilenberg spectral
sequence
\[ E_2^{t,u,v} = \Ext_\Pi^{u,v}(\sH_*(\tilde{X}),\sN^t) \Longrightarrow
\sH^{u+v}(X;\sN^t) \]
is a spectral sequence of algebras.  The multiplication on the $E_2$ page
comes from the diagonal map on $\sH_*(\tilde{X})$ and the multiplication of
$\sN^*$.
\end{conj}

\section{Sample nonequivariant computations}\label{sec::ex}

We are now in a position to use the multiplicative Eilenberg spectral
sequence to do some computations with the nonequivariant Serre spectral
sequence.

%

\begin{exmp}
Again consider the fibration \[B\text{det}\colon BO(2)\to B\cyctwo,\]
but now take $\ground=\bZ$, so that $\ground[\cyctwo]$ is not semisimple.
Then the Serre spectral sequence has
\[ E_2^{s,t} = \neqh^s(\bR P^\infty;\sH^t(BSO(2);\bZ)) \Longrightarrow \neqh^{s+t}(BO(2);\bZ)
\]
We know that the underlying $\bZ$-algebra (i.e.\ ring) of $\sH^t(BSO(2);\bZ)$ 
is the integral polynomial ring $\bZ[x]$; since $\bZ$ is a projective
abelian group, \autoref{thm::multstr} applies.  The action of $\cyctwo$ on
$\bZ[x]$ is again given by $x\mapsto -x$.
The trigraded Eilenberg spectral sequence (of algebras) then has
\[ E_2^{t,u,v} = \Ext_{\bZ[\cyctwo]}^{u}(\neqh_v(S^\infty),\bZ[x]^t).\]
Since $S^\infty\simeq \pt$, each term with $v\ne 0$ vanishes, and so
the spectral sequence collapses with no extension problems.  It follows
that, in the Serre spectral sequence, 
\[ E_2^{s,t} = \Ext_{\bZ[\cyctwo]}^{s}(\bZ,\bZ[x]^t).\]
Since $\cyctwo$ acts on $x$ by $x\mapsto -x$, $\bZ[x]^t$ is either $\bZ$
with the trivial $\cyctwo$ action or $\bZm$ with the action by $-1$.  We
can then find relatively injective resolutions of these groups, coming from
projective resolutions of their duals, and use these to calculate the
desired Ext groups together with their multiplication, as described by
\autoref{thm::multstr}.  

In fact, in this case, the description of the multiplicative structure is
made simpler by the fact that $\bZ[x]^{t_1}\otimes \bZ[x]^{t_2} \to
\bZ[x]^{t_1+t_2}$ is an isomorphism for every $t_1$ and $t_2$.  Thus, if
$I^{t_1,*}$ and $I^{t_2,*}$ are the relatively injective resolutions of
$\bZ[x]^{t_1}$ and $\bZ[x]^{t_2}$ given by \autoref{thm::multstr},
$I^{t_1,*}\otimes I^{t_2,*}$ is a resolution of $\bZ[x]^{t_1+t_2}$, and so
it may be used to calculate $\Ext_{\bZ[\cyctwo]}^s(\bZ,\bZ[x]^{t_1+t_2})$.
Specifically, if $\alpha_i \in \Hom_{\bZ[\cyctwo]}(\bZ,I^{t_i,s_i})$ for
$i=1,2$ represent cohomology classes in
$\Ext_{\bZ[\cyctwo]}^{s_i}(\bZ,\bZ[x]^{t_i})$, then their product in
$\Ext_{\bZ[\cyctwo]}^{s_1+s_2}(\bZ,\bZ[x]^{t_1+t_2})$
is represented by ${\alpha_1\otimes \alpha_2 \in
\Hom_{\bZ[\cyctwo]}(\bZ,I^{t_1,s_1}\otimes I^{t_2,s_2})}$.

From here, the example can be completed by explicitly choosing relatively
injective resolutions $I^*$ of $\bZ$ and $J^*$ of $\bZm$, as we will
outline.  Note that, for both $\bZ$ and $\bZm$, the dual is isomorphic to
the original module.  For the remainder of this example, write $\cyctwo$
multiplicatively with generator $\sigma$.  As is well known, the following
is a resolution for $\bZ$ over $\bZ[\cyctwo]$.
\[ \cdots \to
\bZ[\cyctwo]\xrightarrow{1-\sigma}
\bZ[\cyctwo]\xrightarrow{1+\sigma}
\bZ[\cyctwo]\xrightarrow{1-\sigma}
\bZ[\cyctwo]\xrightarrow{\ \epsilon\ } \bZ \to 0 \]
Dualizing, we see that we may take $I^*$ to be the relatively
injective resolution
\[ 0 \to \bZ[\cyctwo] \xrightarrow{1-\sigma} \bZ[\cyctwo]
\xrightarrow{1+\sigma} \bZ[\cyctwo] \xrightarrow{1-\sigma} \bZ[\cyctwo]
\to\cdots \]
A similar argument shows that we may take $J^*$ to be the relatively
injective resolution
\[ 0 \to \bZ[\cyctwo] \xrightarrow{1+\sigma} \bZ[\cyctwo]
\xrightarrow{1-\sigma} \bZ[\cyctwo] \xrightarrow{1+\sigma} \bZ[\cyctwo]
\to\cdots \]
\newcommand{\zzzz}{\bZ[\cyctwo\!\times\!\cyctwo]}
Note that $\bZ[\cyctwo]\otimes \bZ[\cyctwo]\cong \zzzz$.  The tensor products
$I^*\otimes I^*$, $I^*\otimes J^*$, $J^*\otimes I^*$, and $J^*\otimes J^*$
all have the form
\[ 0 \to \zzzz \to \zzzz^{\oplus 2} \to \zzzz^{\oplus 3} \to \zzzz^{\oplus
4} \to \cdots \]
where the arrows depend on which of the four tensor products we are
considering.  It is an instructive exercise at this point to apply the
functor $\Hom_{\bZ[\cyctwo]}(\bZ,-)$, take homology, and check that we get
the correct Ext groups.  We will omit the details, but \autoref{fig::extbz}
gives a picture of $\Ext_{\bZ[\cyctwo]}^s(\bZ,\bZ[x]^t)$ for varying values
of $s$ and $t$.
\begin{figure}
\[ \xymatrix@R=.5pc@C=.5pc{
t & & & & & & & & & & & & \\
 & & & & & & & & & & & & \\
\bZ & . & \cyctwo & . & \cyctwo & . & \cyctwo & . & \cyctwo & . & \cyctwo & & \\
. & . & . & . & . & . & . & . & . & . & . & & \\
. & \cyctwo & . & \cyctwo & . & \cyctwo & . & \cyctwo & . & \cyctwo & . & & \\
. & . & . & . & . & . & . & . & . & . & . & & \\
\bZ & . & \cyctwo & . & \cyctwo & . & \cyctwo & . & \cyctwo & . & \cyctwo & & \\
. & . & . & . & . & . & . & . & . & . & . & & \\
. & \cyctwo & . & \cyctwo & . & \cyctwo & . & \cyctwo & . & \cyctwo & . & & \\
. & . & . & . & . & . & . & . & . & . & . & & \\
\bZ \ar@{.>}[uuuuuuuuuu] \ar@{.>}[rrrrrrrrrrrr] & . & \cyctwo & . & \cyctwo &
. & \cyctwo & . & \cyctwo & . & \cyctwo  & & s \\
}\]
\caption[\lofspace{}${\Ext^s_{\bZ[\cyctwo]}(\bZ,\bZ[x]^t)}$]{$\Ext_{\bZ[\cyctwo]}^s(\bZ,\bZ[x]^t)$.  A dot represents the zero group.}
\label{fig::extbz}
\end{figure}

We will now describe the multiplicative structure.
It is a routine calculation to check, using the above resolutions $I^*$,
$J^*$, and their tensor products, that the product of any two generators of
nonzero groups appearing in the picture above is a generator of the group
in the appropriate degree.  It follows that we can describe the
multiplicative structure on $\Ext_{\bZ[\cyctwo]}^*(\bZ,\bZ[x]^*)$ as
follows.  Let $p_1$ be a generator of the copy of $\bZ$ at $(s,t) = (0,4)$,
$\alpha$ the generator of the $\cyctwo$ at $(2,0)$, and $\beta$ the
generator of the $\cyctwo$ at $(1,2)$.  Then our algebra has the
presentation
\[ \Ext_{\bZ[\cyctwo]}^*(\bZ,\bZ[x]^*) \cong \bZ[p_1,\alpha,\beta]/{\left(
2\alpha=0,2\beta=0,\beta^2=p_1\alpha\right)}.\]

Having computed the multiplicative structure of the $E_2$ page, we are
finally ready to consider the Serre spectral sequence for
$BO(2)\xrightarrow{B\text{det}} B\cyctwo$.  For dimensional reasons, all
differentials except $d_3$ must be 0.  Note that $d_3 \beta$ lands in a
copy of $\cyctwo$, so $d_3 \beta^2 = 0$ by the Leibniz formula.  Since
$d_3 \alpha = 0$, we have $0 = d_3 \beta^2 = d_3 (p_1 \alpha) = (d_3 p_1)
\alpha$; hence $d_3 p_1 = 0$.  We are thus left with only $d_3 \beta$ to
consider.

\begin{lem}
$\neqh^3(BO(2);\bZ) \ne 0$.
\end{lem}

\begin{proof}
By the Universal Coefficients Theorem, we have (non-canonical)
isomorphisms
\begin{align*}
\neqh^3(BO(2);\bZ) & \cong \text{torsion}(\neqh_2(BO(2);\bZ))\oplus \text{free}(\neqh_3(BO(2);\bZ)) \\
\neqh^2(BO(2);\bZ) & \cong \text{torsion}(\neqh_1(BO(2);\bZ))\oplus \text{free}(\neqh_2(BO(2);\bZ)) \\
\neqh_2(BO(2);\ftwo) & \cong \neqh_2(BO(2);\bZ)\otimes\ftwo \oplus \Tor_1^\bZ(\neqh_1(BO(2);\bZ),\ftwo)
\end{align*}
Since $\pi_1 (BO(2)) \cong \cyctwo$, $\neqh_1(BO(2);\bZ)\cong\cyctwo$ as well.
We can see from the $E_2$ page of the Serre spectral sequence above that
$\neqh^2(BO(2);\bZ) \cong \cyctwo$, so we can conclude from the middle
isomorphism that $\neqh_2(BO(2);\bZ)$ consists entirely of torsion.

The mod 2 homology $\neqh_*(BO(2);\ftwo)$ is well known;
$\neqh_2(BO(2);\ftwo)$ is an $\ftwo$-vector space of dimension two.  Since
$\neqh_1(BO(2);\bZ) = \cyctwo$, we see that \[\neqh_2(BO(2);\bZ)\otimes\ftwo\] is
nonzero.  Since $\neqh_2(BO(2);\bZ)$ consists entirely of torsion, we conclude
from the first isomorphism that $\neqh^3(BO(2);\bZ)$ must be nonzero as well.
\end{proof}

\begin{rem}
The usual method for calculating $\neqh^*(BO(2);\bZ)$ makes use of the
Bockstein spectral sequence and of the well-known fact $\neqh^*(BO(2);\ftwo)
\cong \ftwo[w_1,w_2]$.  We could alternatively follow this argument just far
enough to see that $\neqh^3(BO(2);\bZ)\ne 0$.
\end{rem}

Since $\neqh^3(BO(2);\bZ)\ne 0$, we see from the $E_2$ page of the Serre
spectral sequence above that $d_3 \beta$ must be zero.  Hence the spectral
sequence collapses at $E_2 = E_\infty$, and so we need only resolve the
extension problems.

Abbreviate $\neqh^* := \neqh^*(BO(2);\bZ)$ for the moment.  Since the Serre
spectral sequence in question is a spectral sequence of algebras converging
to $\neqh^*$, there is a product-respecting filtration $F$ on $\neqh^*$ such that
\[ (F^s \neqh^{t+s})/(F^{s+1} \neqh^{t+s}) \cong
\Ext_{\bZ[\cyctwo]}^s(\bZ,\bZ[x]^t)\cong
\bZ[p_1,\alpha,\beta]/{\left(
2\alpha=0,2\beta=0,\beta^2=p_1\alpha\right)}.\]
In particular, we see that $\neqh^2 \cong \cyctwo$, generated by $\alpha$; and
$\neqh^3 \cong \cyctwo$, generated by $\beta$.  For $\neqh^4$, we have
$\neqh^4 \supset F^1 \neqh^4 = F^2 \neqh^4 = F^3 \neqh^4 = F^4 \neqh^4 = \cyctwo \supset
0$, satisfying the short exact sequence
\[ 0 \to \cyctwo \to \neqh^4 \to \bZ \to 0.\]
Since $\bZ$ is projective, the sequence splits, and $\neqh^4\cong \bZ \oplus
\cyctwo$.  The $\cyctwo$ is generated by $\alpha^2$, and the $\bZ$ by some
element which we will call $p_1$.  A similar argument shows that $\neqh^6 \cong
\cyctwo \oplus \cyctwo$, and either $\alpha p_1 = \beta^2 \text{ or }
\alpha p_1 = \beta^2 + \alpha^3.$ By replacing $p_1$ by $p_1+\alpha^2$, if
necessary, we may assume that $\alpha p_1 = \beta^2$.  It then follows
that, as an algebra,
\[ \neqh^* = \neqh^*(BO(2);\bZ) \cong \bZ[p_1,\alpha,\beta]/{\left(
2\alpha=0,2\beta=0,\beta^2=p_1\alpha\right)}.\]
\end{exmp}


\bibliographystyle{alpha}
\bibliography{meganshulman-thesis}

\end{document}